\pdfoutput=1
\RequirePackage{ifpdf}
\ifpdf 
\documentclass[pdftex]{sigma}
\else
\documentclass{sigma}
\fi

\usepackage{amsfonts,mathrsfs}
\usepackage{amscd, bm}
\usepackage{tikz}
\usetikzlibrary{arrows,calc,matrix,topaths,positioning,scopes,shapes,decorations,decorations.markings}
\usepackage{dsfont}
\usepackage{tikz-cd}
\usepackage{adjustbox}
\usepackage{enumitem}

\newcommand{\quotient}[2]{{\raisebox{.2em}{$#1$}\left/\raisebox{-.2em}{$#2$}\right.}}

\newcommand{\Hom}{\operatorname{Hom}}
\newcommand{\tr}{\operatorname{tr}}
\newcommand{\Tr}{\operatorname{Tr}}
\newcommand{\SL}{\operatorname{SL}}
\newcommand{\id}{{\rm id}}
\newcommand{\Span}{\operatorname{Span}}
\newcommand{\End}{\operatorname{End}}
\newcommand{\GL}{\operatorname{GL}}
\newcommand{\PGL}{\operatorname{PGL}}

\newcommand{\Mod}{\operatorname{Mod}}

\newcommand{\St}{\operatorname{St}}
\newcommand{\Mat}{\operatorname{Mat}}

\newcommand{\heightexch}[3]{
	\begin{tikzpicture}[baseline=-0.4ex,scale=0.5, >=stealth]
	\draw [fill=gray!60,gray!45] (-.7,-.75) rectangle (.4,.75) ;
	\draw[#1] (0.4,-0.75) to (.4,.75);
	\draw[line width=1.2] (0.4,-0.3) to (-.7,-.3);
	\draw[line width=1.2] (0.4,0.3) to (-.7,.3);
	\draw (0.65,0.3) node {\scriptsize{$#2$}};
	\draw (0.65,-0.3) node {\scriptsize{$#3$}};
	\end{tikzpicture}
}
\newcommand{\heightcurve}{
\begin{tikzpicture}[baseline=-0.4ex,scale=0.5]
\draw [fill=gray!20,gray!45] (-.7,-.75) rectangle (.4,.75) ;
\draw[-] (0.4,-0.75) to (.4,.75);
\draw[line width=1.2] (-.7,-0.3) to (-.4,-.3);
\draw[line width=1.2] (-.7,0.3) to (-.4,.3);
\draw[line width=1.15] (-.4,0) ++(-90:.3) arc (-90:90:.3);
\end{tikzpicture}
}

\numberwithin{equation}{section}

\newtheorem{Theorem}{Theorem}[section]
\newtheorem*{Theorem*}{Theorem}
\newtheorem{Corollary}[Theorem]{Corollary}
\newtheorem{Lemma}[Theorem]{Lemma}
\newtheorem{Proposition}[Theorem]{Proposition}
 { \theoremstyle{definition}
\newtheorem{Definition}[Theorem]{Definition}

\newtheorem{Example}[Theorem]{Example}
\newtheorem{Remark}[Theorem]{Remark}
\newtheorem{Notations}[Theorem]{Notations} }

\begin{document}
\allowdisplaybreaks

\newcommand{\arXivNumber}{2202.07649}

\renewcommand{\PaperNumber}{064}

\FirstPageHeading

\ShortArticleName{Mapping Class Group Representations Derived from Stated Skein Algebras}

\ArticleName{Mapping Class Group Representations\\ Derived from Stated Skein Algebras}

\Author{Julien KORINMAN}

\AuthorNameForHeading{J.~Korinman}

\Address{Department of Mathematics, Faculty of Science and Engineering, Waseda University,\\
3-4-1 Ohkubo, Shinjuku-ku, Tokyo, 169-8555, Japan}
\Email{\href{julien.korinman@gmail.com}{julien.korinman@gmail.com}}
\URLaddress{\url{https://sites.google.com/site/homepagejulienkorinman/}}

\ArticleDates{Received March 09, 2022, in final form August 22, 2022; Published online August 26, 2022}

\Abstract{We construct finite-dimensional projective representations of the mapping class groups of compact connected oriented surfaces having one boundary component using stated skein algebras.}

\Keywords{mapping class groups; stated skein algebras; quantum moduli spaces; quantum Teichm\"uller spaces}

\Classification{57R56; 57N10; 57M25}

\section{Introduction}

Let $\Sigma$ be a compact oriented surface and denote by $\Mod(\Sigma)$ its mapping class group. The goal of this paper is to define and study finite-dimensional projective representations $\rho\colon \Mod(\Sigma)\to \PGL(V)$. The techniques to define them are quite standard and make use of the skein algebras and their recent generalizations: the (reduced) stated skein algebras. Let $\mathcal{A}\subset \partial \Sigma$ be a (possibly empty) finite set of embedded open intervals. The pair $\mathbf{\Sigma}=(\Sigma, \mathcal{A})$ will be referred to as a~\textit{marked surface}. Given $A^{1/2} \in \mathbb{C}^*$, the \textit{stated skein algebra} $\mathcal{S}_A(\mathbf{\Sigma})$ is a~generalization of the Kauffman-bracket skein algebra introduced by Bonahon--Wong~\cite{BonahonWongqTrace} and L\^e~\cite{LeStatedSkein} which admits an interesting quotient $\overline{\mathcal{S}}_A(\mathbf{\Sigma})$ named \textit{the reduced stated skein algebra} in~\cite{CostantinoLe19}. When $\mathcal{A}=\varnothing$, then $\mathcal{S}_A(\mathbf{\Sigma})=\overline{\mathcal{S}}_A(\mathbf{\Sigma})$ is the usual Kauffman-bracket skein algebra. In general, the stated skein algebra and its reduced version admit non trivial finite-dimensional representations if and only if the parameter $A$ is a root of unity. In all the paper, we will assume that $A\in \mathbb{C}^{*}$ is a root of unity of odd order~$N$. Let $\mathcal{S}$ be either $\mathcal{S}_A(\mathbf{\Sigma})$ or $\overline{\mathcal{S}}_A(\mathbf{\Sigma})$ and consider a finite-dimensional representation
\[ r\colon \ \mathcal{S} \to \End(V).\]
The mapping class group $\Mod(\Sigma)$ admits a natural right action on $\mathcal{S}$ by sending a stated diagram $(D,s)$ to $\big(\phi^{-1}(D), s\circ \phi\big)$ for $\phi \in \Mod(\Sigma)$ (see Section~\ref{sec_mcg} for precise definitions). As a consequence, we can twist the representation $r$ to a new representation $\phi \cdot r\colon \mathcal{S} \to \End(V)$ defined by the formula $(\phi \cdot r) (X):= r ( X^{\phi})$. For $G\subset \Mod(\Sigma)$ a subgroup (which will be either $\Mod(\Sigma)$ or its Torelli subgroup), the representation $r$ will be called $G$-\textit{stable} if $r$ and $\phi \cdot r$ are canonically isomorphic for all $g\in G$. By definition, this means that for all $g\in G$, there exists a linear operator $\rho(\phi) \in \GL(V)$ unique up to multiplication by a scalar (this is the meaning of being canonically isomorphic) such that the following Egorov identity holds:
\[ (\phi \cdot r) (X) = \rho(\phi) r(X) \rho(\phi)^{-1} \qquad \text{for all } X\in \mathcal{S}.\]
The assignement $\phi \mapsto \rho(\phi)$ then defines a projective representation $\rho\colon G \to \PGL(V)$. Therefore a $G$-stable representation of $\mathcal{S}$ induces a projective representation of $G$.
 For instance the Witten--Reshetikhin--Turaev representations $r^{\rm WRT}\colon \mathcal{S}_A(\Sigma, \varnothing) \to \End(V(\Sigma))$ arising in modular TQFTs \cite{RT, Tu} are $\Mod(\Sigma)$-stable and induce well-studied representations $\rho^{\rm WRT}\colon \Mod(\Sigma) \to \PGL(V(\Sigma))$ (see, e.g.,~\cite{Koju2} and references therein for details on these representations). A second family of examples arises from the skein representations in non semi-simple TQFTs~\cite{BCGPTQFT} which induce representations of the mapping class group and its Torelli subgroup (see~\cite{DeRenziGeerPatureaRunkel_MCG}).

 In general, finding $G$-stable representations is a difficult problem. Let $Z$ be the center of $\mathcal{S}$ and $\widehat{\mathcal{X}}:= \operatorname{Specm}(Z)$. An irreducible representation $r\colon \mathcal{S}\to \End(V)$ sends central elements
to scalar operators so it induces a character in $\widehat{\mathcal{X}}$ and we get a character map $\chi\colon \mathrm{Irrep}(\mathcal{S}) \to \widehat{\mathcal{X}}$. In \cite{FrohmanKaniaLe_UnicityRep, KojuAzumayaSkein}, it is proved that $\chi$ is surjective and that there exists a Zariski open subset $\mathcal{AL} \subset \widehat{\mathcal{X}}$, named \textit{Azumaya locus}, such that the restricted map $\chi\colon \chi^{-1}(\mathcal{AL}) \to \mathcal{AL}$ is a bijection. The right action of the mapping class group on~$\mathcal{S}$ restricts to a right action on the center $Z$ so it induces a left action of $\Mod(\Sigma)$ on $\mathcal{X}$. Now if $x\in \mathcal{AL}$ is $G$-stable and $r_x\colon \mathcal{S}\to \End(V)$ is an irreducible representation (unique up to unique isomorphism) with $\chi(r_x)=x$, then $r_x$ is clearly $G$-stable and we get a projective representation $\rho_x\colon \Mod(\Sigma)\to \PGL(\Sigma)$. The affine variety $\widehat{\mathcal{X}}$ admits a well-understood geometric interpretation as a finite cover of the $\SL_2$ relative representation variety of $\mathbf{\Sigma}$ and it is easy to find points which are stable under the action of the mapping class group or the Torelli group. However computing the Azumaya locus is a quite difficult problem (see \cite{GanevJordanSafranov_FrobeniusMorphism, KojuSurvey} for recent developments) and there is no reason to believe, in general, that it contains such fixed points. For instance for the Witten--Reshetikhin--Turaev representations, the induced characters do not belong to the Azumaya locus (see~\cite{BonahonWong4}) and it is a highly non trivial fact (which follows from TQFTs properties) that they are $\Mod(\Sigma)$-stable.

The main idea of this paper is to replace the standard skein algebras by particular (reduced) stated skein algebras. Let $\Sigma_{g,1}$ be a genus $g$ surface with $1$ boundary component, $a_{\partial} \subset \partial \Sigma_{g,1}$ a~boundary arc and consider the marked surface $\mathbf{\Sigma}_g^*:=(\Sigma_{g,1}, \{a_{\partial}\})$. In \cite{GanevJordanSafranov_FrobeniusMorphism}, the authors computed explicitly the Azumaya locus of $\mathcal{S}_A(\mathbf{\Sigma}_g^*)$. In the reduced case, we will prove

\begin{Theorem}\label{theorem1}
When $A$ is a root of unity of odd order, the reduced stated skein algebra $\overline{\mathcal{S}}_A(\mathbf{\Sigma}_g^*)$ is Azumaya.
\end{Theorem}

Note that the non-reduced stated skein algebra $\mathcal{S}_A(\mathbf{\Sigma}_g^*)$ appeared in literature in various contexts.
\begin{enumerate}\itemsep=0pt
\item They were first defined independently by Buffenoir--Roche and Alekseev--Grosse--Scho\-me\-rus under the name \textit{quantum moduli spaces} in \cite{AlekseevGrosseSchomerus_LatticeCS1,AlekseevGrosseSchomerus_LatticeCS2, AlekseevSchomerus_RepCS, BuffenoirRoche, BuffenoirRoche2} where they appeared as deformation quantization of the Fock--Rosly \cite{FockRosly} and Alekseev--Kosman--Malkin--Meinrenken \cite{AlekseevKosmannMeinrenken,AlekseevMalkin_PoissonLie, AlekseevMalkin_PoissonCharVar} moduli spaces (see also \cite{KojuTriangularCharVar}), that we will denote by $\mathcal{X}_{\SL_2}(\mathbf{\Sigma}_g^*)$. As an affine variety, $\mathcal{X}_{\SL_2}(\mathbf{\Sigma}_g^*)= \Hom(\pi_1(\Sigma_{g,1}), \SL_2)$ is the representation variety of $\Sigma_{g,1}$ (see Section \ref{sec_geometric} for details).
\item They were rediscovered independently by Habiro under the name \textit{quantum representation variety} in \cite{Habiro_QCharVar}.
\item They appear as a particular case of stated skein algebras defined by Bonahon--Wong \cite{BonahonWongqTrace} and L\^e \cite{LeStatedSkein}.
\item They eventually appeared under the name \textit{internal skein algebras} in the work of Ben~Zvi--Brochier--Jordan \cite{BenzviBrochierJordan_FactAlg1, BenzviBrochierJordan_FactAlg2} and further studied by Cooke--Ganev--Gunningham--Jordan--Saf\-ro\-nov in \cite{Cooke_FactorisationHomSkein, GanevJordanSafranov_FrobeniusMorphism, GunninghamJordanSafranov_FinitenessConjecture}.
\end{enumerate}
The equivalence between quantum moduli spaces and stated skein algebras was proved in \cite{BullockFrohmanKania_LGFT,Faitg_LGFT_SSkein, KojuPresentationSSkein}. The equivalence between internal skein algebras and quantum moduli spaces is proved in~\cite{BenzviBrochierJordan_FactAlg1}. The equivalence between internal skein algebras and stated skein algebras is proved in~\cite{Haioun_Sskein_FactAlg}. The equivalence between quantum representation varieties and stated skein algebras will appear in the forthcoming paper~\cite{KojuMurakami_QCharVar} (and will not be used in the present paper).
Note that among these four equivalent definitions of $\mathcal{S}_A(\mathbf{\Sigma}_g^*)$, only the stated skein approach makes appear the reduced stated skein algebra $\overline{\mathcal{S}}_A(\mathbf{\Sigma}_g^*)$ naturally: it is the quotient of the stated skein algebra by the kernel of the quantum trace (bad arcs ideal).

In order to describe the representations, let us describe $\widehat{\mathcal{X}}$. Like before, $\mathcal{S}=\mathcal{S}_A$ denotes either a stated skein algebra or a reduced stated skein algebra at a root unity~$A$ of odd order. Let $\mathcal{S}_{+1}$ be the algebra $\mathcal{S}$ taken with parameter $A^{1/2}=+1$. The algebra $\mathcal{S}_{+1}$ is commutative and there exists an embedding \cite{BonahonWong1, KojuQuesneyClassicalShadows}
\[ \operatorname{Ch}_A \colon \ \mathcal{S}_{+1} \hookrightarrow \mathcal{Z}(\mathcal{S}_A), \]
into the center of $\mathcal{S}_A$, named \textit{Chebyshev--Frobenius} morphism. Write $\widehat{\mathcal{X}}:= \operatorname{Specm}\left( \mathcal{Z}(\mathcal{S}_A) \right)$ and $\mathcal{X} := \operatorname{Specm}(\mathcal{S}_{+1})$. Then $\operatorname{Ch}_A$ induces a dominant map $\pi\colon \widehat{\mathcal{X}} \to \mathcal{X}$ which is a finite branched covering and is equivariant for the mapping class group action. As described in Section~\ref{sec_mcg}, a~point $x\in \widehat{\mathcal{X}}$ is $G$-invariant if and only if its image $\pi(x)$ is $G$-invariant. If $\mathcal{S}=\mathcal{S}_A(\mathbf{\Sigma}_g^*)$, then $\widehat{\mathcal{X}}=\mathcal{X}$ and by \cite[Theorem~1.3]{KojuQuesneyClassicalShadows}, one has a Poisson $\Mod(\Sigma_{g,1})$-equivariant isomorphism
\[ \widehat{\mathcal{X}} = \mathcal{X} \cong \mathcal{X}_{\SL_2}(\mathbf{\Sigma}_{g}^*) = \Hom (\pi_1(\Sigma_{g,1}), \SL_2). \]
The proof of Theorem \ref{theorem1} essentially follows from Brown--Gordon theory of (equivariant) Poisson orders and from the study in \cite{AlekseevKosmannMeinrenken, AlekseevMalkin_PoissonLie, AlekseevMalkin_PoissonCharVar, GanevJordanSafranov_FrobeniusMorphism} of the symplectic leaves of $ \mathcal{X}_{\SL_2}(\mathbf{\Sigma}_g^*)$. Let $\gamma_{\partial}\in \pi_1(\Sigma_{g,1})$ be the class of a small curve encircling once the boundary component of $\Sigma_{g,1}$. The \textit{moment map} $\mu\colon \Hom(\pi_1(\Sigma_{g,1}), \SL_2) \to \SL_2$ is the map sending a representation $\rho$ to $\mu(\rho):= \rho(\gamma_{\partial})$. Note that the subsets $\mu^{-1}(g)$ are preserved by the mapping class group action.

The main construction of the present paper is summarized as follows. Let $A$ a root of unity of odd order~$N$.
Let $G\subset \Mod(\Sigma_{g,1})$ be a subgroup of the mapping class group and $\mathcal{O}\subset \Hom(\pi_1(\Sigma_{g,1}), \SL_2)$ a finite $G$-orbit included in a leaf $\mu^{-1}(g)$ with $g=\left(\begin{smallmatrix} a & b \\c & d \end{smallmatrix}\right) \in \SL_2$. The orbit is said in the big cell if $a\neq 0$ and in the reduced cell else. To each such orbit, thanks to Theorem~\ref{theorem1}, we will associate a $G$-stable representation of either the stated skein algebra $\mathcal{S}_A(\mathbf{\Sigma}_g^*)$ when the orbit is in the big cell, or of the reduced stated skein algebra $\overline{\mathcal{S}}_A(\mathbf{\Sigma}_g^*)$ if the orbit is in the reduced cell.
Thanks to this $G$-stable representation, we will construct in Section~\ref{sec_representations} a finite-dimensional projective representation
\[ \pi_{\mathcal{O}}\colon \ G \to \PGL(W(\mathcal{O}))\]
such that $\dim ( W(\mathcal{O}) ) = \begin{cases}
N^{3g} |\mathcal{O}|& \text{if $\mathcal{O}$ is in the big cell}, \\
N^{3g-1} | \mathcal{O} | & \text{if $\mathcal{O}$ is reduced}.
\end{cases}$

The end of the paper is devoted to the study of the kernel of the representations $\pi$. More precisely, we will find criteria to prove that a given mapping class does not belong to the kernel of $\pi$. The simplest such criterion is the following, which should be compared to the work of Costantino--Martelli in \cite[Theorem~7.1]{CostantinoMartelli} on Turaev--Viro representations.

\begin{Theorem}\label{theorem2}
Let $\pi\colon G \to \End(W(\mathcal{O}))$ be a representation of $G\subset \Mod(\Sigma_{g,1})$ associated to a~root of unity of odd order $N$. Consider a mapping class $\phi\in G$ and a simple closed curve $\alpha \subset \Sigma_{g,1}$ such that $\alpha$ and $\beta:= \phi(\alpha)$ are not isotopic.
Suppose that there exists a~triangulation~$\Delta$ of $\mathbf{\Sigma}_g^*$ such that
\begin{enumerate}\itemsep=0pt
\item[$1)$] there exists $\rho \in \mathcal{O}$ which admits a $\Delta$-lift,
\item[$2)$] both $\alpha$ and $\beta$ intersect each edge of $\Delta$ at most $N-1$ times.
\end{enumerate}
Then $\pi(\phi)\neq \id$.
\end{Theorem}

The condition ``$\rho$ admits a $\Delta$ lift'' is a generic condition which permits the use of quantum Teichm\"uller theory to study $\pi$. It says that the representation $\rho$ is described by a $\SL_2$ version of the shear-bend coordinates (see Sections~\ref{sec_QT_reduced} and~\ref{sec_QT_nonreduced} for details).

{\bf Plan of the paper.}
The paper is organized as follows. In Section~\ref{sec_SSkein} we review the definition and main properties of stated skein algebras at odd roots of unity and their reduced version. Section~\ref{sec_geometric} is devoted to the geometric study of $\mathcal{X}_{\SL_2}(\mathbf{\Sigma}_g^*)$ following \cite{AlekseevKosmannMeinrenken, AlekseevMalkin_PoissonLie, AlekseevMalkin_PoissonCharVar, GanevJordanSafranov_FrobeniusMorphism}. In Section~\ref{sec_algebraic}, we review the notions of Azumaya loci and Poisson orders which led to the proof of Theorem~\ref{theorem1}.
Parts of the exposition of these three sections already appeared in the author's unpublished survey~\cite{KojuSurvey}.
In Section~\ref{sec_representations}, we construct the representations $ \pi\colon G \to \PGL(V(\mathcal{O}))$ and provide some tools to prove that a given mapping class does not belong to the kernel.

\section{Reduced stated skein algebras}\label{sec_SSkein}

\subsection{Definition of the reduced stated skein algebras}

\begin{Definition} A \textit{marked surface} $\mathbf{\Sigma}=(\Sigma, \mathcal{A})$ is a compact oriented surface $\Sigma$ (possibly with boundary) with a finite set $\mathcal{A}=\{a_i\}_i$ of orientation-preserving immersions $a_i\colon [0,1] \hookrightarrow \partial \Sigma$, \textit{named boundary arcs}, whose restrictions to $(0,1)$ are embeddings and whose interiors are pairwise disjoint. We say that $\mathbf{\Sigma}=(\mathbf{\Sigma}, \mathcal{A})$ is \textit{unmarked} if $\mathcal{A}=\varnothing$. A \textit{puncture} is a connected component of $\partial \Sigma \setminus \mathcal{A}$.

An \textit{embedding} $f\colon (\Sigma, \mathcal{A}) \to (\Sigma', \mathcal{A}')$ of marked surfaces is a orientation-preserving proper embedding $f\colon \Sigma \to \Sigma'$ so that for each boundary arc $a \in \mathcal{A}$ there exists $a' \in \mathcal{A}$ such that $f\circ a$ is the restriction of $a'$ to some subinterval of $[0,1]$. When several boundary arcs $a_1, \dots, a_n$ in $\mathbf{\Sigma}$ are mapped to the same boundary arc $b$ of $\mathbf{\Sigma}'$ we include in the definition of $f$ the datum of a total ordering of $\{a_1,\dots, a_n\}$.
 Marked surfaces with embeddings form a category $\mathrm{MS}$. We denote by $\mathrm{MS}^{(1)}\subset \mathrm{MS}$ the full subcategory generated by connected marked surfaces having exactly one boundary arc.
\end{Definition}

By abuse of notation, we will often denote by the same letter the embedding $a_i$ and its image $a_i((0,1)) \subset \partial \Sigma$ and both call them boundary arcs. We will also abusively identify $\mathcal{A}$ with the disjoint union $\bigsqcup_i a_i((0,1)) \subset \partial \Sigma$ of open intervals.

 \begin{Notations} Let us name some marked surfaces. Let $\Sigma_{g,n}$ be an oriented connected surface of genus $g$ with $n$ boundary components.
 \begin{enumerate}\itemsep=0pt
 \item The \textit{once-punctured monogon} $\mathbf{m}_1=(\Sigma_{0,2}, \{a\})$ is an annulus with one boundary arc in one of its boundary component.
 \item The \textit{bigon} $\mathbb{B}$ is a disc with two boundary arcs, the \textit{triangle} $\mathbb{T}$ is a disc with three boundary arcs. We also write $\mathbb{D}_1^+:= (\Sigma_{0,2}, \{a,b\})$ the annulus with one boundary arc in each boundary component.
 \item We denote by $\mathbf{\Sigma}_{g}^*=(\Sigma_{g,1}, \{a_{\partial}\} )$ the surface $\Sigma_{g,1}$ with a single boundary arc in its only boundary component.
 \end{enumerate}

 \end{Notations}

 There are three natural operations on the category of marked surfaces.
 \begin{enumerate}\itemsep=0pt
 \item The \textit{disjoint union} $\bigsqcup$ which endows $\mathrm{MS}$ with a symmetric monoidal structure. (It is actually a coproduct in $\mathrm{MS}$.)
 \item The \textit{gluing operation} described as follows.
 Let $\mathbf{\Sigma}=(\Sigma, \mathcal{A})$ be a marked surface and $a,b\in \mathcal{A}$ two boundary arcs. Set $\Sigma_{a\#b} := \quotient{\Sigma}{a(t) \sim b(1-t)}$ and $\mathcal{A}_{a\#b}:= \mathcal{A}\setminus (a\cup b)$. The marked surface $\mathbf{\Sigma}_{a\#b}=(\Sigma_{a\#b}, \mathcal{A}_{a\#b})$ is said obtained from $\mathbf{\Sigma}$ by gluing $a$ and $b$.
 \item The \textit{fusion operation} is described as follows. Consider a marked surface $\mathbf{\Sigma}$ with two boundary arcs $a$, $b$ as before and denote by $i$, $j$, $k$ the three boundary arcs (edges) of the triangle~$\mathbb{T}$. The fusion of $\mathbf{\Sigma}$ along~$a$ and~$b$ is the marked surface $\mathbf{\Sigma}_{a\circledast b} := ( \mathbf{\Sigma}\bigsqcup \mathbb{T} )_{a\# i , b\# j }$. For $\mathbf{\Sigma}_1=(\Sigma_1, \{a_1\}), \mathbf{\Sigma}_2=(\Sigma_2, \{a_2\}) \in \mathrm{MS}^{(1)}$, we denote by $\mathbf{\Sigma}_1\wedge \mathbf{\Sigma}_2$ the fusion $(\mathbf{\Sigma}_1\bigsqcup \mathbf{\Sigma}_2)_{a_1 \circledast a_2}$.
 \end{enumerate}

 \begin{Remark}\label{remark_fusion}\quad
 \begin{enumerate}\itemsep=0pt
 \item $\mathbf{m}_1$ is obtained from $\mathbb{B}$ by fusioning its two boundary arcs.
 \item $\mathbf{\Sigma}_1^*$ is obtained from $\mathbb{D}_1^+$ by fusioning its two boundary arcs.
 \item We have $\mathbf{\Sigma}_{g}^* \wedge \mathbf{\Sigma}_{g'}^* \cong \mathbf{\Sigma}_{g+g'}^*$ so $\mathbf{\Sigma}_g^*\cong (\mathbf{\Sigma}_1^*)^{\wedge n}$.
 \end{enumerate}
 These remarks permitted the authors of \cite{AlekseevKosmannMeinrenken, AlekseevMalkin_PoissonCharVar} to reduce the study of the Poisson geometry of the relative character variety of $\mathbf{\Sigma}_g^*$ to the study of the representation variety of $\mathbb{D}_1^+$, made in~\cite{AlekseevMalkin_PoissonLie}, as we shall briefly review in Section~\ref{sec_geometric}.
 \end{Remark}

A \textit{tangle} is a compact framed, properly embedded $1$-dimensional manifold $T\subset \Sigma \times (0,1)$ with $\partial T \subset \mathcal{A}\times (0,1)$ such that for every point of~$\partial T$ the framing is parallel to the $(0,1)$ factor and points to the direction of~$1$. The \textit{height} of $(v,h)\in \Sigma \times (0,1)$ is~$h$. If $a$ is a boundary arc and~$T$ a~tangle, we impose that no two points in $\partial_aT:= \partial T \cap a\times(0,1)$ have the same heights, hence the set $\partial_aT$ is totally ordered by the heights. Two tangles are isotopic if they are isotopic through the class of tangles that preserve the boundary height orders. By convention, the empty set is a~tangle only isotopic to itself. A~\textit{state} is a~map $s\colon \partial T \to \{ \pm \}$ and a \textit{stated tangle} is a~pair~$(T,s)$.

 \begin{figure}[!h]
\centerline{\includegraphics[width=6cm]{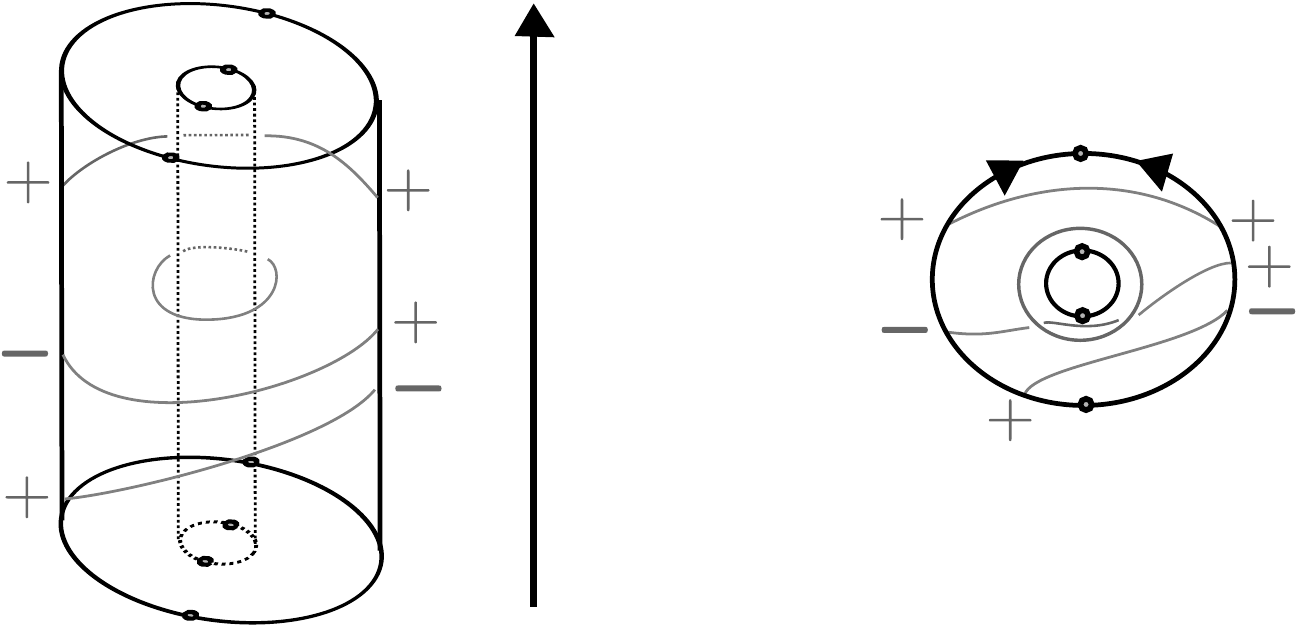}}
\caption{On the left: a stated tangle. On the right: its associated diagram. The arrows represent the height orders. }
\label{fig_statedtangle}
\end{figure}

 \begin{Definition}
 Let $k$ be a (unital associative) commutative ring and let $A^{1/2} \in k^{\times}$ be an invertible element. The \textit{stated skein algebra} $\mathcal{S}_{A}(\mathbf{\Sigma})$ is the free $k$-module generated by isotopy classes of stated tangles in $\Sigma\times (0, 1)$ modulo the following skein relations
 	\begin{equation*}
 \begin{split}&
\begin{tikzpicture}[baseline=-0.4ex,scale=0.5,>=stealth]	
\draw [fill=gray!45,gray!45] (-.6,-.6) rectangle (.6,.6) ;
\draw[line width=1.2,-] (-0.4,-0.52) -- (.4,.53);
\draw[line width=1.2,-] (0.4,-0.52) -- (0.1,-0.12);
\draw[line width=1.2,-] (-0.1,0.12) -- (-.4,.53);
\end{tikzpicture}
=A
\begin{tikzpicture}[baseline=-0.4ex,scale=0.5,>=stealth]
\draw [fill=gray!45,gray!45] (-.6,-.6) rectangle (.6,.6) ;
\draw[line width=1.2] (-0.4,-0.52) ..controls +(.3,.5).. (-.4,.53);
\draw[line width=1.2] (0.4,-0.52) ..controls +(-.3,.5).. (.4,.53);
\end{tikzpicture}
+A^{-1}
\begin{tikzpicture}[baseline=-0.4ex,scale=0.5,rotate=90]	
\draw [fill=gray!45,gray!45] (-.6,-.6) rectangle (.6,.6) ;
\draw[line width=1.2] (-0.4,-0.52) ..controls +(.3,.5).. (-.4,.53);
\draw[line width=1.2] (0.4,-0.52) ..controls +(-.3,.5).. (.4,.53);
\end{tikzpicture}
\qquad
\text{and}\qquad
\begin{tikzpicture}[baseline=-0.4ex,scale=0.5,rotate=90]
\draw [fill=gray!45,gray!45] (-.6,-.6) rectangle (.6,.6) ;
\draw[line width=1.2,black] (0,0) circle (.4) ;
\end{tikzpicture}
= -\big(A^2+A^{-2}\big)
\begin{tikzpicture}[baseline=-0.4ex,scale=0.5,rotate=90]
\draw [fill=gray!45,gray!45] (-.6,-.6) rectangle (.6,.6) ;
\end{tikzpicture}
;
\\
&
\begin{tikzpicture}[baseline=-0.4ex,scale=0.5,>=stealth]
\draw [fill=gray!45,gray!45] (-.7,-.75) rectangle (.4,.75) ;
\draw[->] (0.4,-0.75) to (.4,.75);
\draw[line width=1.2] (0.4,-0.3) to (0,-.3);
\draw[line width=1.2] (0.4,0.3) to (0,.3);
\draw[line width=1.1] (0,0) ++(90:.3) arc (90:270:.3);
\draw (0.65,0.3) node {\scriptsize{$+$}};
\draw (0.65,-0.3) node {\scriptsize{$+$}};
\end{tikzpicture}
=
\begin{tikzpicture}[baseline=-0.4ex,scale=0.5,>=stealth]
\draw [fill=gray!45,gray!45] (-.7,-.75) rectangle (.4,.75) ;
\draw[->] (0.4,-0.75) to (.4,.75);
\draw[line width=1.2] (0.4,-0.3) to (0,-.3);
\draw[line width=1.2] (0.4,0.3) to (0,.3);
\draw[line width=1.1] (0,0) ++(90:.3) arc (90:270:.3);
\draw (0.65,0.3) node {\scriptsize{$-$}};
\draw (0.65,-0.3) node {\scriptsize{$-$}};
\end{tikzpicture}
=0,
\qquad
\begin{tikzpicture}[baseline=-0.4ex,scale=0.5,>=stealth]
\draw [fill=gray!45,gray!45] (-.7,-.75) rectangle (.4,.75) ;
\draw[->] (0.4,-0.75) to (.4,.75);
\draw[line width=1.2] (0.4,-0.3) to (0,-.3);
\draw[line width=1.2] (0.4,0.3) to (0,.3);
\draw[line width=1.1] (0,0) ++(90:.3) arc (90:270:.3);
\draw (0.65,0.3) node {\scriptsize{$+$}};
\draw (0.65,-0.3) node {\scriptsize{$-$}};
\end{tikzpicture}
=A^{-1/2}
\begin{tikzpicture}[baseline=-0.4ex,scale=0.5,>=stealth]
\draw [fill=gray!45,gray!45] (-.7,-.75) rectangle (.4,.75) ;
\draw[-] (0.4,-0.75) to (.4,.75);
\end{tikzpicture}
\qquad \text{and}
\qquad
A^{1/2}
\heightexch{->}{-}{+}
- A^{5/2}
\heightexch{->}{+}{-}
=
\heightcurve.
\end{split}
\end{equation*}
The product of two classes of stated tangles $[T_1,s_1]$ and $[T_2,s_2]$ is defined by isotoping $T_1$ and $T_2$ in $\Sigma\times (1/2, 1) $ and $\Sigma\times (0, 1/2)$ respectively and then setting $[T_1,s_1]\cdot [T_2,s_2]=[T_1\cup T_2, s_1\cup s_2]$. Now consider an embedding $f\colon \mathbf{\Sigma}_1\to \mathbf{\Sigma}_2$ of marked surfaces and define a proper embedding $\widetilde{f}\colon \Sigma_1\times (0,1) \to \Sigma_2\times (0,1)$ such that: $(1)$ $\widetilde{f}(x,t)=(f(x), \varphi(x,t))$ for $\varphi$ a smooth map and $(2)$~if $a_1$, $a_2$ are two boundary arcs of $\mathbf{\Sigma}_1$ mapped to the same boundary arc $b$ of $\mathbf{\Sigma}_2$ and the ordering of $f$ is $a_1< a_2$, then for all $x_1\in a_1$, $x_2\in a_2$, $t_1, t_2 \in (0,1)$ one has $\varphi(x_1,t_1)<\varphi(x_2, t_2)$. We define $f_* \colon \mathcal{S}_A(\mathbf{\Sigma}_1)\to \mathcal{S}_A(\mathbf{\Sigma}_2)$ by $f_*([T,s]):= \big[\widetilde{f}(T), s\circ \widetilde{f}^{-1}\big]$.
The assignment $\mathbf{\Sigma}\to \mathcal{S}_A(\mathbf{\Sigma})$ defines a symmetric monoidal functor
$ \mathcal{S}_A \colon \mathrm{MS} \to \mathrm{Alg}_k$.
\end{Definition}

\begin{figure}[!h]
\centerline{\includegraphics[width=8cm]{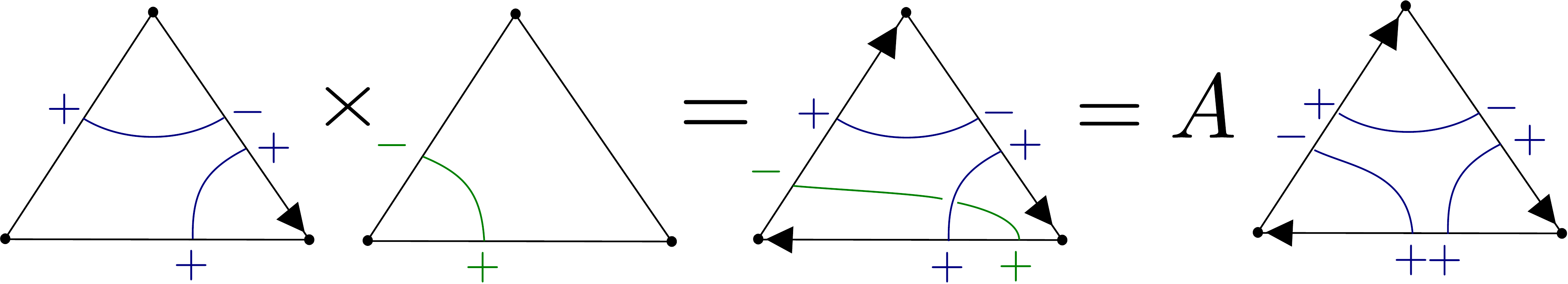} }
\caption{An illustration of the product in stated skein algebras.}\label{fig_product}
\end{figure}

\subsection{Splitting morphism and comodule structure}

Let $a$, $b$ be two distinct boundary arcs of $\mathbf{\Sigma}$, denote by $\pi\colon \Sigma\rightarrow \Sigma_{a\#b}$ the projection and $c:=\pi(a)=\pi(b)$. Let $(T_0, s_0)$ be a stated framed tangle of $\Sigma_{a\#b}\times (0,1)$ transversed to $c\times (0,1)$ and such that the heights of the points of $T_0 \cap c\times (0,1)$ are pairwise distinct and the framing of the points of $T_0 \cap c\times (0,1)$ is vertical towards $1$. Let $T\subset \Sigma \times (0,1)$ be the framed tangle obtained by cutting~$T_0$ along~$c$.
Any two states $s_a\colon \partial_a T \rightarrow \{-,+\}$ and $s_b\colon \partial_b T \rightarrow \{-,+\}$ give rise to a state $(s_a, s_0, s_b)$ on~$T$.
Both the sets $\partial_a T$ and $\partial_b T$ are in canonical bijection with the set $T_0\cap c$ by the map $\pi$. Hence the two sets of states~$s_a$ and~$s_b$ are both in canonical bijection with the set ${\rm St}(c):=\{ s\colon c\cap T_0 \rightarrow \{-,+\} \}$.

\begin{Definition}\label{def_gluing_map}
Let $\theta_{a\#b}\colon \mathcal{S}_{A}(\mathbf{\Sigma}_{a\#b}) \rightarrow \mathcal{S}_{A}(\mathbf{\Sigma})$ be the linear map given, for any $(T_0, s_0)$ as above, by
\[ \theta_{a\#b} ( [T_0,s_0] ) := \sum_{s \in {\rm St}(c)} [T, (s, s_0 , s) ].\]
\end{Definition}

\begin{figure}[!h]
\centerline{\includegraphics[width=8cm]{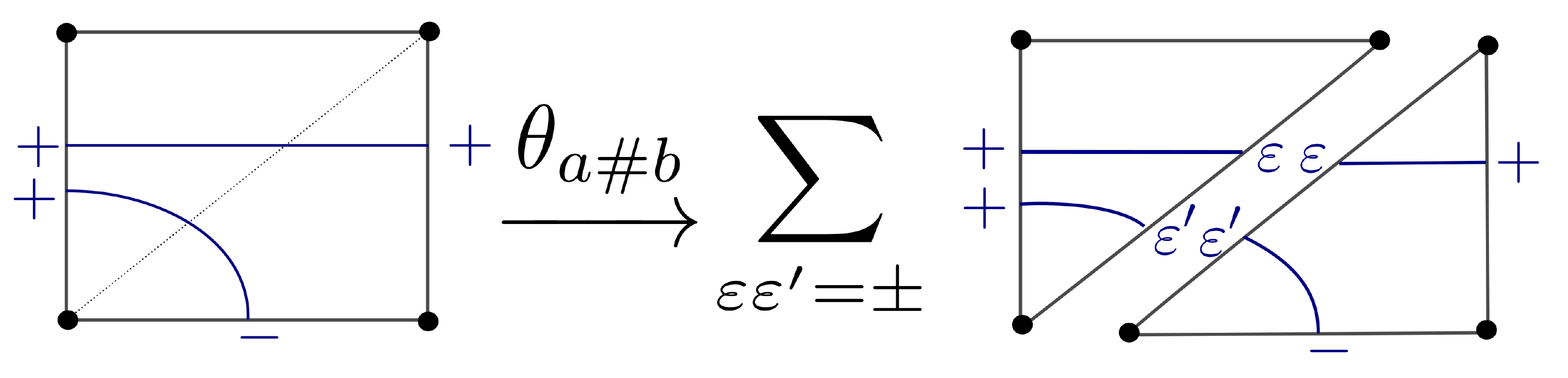} }
\caption{An illustration of the gluing map $\theta_{a\#b}$.}\label{fig_gluingmap}
\end{figure}

\begin{Theorem}[{\cite[Theorem~3.1]{LeStatedSkein}}] The map $\theta_{a\#b}$ is an injective morphism of algebras. \end{Theorem}

Recall that the bigon $\mathbb{B}$ is a disc with two boundary arcs, says $a_L$, $a_R$. While gluing two bigons $\mathbb{B}$ and $\mathbb{B}'$ together along $a_R \# a_L'$, we get another bigon. Thus we have a coproduct
\[ \Delta:= \theta_{a_R\# a_L'} \colon \ \mathcal{S}_A(\mathbb{B}) \to \mathcal{S}_A(\mathbb{B})^{\otimes 2}, \]
which endows $\mathcal{S}_A(\mathbb{B})$ with a structure of Hopf algebra. Further define a co-R-matrix $\mathbf{r}\colon \mathcal{S}_A(\mathbb{B})^{\otimes 2} \allowbreak \to k$ by
\[ \mathbf{r} \left( \adjustbox{valign=c}{\includegraphics[width=1.7cm]{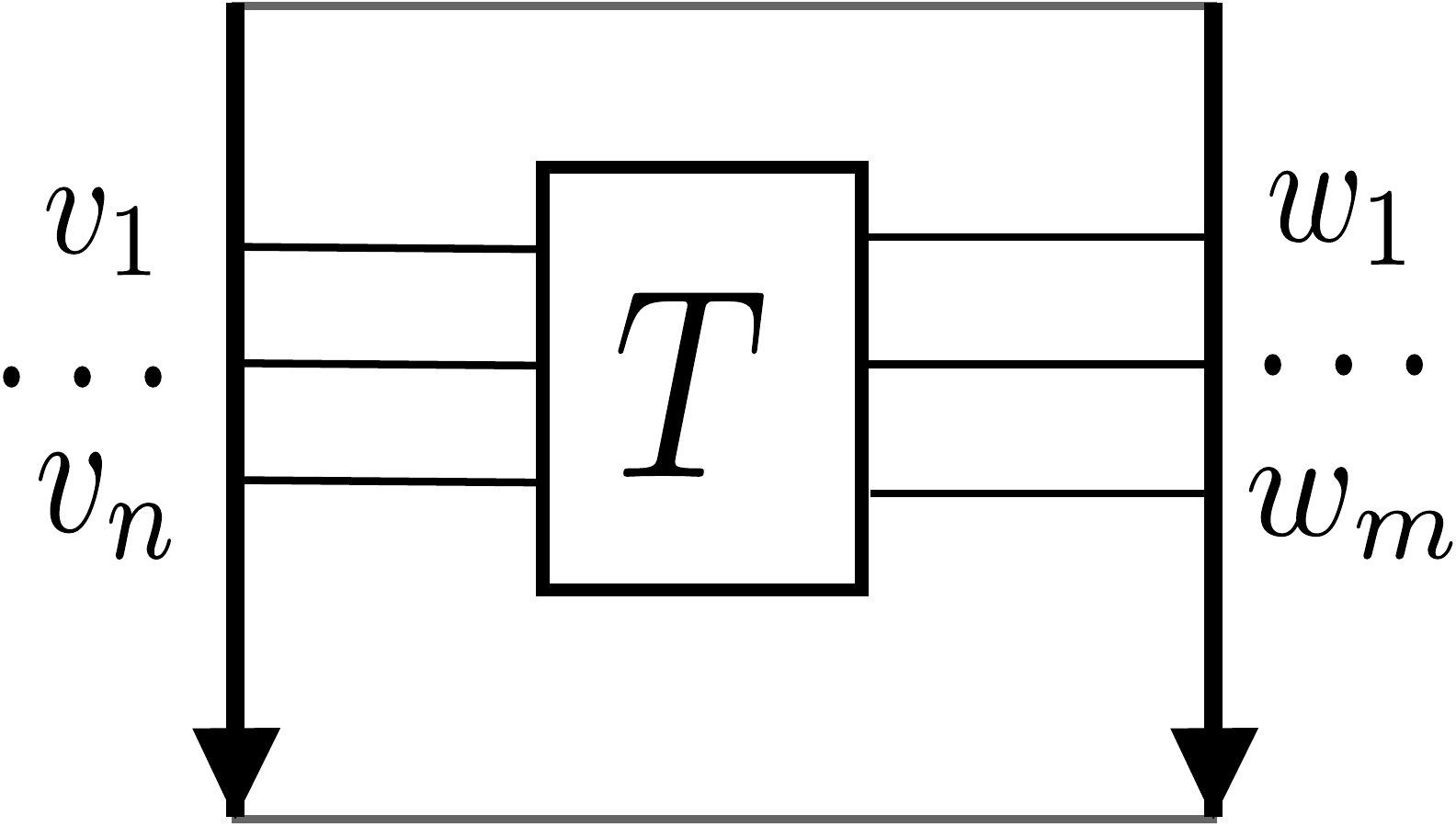}} \otimes\adjustbox{valign=c}{\includegraphics[width=1.7cm]{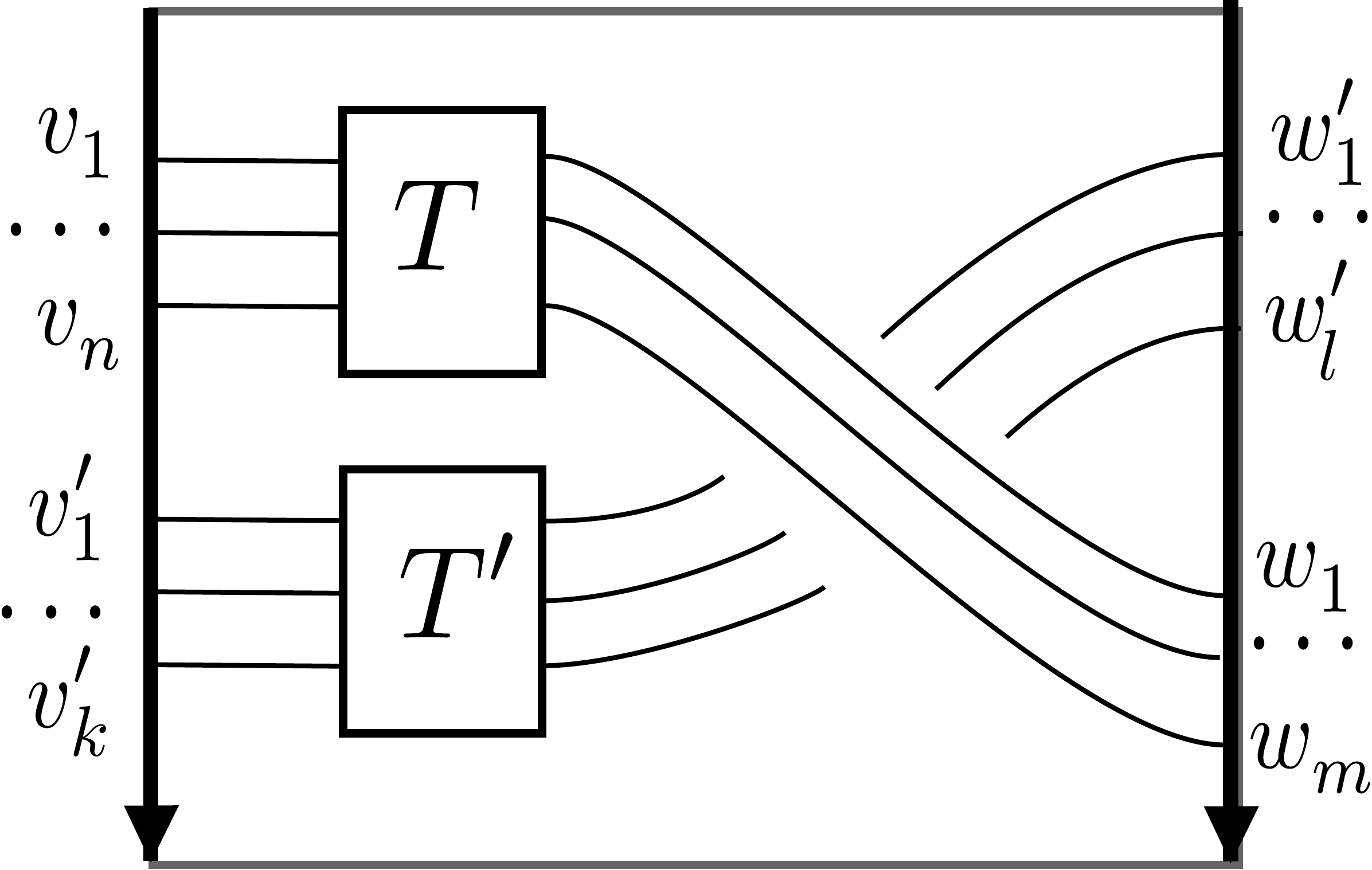}} \right)
 := \epsilon\left( \adjustbox{valign=c}{\includegraphics[width=3cm]{Figure-03c}} \right),\]
and a co-twist $\Theta\colon\mathcal{S}_A(\mathbb{B})\to k$ by
\[\Theta \left(\adjustbox{valign=c}{\includegraphics[width=1.7cm]{Figure-03a}} \right) := \epsilon \left( \adjustbox{valign=c}{\includegraphics[width=2cm]{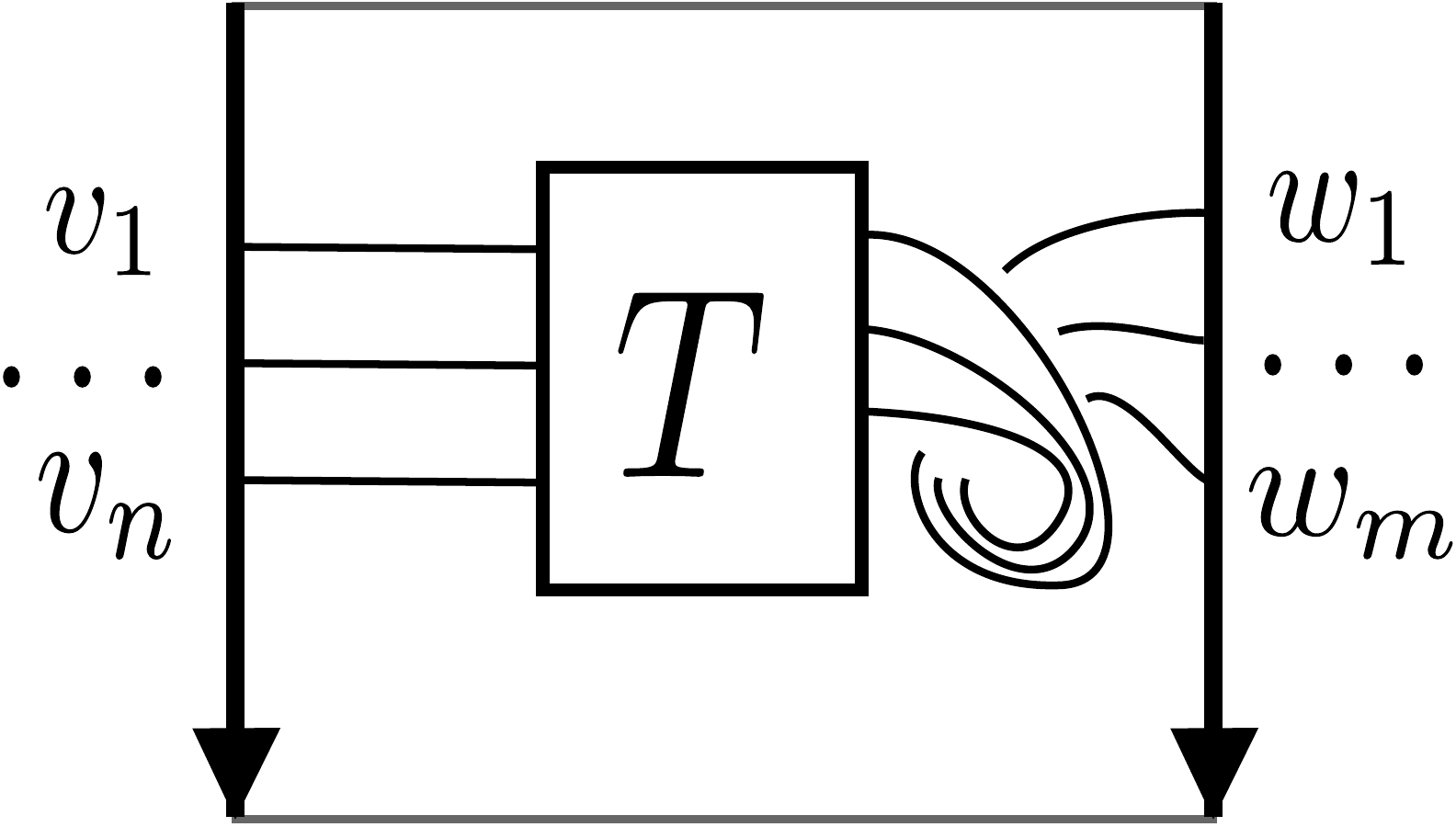}} \right).\]

Let $\alpha_{ij}:= \adjustbox{valign=c}{\includegraphics[width=0.8cm]{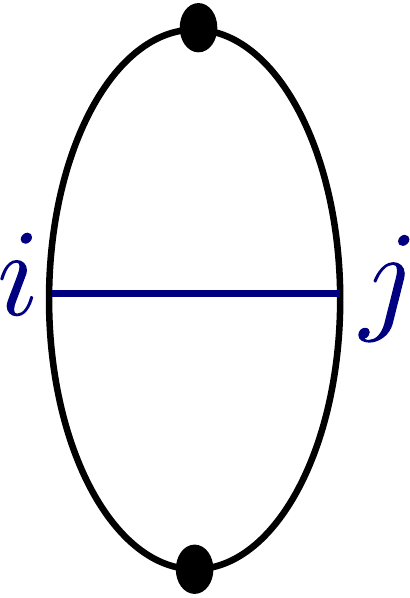}} \in \mathcal{S}_q(\mathbb{B})$ for $i,j = \pm$.

\begin{Theorem}[{\cite{CostantinoLe19, KojuQuesneyClassicalShadows}}] \label{theorem_bigon}The co-ribbon Hopf algebra $(\mathcal{S}_A(\mathbb{B}), \Delta, \mathbf{r}, \Theta)$ is isomorphic to $\mathcal{O}_q[\SL_2]$ through an isomorphism sending $\alpha_{++}$, $\alpha_{+-}$, $\alpha_{-+}$, $\alpha_{--}$ to $a$, $b$, $c$, $d$ respectively.
\end{Theorem}

For a general marked surface $\mathbf{\Sigma}$ with a boundary arc $c$, since $\left( \mathbb{B} \cup \mathbf{\Sigma}\right)_{a_R\# c}\allowbreak =\mathbf{\Sigma}$ and $ (\mathbf{\Sigma}\cup \mathbb{B} )_{c\# a_L} = \mathbf{\Sigma}$, we have some left and right comodule maps
\[\Delta_c^L := \theta_{a_R\#c} \colon \ \mathcal{S}_A(\mathbf{\Sigma}) \to \mathcal{O}_q[\SL_2] \otimes \mathcal{S}_A(\mathbf{\Sigma})\]
 and
\[\Delta_c^R:= \theta_{c\#a_L}\colon \ \mathcal{S}_A(\mathbf{\Sigma}) \to \mathcal{S}_A(\mathbf{\Sigma})\otimes \mathcal{O}_q[\SL_2] , \]
as illustrated in Figure~\ref{fig_comodule_maps}.

\begin{figure}[!h]\centering
\includegraphics[width=11cm]{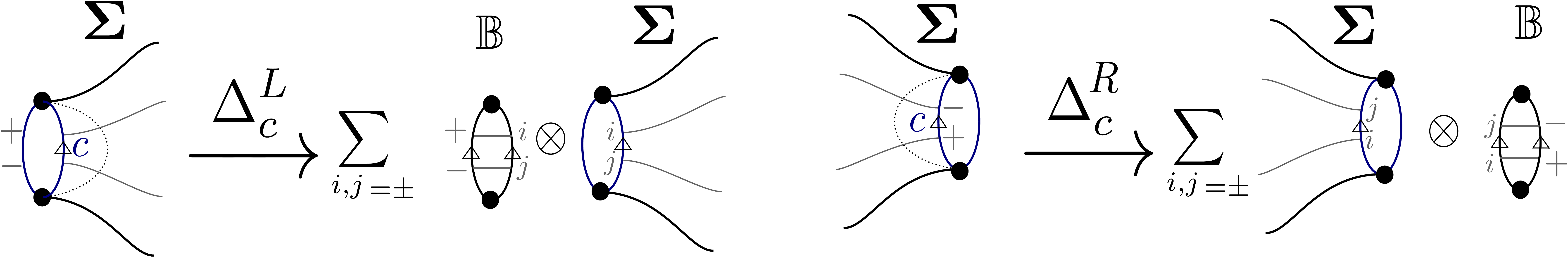}
\caption{Comodule maps.}\label{fig_comodule_maps}
\end{figure}

Note that by composing the $\Delta_a^L$ for $a\in \mathcal{A}$, we get a comodule map $\Delta^L\colon \mathcal{S}_A(\mathbf{\Sigma}) \to (\mathcal{O}_q[\SL_2])^{\otimes \mathcal{A}} \allowbreak \otimes \mathcal{S}_A(\mathbf{\Sigma})$.

\subsection{Bad arcs and reduced stated skein algebras}

Let $\mathbf{\Sigma}$ a marked surface and $p$ a boundary puncture between two consecutive boundary arcs $a$ and $b$ on the same boundary component $\partial$ of $\partial \Sigma$.
The orientation of $\Sigma$ induces an orientation of $\partial$ so a cyclic ordering of the elements of $\partial \cap \mathcal{A}$. We suppose that $a$ is followed by $p$ which is followed by $b$ in this ordering. We denote by $\alpha(p)$ an arc with one endpoint $v_a\in a$ and one endpoint $v_b \in b$ such that $\alpha(p)$ can be isotoped inside $\partial$. Let $\alpha(p)_{ij}\in \mathcal{S}_A(\mathbf{\Sigma})$ be the class of the stated arc $(\alpha(p), s)$ where $s(v_b)=i$ and $s(v_a)=j$ (see Figure~\ref{fig_bad_arc}).

\begin{figure}[!h]\centering
\includegraphics[width=4cm]{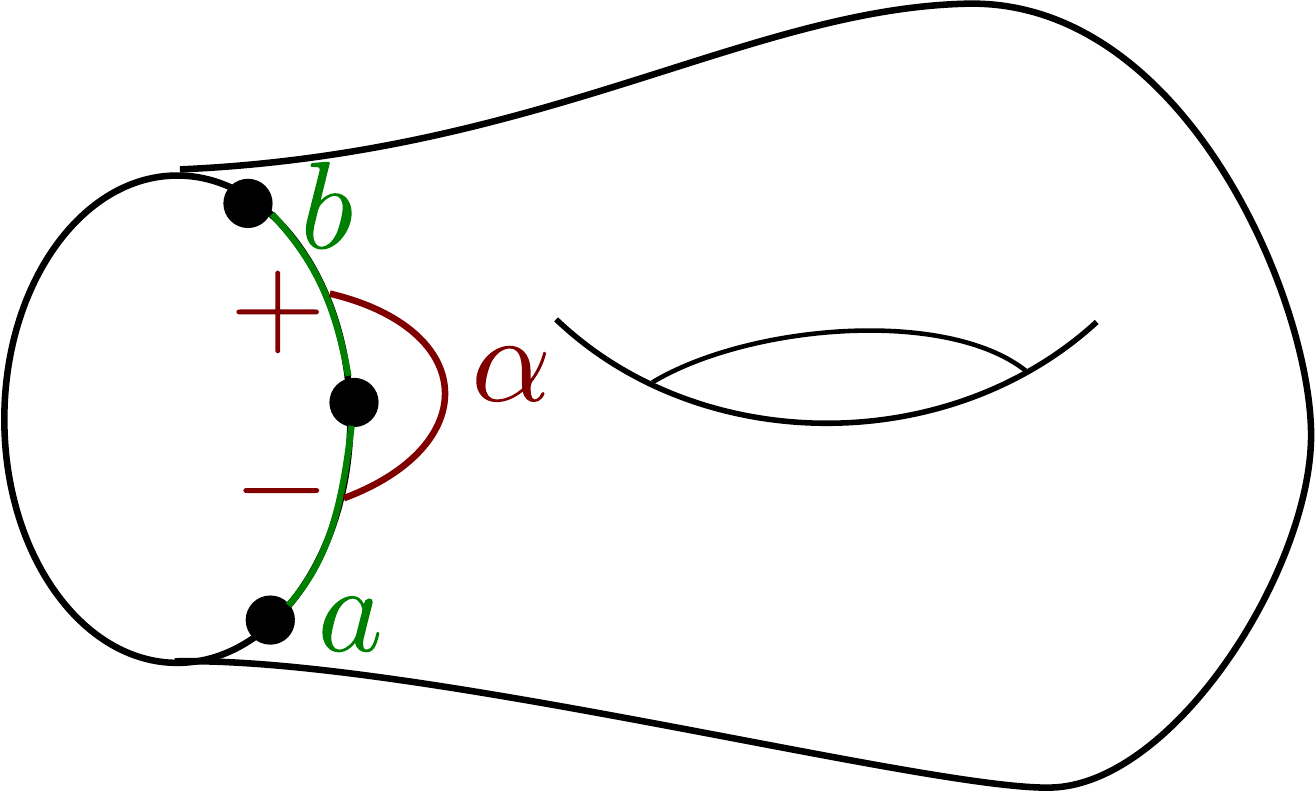}
\caption{A bad arc.}\label{fig_bad_arc}
\end{figure}

\begin{Definition}
The element $\alpha(p)_{-+}\in \mathcal{S}_A(\mathbf{\Sigma})$ is called the \textit{bad arc associated to}~$p$.
The \textit{reduced stated skein algebra} $\overline{\mathcal{S}}_A(\mathbf{\Sigma})$ is the quotient of $\mathcal{S}_A(\mathbf{\Sigma})$ by the ideal generated by all bad arcs.
\end{Definition}

\subsection{Basic properties}

\begin{Theorem}\label{theorem_basic_ppty} Let $k$ be an arbitrary commutative, unital ring with $A^{1/2}\in k^{\times}$.
\begin{enumerate}\itemsep=0pt
\item[$1.$] $\mathcal{S}_{A}(\mathbf{\Sigma})$ and $\overline{\mathcal{S}}_A(\mathbf{\Sigma})$ are free $k$ modules with $($explicit$)$ bases~$\mathcal{B}$ and $\overline{\mathcal{B}}$ made of $($classes$)$ of stated tangles {\rm \cite[Theorem~7.1]{CostantinoLe19}}, {\rm \cite[Theorem~2.11]{LeStatedSkein}}.
 \item[$2.$] $\mathcal{S}_A(\mathbf{\Sigma})$ and $\overline{\mathcal{S}}_A(\mathbf{\Sigma})$ are domains~{\rm \cite{BonahonWongqTrace, CostantinoLe19, LeStatedSkein, PrzytyckiSikora_SkeinDomain}}.
 \item[$3.$] $\mathcal{S}_A(\mathbf{\Sigma})$ and $\overline{\mathcal{S}}_A(\mathbf{\Sigma})$ are finitely generated. When $\mathbf{\Sigma}$ does not have unmarked component, we even have explicit finite presentations of them {\rm (\cite{BullockGeneratorsSkein}} for unmarked surfaces, {\rm \cite{KojuPresentationSSkein}} for marked surfaces$)$.
 \end{enumerate}
 \end{Theorem}

\subsection{The quantum fusion operation}

Recall that the triangle $\mathbb{T}$ is a disc with three boundary arcs, say $i$, $j$, $k$ and that we defined the fusion of a marked surface $\mathbf{\Sigma}$ along two boundary arcs $a$, $b$ by
\[ \mathbf{\Sigma}_{a \circledast b}:= \left( \mathbf{\Sigma}\bigsqcup \mathbb{T} \right) _{a\#i, b\#j}.
\]

\begin{Definition}\label{def_fusion_quantique}
Let $(H,\mu_H, \eta, \Delta, \epsilon_H, S, \mathbf{r})$ be a cobraided Hopf algebra and $\mathbf{A}=(A, \mu, \epsilon)$ be an algebra object in $H\otimes H-\mathrm{Comod}$ and denote by $\Delta_{H\otimes H} \colon A \to H \otimes H \otimes A$ its comodule map. Write $\Delta^1:= (\id \otimes \epsilon \otimes \id) \circ \Delta_{H\otimes H}\colon A\to H\otimes A$ and $\Delta^2\colon (\epsilon \otimes \id \otimes \id) \circ \Delta_{H\otimes H}\colon A \to H \otimes A$.
 The \textit{fusion} $\mathbf{A}_{1\circledast 2}$ is the algebra object $(A_{1\circledast 2}, \mu_{1\circledast 2}, \epsilon_{1\circledast 2})$ in $H-\mathrm{Comod}$ where
 \begin{enumerate}\itemsep=0pt
 \item[1)] $A_{1\circledast 2}= A$ as a $k$-module and $\epsilon_{1 \circledast 2}=\epsilon$,
 \item[2)] the product is the composition
\[ \mu_{1 \circledast 2} \colon \ A\otimes A \xrightarrow{\Delta_1\otimes \Delta_2} H\otimes A \otimes H \otimes A \xrightarrow{\id \otimes \tau_{A,H} \otimes \id} H\otimes H \otimes A \otimes A \xrightarrow{\mathbf{r}\circ \tau_{H,H} \otimes \mu} A,\]
 \item[3)] the comodule map is $\Delta_H:= (\mu_H \otimes \id) \circ \Delta_{H\otimes H} $.
 \end{enumerate}
 \end{Definition}
 For instance, if $V$ and $W$ are two algebra objects in $H-\mathrm{Comod}$, then $V\otimes_k W$ is an algebra object in $H^{\otimes 2}-\mathrm{Comod}$ and its fusion $(V\otimes_k W)_{1\circledast 2}$ is called the \textit{cobraided tensor product} and is denoted by $V\overline{\otimes} W$. Identify $V$ with $V\otimes 1$ and $W$ with $1\otimes W$ in $V\otimes W$. Its product is characterized by the formula
\[ \mu (x\otimes y) = \begin{cases}
 \mu_V(x\otimes y), &\text{if }x,y \in V, \\
 \mu_W(x\otimes y), &\text{if }x,y \in W, \\
 x\otimes y, & \text{if }x\in V, y \in W, \\
 c_{W,V}(x\otimes y), & \text{if }x\in W, y \in V.
 \end{cases}
\]
 Here, the braiding is defined by
\begin{align*}
 c_{V,W} \colon \ V\otimes W & \xrightarrow{\tau_{V,W}} W\otimes V \xrightarrow{\Delta_W\otimes \Delta_V} H\otimes W \otimes H \otimes V \\
 & \xrightarrow{\id_H \otimes \tau_{W,H} \otimes \id_V} H\otimes H \otimes W \otimes V \xrightarrow{r\otimes \id_W \otimes \id_V} W\otimes V.
\end{align*}

Now consider a marked surface $\mathbf{\Sigma}_{a\circledast b}$ obtained by fusioning two boundary arcs $a$ and $b$. Fix an orientation $\mathfrak{o}
 $ of the boundary arcs of $\mathbf{\Sigma}\bigsqcup \mathbb{T}$, as in Figure~\ref{fig_fusion} and let~$\mathfrak{o}'$ the induced orientation of the boundary arcs of~$\mathbf{\Sigma}'$.
 Define a linear map $\Psi_{a\circledast b}\colon \mathcal{S}_A(\mathbf{\Sigma}) \to \mathcal{S}_A(\mathbf{\Sigma}_{a\#b})$ by $\Psi_{a\circledast b} ([D,s]^{\mathfrak{o}}):= [D',s']$ where $(D',s')$ is obtained from $(D,s)$ by gluing to each point of $D\cap a$ a straight line in $\mathbb{T}$ between~$e_1$ and~$e_3$ and by gluing to each point of $D\cap b$ a straight line in $\mathbb{T}$ between~$e_2$ and~$e_3$. Figure~\ref{fig_fusion} illustrates $\Psi_{a\circledast b}$.
 \begin{figure}[!h]\centering\includegraphics[width=6cm]{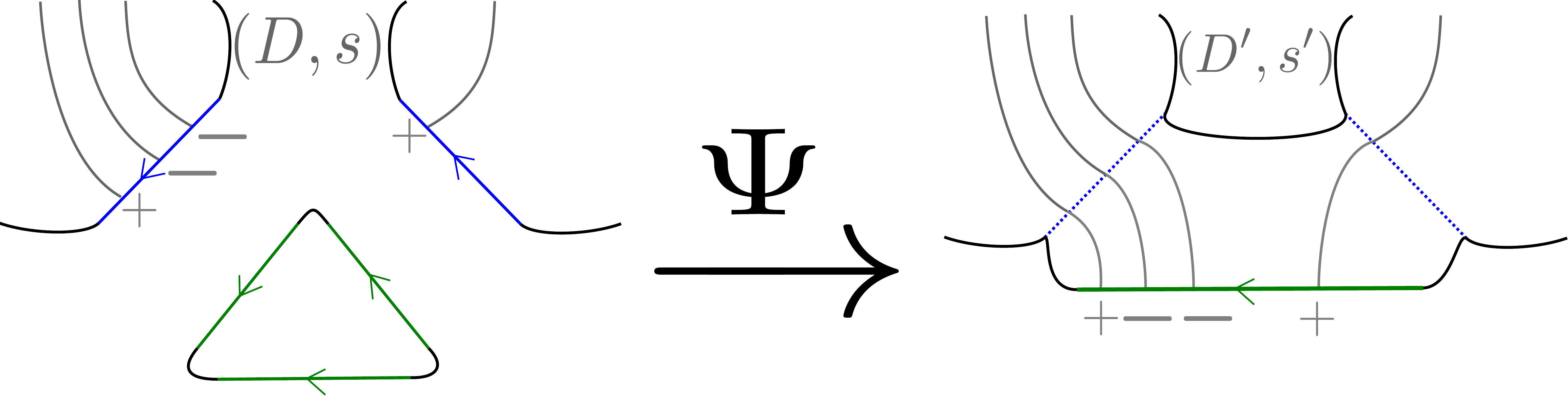}
\caption{An illustration of $\Psi_{a\circledast b}$.}\label{fig_fusion}
\end{figure}
 \begin{Theorem}[{Costantino--L\^e \cite[Theorem~4.13]{CostantinoLe19}}]\label{theorem_fusion} The linear map $\Psi_{a\circledast b}$ is an isomorphism of $k$-modules which identifies $\mathcal{S}_A(\mathbf{\Sigma}_{a\circledast b})$ with the fusion $\mathcal{S}_A(\mathbf{\Sigma})_{a\circledast b}$.
 \end{Theorem}

 Costantino and L\^e proved Theorem \ref{theorem_fusion} in the particular case where $\mathbf{\Sigma}=\mathbf{\Sigma}_1\bigsqcup \mathbf{\Sigma}_2$ with~$a$,~$b$ boundary arcs of~$\mathbf{\Sigma}_1$ and~$\mathbf{\Sigma}_2$ respectively. In this case, the conclusion of the theorem reads:
\[ \mathcal{S}_A(\mathbf{\Sigma}_{a\circledast b}) \cong \mathcal{S}_A(\mathbf{\Sigma}_1) \overline{\otimes} \mathcal{S}_A(\mathbf{\Sigma}_2).\]
 Higgins generalized in~\cite{Higgins_SSkeinSL3} this result for $\SL_3$ stated skein algebras. Higgins' proof extends word-by-word to prove the more general statement of Theorem~\ref{theorem_fusion} as reproduced here:

\begin{proof} The fact that $\Psi_{a\circledast b}$ is surjective is an easy consequence of the skein relation:
\[
 A^{1/2}
\heightexch{->}{-}{+}
- A^{5/2}
\heightexch{->}{+}{-}
=
\heightcurve.
\] The injectivity is proved using the following elegant argument of Higgins in \cite{Higgins_SSkeinSL3}. Since $\mathbf{\Sigma}_{a\circledast b}$ is obtained from $\mathbf{\Sigma}\bigsqcup \mathbb{T}$ by gluing some boundary arcs, we have a gluing map $\theta\colon \mathcal{S}_A(\mathbf{\Sigma}_{a\circledast b}) \to \mathcal{S}_A (\mathbb{T}) \otimes \mathcal{S}_A(\mathbf{\Sigma})$. Let $i\colon \mathbb{T} \to \mathbb{B}$ be the embedding of marked surfaces sending $e_3$ to $a_R$ and~$e_1$,~$e_2$ to~$a_L$ with $e_1>e_2$ and denote by $i_* \colon \mathcal{S}_A \to \mathcal{O}_q[\SL_2]$ the induced morphism. Consider the composition
\[ \Phi\colon \ \mathcal{S}_A(\mathbf{\Sigma}_{a\circledast b}) \xrightarrow{\theta} \mathcal{S}_A (\mathbb{T}) \otimes \mathcal{S}_A(\mathbf{\Sigma})\xrightarrow{i_*\otimes \id} \mathcal{O}_q[\SL_2] \otimes \mathcal{S}_A(\mathbf{\Sigma})\xrightarrow{\epsilon \otimes \id} \mathcal{S}_A(\mathbf{\Sigma}).
\]
 As illustrated in Figure~\ref{fig_defusion}, it is easy to see that $\Phi$ is a left inverse to $\Psi_{a\circledast b}$, thus $\Psi_{a\circledast b}$ is an isomorphism. It remains to prove that
 the pullback by $\Psi_{a\circledast b}$ of the product in $ \mathcal{S}_A(\mathbf{\Sigma}_{a\circledast b}) $ is the fusion product
$ \mu_{a\circledast b}$. This fact is illustrated in Figure~\ref{fig_fusion_product}.\end{proof}
\begin{figure}[!h]\centering \includegraphics[width=16cm]{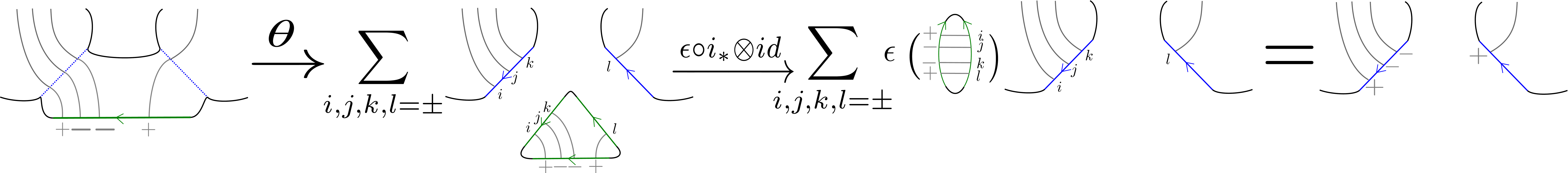}
\caption{An illustration of the equality $\Phi \circ \Psi_{a\circledast b} = \id$.}\label{fig_defusion}
\end{figure}

\begin{figure}[!h]\centering \includegraphics[width=10cm]{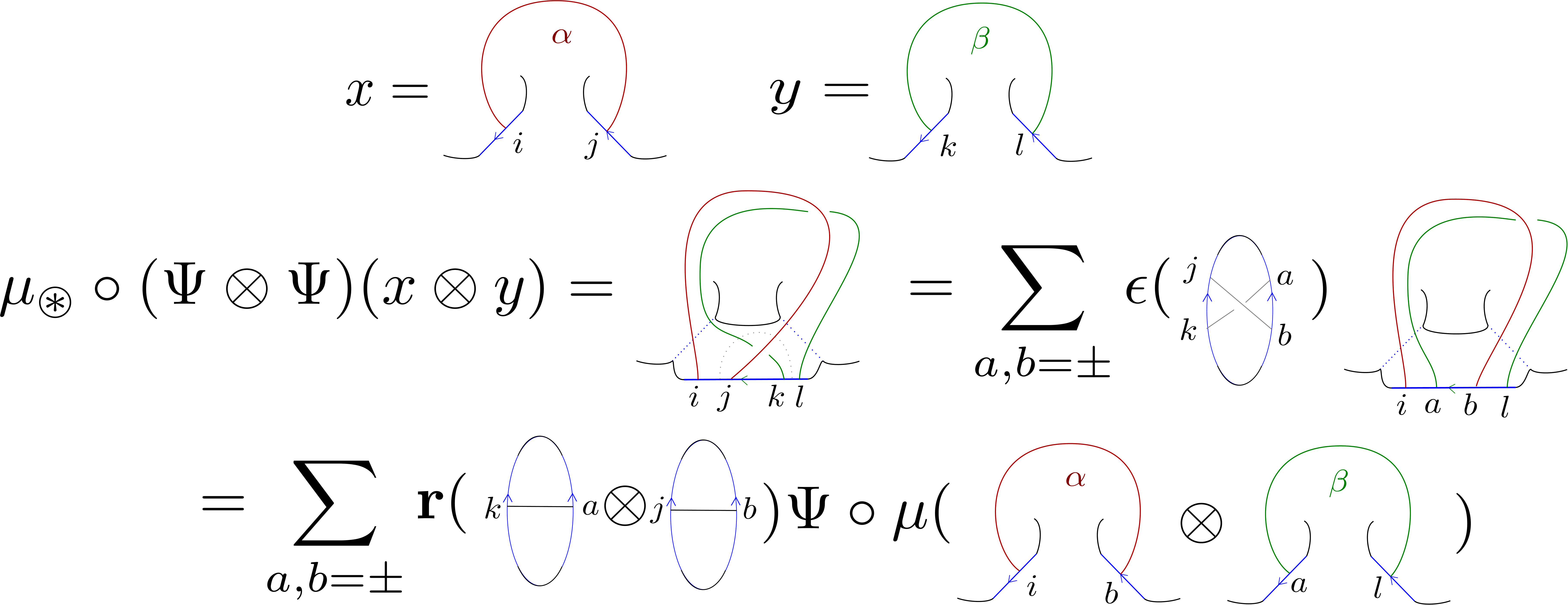}
\caption{An illustration of the fusion product $\mu_{\circledast}$.}\label{fig_fusion_product}
\end{figure}

\subsection{Centers and PI-degrees at roots of unity}\label{sec_center}

\begin{Notations}
Suppose that $A\in \mathbb{C}^*$ is a root of unity of odd order $N$. We denote by
\begin{itemize}\itemsep=0pt
\item $\mathcal{Z}(\mathbf{\Sigma})$ the center of $\mathcal{S}_A(\mathbf{\Sigma})$ and $\overline{\mathcal{Z}}(\mathbf{\Sigma})$ the center of $\overline{\mathcal{S}}_A(\mathbf{\Sigma})$,
\item $\widehat{\mathcal{X}}(\mathbf{\Sigma}):= \operatorname{Specm}( \mathcal{Z}(\mathbf{\Sigma}))$ and $\widehat{\overline{\mathcal{X}}}(\mathbf{\Sigma}):= \operatorname{Specm}\big( \overline{\mathcal{Z}}(\mathbf{\Sigma})\big)$,
\item $\mathcal{X}(\mathbf{\Sigma}):= \operatorname{Specm}(\mathcal{S}_{+1}(\mathbf{\Sigma}))$ and $\overline{\mathcal{X}}(\mathbf{\Sigma}):= \operatorname{Specm}( \overline{\mathcal{S}}_{+1}(\mathbf{\Sigma}))$.
\end{itemize}
\end{Notations}

\begin{Theorem}[{\cite{BonahonWongqTrace} for unmarked surfaces, \cite{KojuQuesneyClassicalShadows} for marked surfaces}] \label{theorem_chebyshev}
When $A\in \mathbb{C}^*$ is a root of unity of odd order $N$, there exist embeddings
\[ \operatorname{Ch}_A \colon \ \mathcal{S}_{+1}(\mathbf{\Sigma}) \hookrightarrow \mathcal{Z} ( \mathbf{\Sigma} )\qquad \text{and} \qquad \operatorname{Ch}_A \colon \ \overline{\mathcal{S}}_{+1}(\mathbf{\Sigma}) \hookrightarrow \overline{\mathcal{Z}} ( \mathbf{\Sigma} )
\]
named \textit{Chebyshev--Frobenius morphisms}, sending the $($commutative$)$ algebra at $+1$ into the center of the skein algebra at $A^{1/2}$. Moreover, $\operatorname{Ch}_A$ is characterized by the facts that if $\gamma$ is a closed curve, then $\operatorname{Ch}_A(\gamma) = T_N(\gamma)$, where $T_N(X)$ is the $N^{\rm th}$ Chebyshev polynomial of first type, and if $\alpha_{ij}$ is a stated arc, then $\operatorname{Ch}_A(\alpha_{ij})= \alpha_{ij}^{(N)}$ is the class of $N$ parallel copies of $\alpha_{ij}$ pushed along the framing direction.
\end{Theorem}

 \begin{Definition} \label{def_central_elements}\quad
 \begin{enumerate}\itemsep=0pt
 \item For $p$ an inner puncture (i.e., an unmarked connected component of $\partial \Sigma$), we denote by $\gamma_p \in \mathcal{S}_A(\mathbf{\Sigma})$ the class of a peripheral curve encircling $p$ once.
 \item For $\partial \in \pi_0(\partial \Sigma)$ a boundary component which intersects $\mathcal{A}$ non trivially, denote by $p_1, \dots, p_n$ the boundary punctures in $\partial$ cyclically ordered by the orientation of $\partial$ (induced from that of $\Sigma$) and define the elements in $\overline{\mathcal{S}}_A(\mathbf{\Sigma})$:
\[ \alpha_{\partial} := \alpha(p_1)_{++} \cdots \alpha(p_n)_{++}, \qquad \text{and} \qquad \alpha_{\partial}^{-1}:= \alpha(p_1)_{--} \cdots \alpha_(p_n)_{--}.\]
 In $\overline{\mathcal{S}}_A(\mathbf{\Sigma})$, we have $\alpha_{\partial} \alpha_{\partial}^{-1} =1$ (see~\cite{KojuAzumayaSkein} for a~proof).
\end{enumerate}
\end{Definition}

 For a prime ring $R$ with finite rank $r$ over its center, the rank $r$ is a~perfect square (by a~celebrated theorem of Posner--Formanek detailed in \cite[Theorem~1.13.3]{BrownGoodearl}) and we call \textit{PI-degree} of~$R$ the square root $\sqrt{r}$.

 \begin{Theorem}\label{theorem_center} Suppose that $A\in \mathbb{C}^*$ is a root of unity of odd order~$N$.
 \begin{enumerate}\itemsep=0pt
\item[$1.$] If $\mathbf{\Sigma}$ is unmarked, then $(i)$~the center of $\mathcal{S}_A(\mathbf{\Sigma})$ is generated by the image of the Chebyshev--Frobenius morphism together with the eventual peripheral curves $\gamma_p$ for $p$ an inner puncture. $(ii)$~$\mathcal{S}_A(\mathbf{\Sigma})$ is finitely generated over the image of the Chebyshev--Frobenius morphism $($so over its center$)$ and $(iii)$ for $\mathbf{\Sigma}=(\Sigma_{g,n}, \varnothing)$ the PI-degree of $\mathcal{S}_A(\mathbf{\Sigma})$ is $N^{3g-3+n}$~{\rm \cite{FrohmanKaniaLe_UnicityRep}}.
\item[$2.$] For any marked surface then $(i)$~the center of $\overline{\mathcal{S}}_A(\mathbf{\Sigma})$ is generated by the image of the Chebyshev--Frobenius morphism together with the peripheral curves $\gamma_p$ associated to inner punctures and the elements $\alpha_{\partial}^{\pm 1}$ associated to boundary components $\partial \in \pi_0(\partial \Sigma)$. $(ii)$~both~$\mathcal{S}_A(\mathbf{\Sigma})$ and $\overline{\mathcal{S}}_A(\mathbf{\Sigma})$ are finitely generated over the image of the Chebyshev--Fro\-be\-nius morphisms $($so over their center$)$. $(iii)$~For $\mathbf{\Sigma}=(\Sigma_{g,n}, \mathcal{A})$, the PI-degree of $\overline{\mathcal{S}}_A(\mathbf{\Sigma})$ is $N^{3g-3+n+|\mathcal{A}|}$~{\rm \cite{KojuAzumayaSkein}}.
\item[$3.$] The center of $\mathcal{S}_A(\mathbf{\Sigma}_{g}^*)$ is equal to the image of the Chebyshev--Frobe\-nius morphism~{\rm \cite{GanevJordanSafranov_FrobeniusMorphism}}.
\end{enumerate}
\end{Theorem}

Denote by $\pi\colon \widehat{\mathcal{X}}(\mathbf{\Sigma}) \to \mathcal{X}(\mathbf{\Sigma})$ and by $\overline{\pi}\colon \widehat{\overline{\mathcal{X}}}(\mathbf{\Sigma}) \to \overline{\mathcal{X}}(\mathbf{\Sigma})$ the dominant maps induced by the Chebyshev--Frobenius morphisms. Note that, since $\overline{\mathcal{S}}_{+1}(\mathbf{\Sigma})$ is a quotient of $\mathcal{S}_{+1}(\mathbf{\Sigma})$, then $\overline{\mathcal{X}}(\mathbf{\Sigma})$ is a subvariety of $\mathcal{X}(\mathbf{\Sigma})$. Theorem~\ref{theorem_center} implies that $\pi$ and $\overline{\pi}$ are finite branched coverings.

\begin{Theorem}\label{theorem_PI_Deg} Suppose that $A\in \mathbb{C}^*$ is a root of unity of odd order $N$. Then the PI-degree of~$\mathcal{S}_A(\mathbf{\Sigma}_{g}^*)$ is $N^{3g}$.
\end{Theorem}

\begin{figure}[!h]\centering \includegraphics[width=6cm]{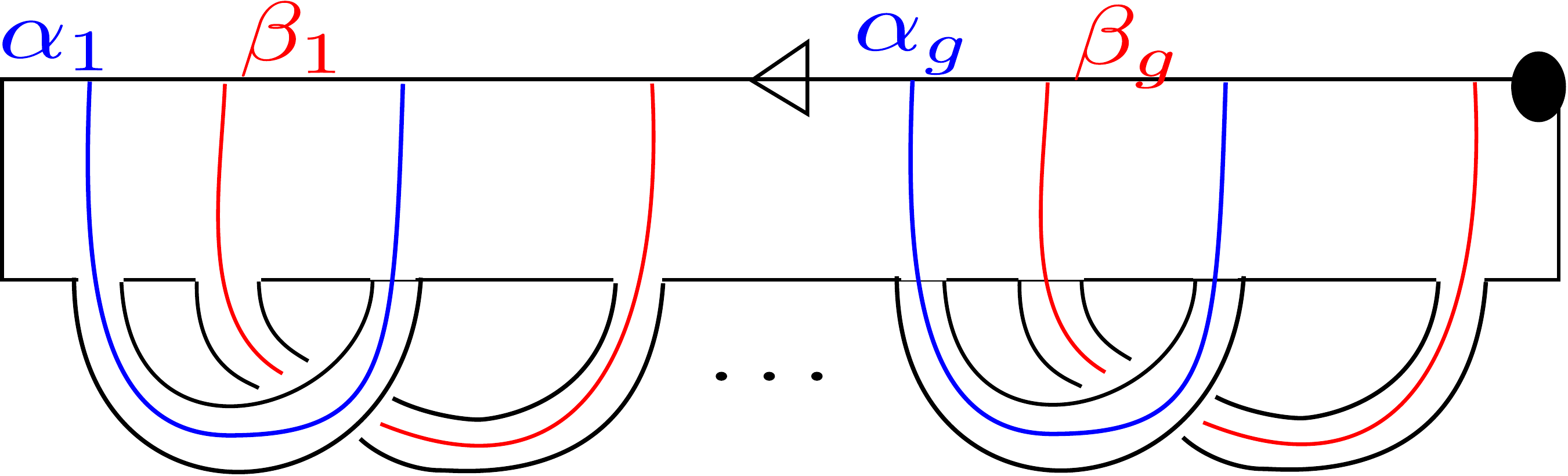}
\caption{Some arcs in $\mathbf{\Sigma}_g^*$.}\label{fig_arcs_Alekseev}
\end{figure}

\begin{proof}Let $\mathbb{G}=\{ \alpha_1, \beta_1, \dots, \alpha_g, \beta_g\}$ be the set of arcs in $\mathbf{\Sigma}_g^*$ drawn in Figure~\ref{fig_arcs_Alekseev}. For $\gamma \in \mathbb{G}$ with endpoints $v$, $w$ such that $h(v)<h(w)$ and for $i,j = \pm$, write $\gamma_{ij} \in \mathcal{S}_A(\mathbf{\Sigma}_g^*)$ the class of the arc~$\gamma$ with state sending~$v$ to $i$ and~$w$ to~$j$. For $n\geq 0$, write $n=q_n N +r_n$ its Euclidian division characterized by the property that $0\leq r_n <N$. Set
\[\gamma_{ij}^{\langle n\rangle}:= \big(\gamma_{ij}^{(N)}\big)^{q_n} (\gamma_{ij})^{r_n}, \]
where we recall that $\gamma_{ij}^{(N)}$ is the central element made of $N$ parallel copies of $\gamma_{ij}$ \big(here $\gamma_{ij}^{(N)} \neq \gamma_{ij}^N$\big). Let
\[\mathcal{B}(\gamma):= \big\{ \gamma_{++}^{\langle a \rangle } \gamma_{+-}^{\langle b \rangle } \gamma_{--}^{\langle c \rangle } , \,a,b,c \geq 0\big\} \cup \big\{ \gamma_{++}^{\langle a \rangle } \gamma_{-+}^{\langle b \rangle } \gamma_{--}^{\langle c \rangle } ,\, a,b,c \geq 0\big\}\]
and
\[\mathcal{B}^{0}(\gamma):= \big\{ \gamma_{++}^{\langle a \rangle } \gamma_{+-}^{\langle b \rangle } \gamma_{--}^{\langle c \rangle } ,\, 0\leq a,b,c \leq N-1\big\} \cup \big\{ \gamma_{++}^{\langle a \rangle } \gamma_{-+}^{\langle b \rangle } \gamma_{--}^{\langle c \rangle } ,\, 0\leq a,b,c \leq N-1\big\}.\]
Consider the sets
\[\mathcal{B}:= \big\{ x_1 y_1 \cdots x_g, y_g, \, x_i \in \mathcal{B}(\alpha_i), \, y_i \in \mathcal{B}(\beta_i)\big\},\]
and
\[\mathcal{B}^{0}:= \big\{ x_1 y_1 \cdots x_g, y_g, \, x_i \in \mathcal{B}^0(\alpha_i),\, y_i \in \mathcal{B}^0(\beta_i)\big\},\]
Then by \cite[Theorem~3.7]{KojuPresentationSSkein}, $\mathcal{B}$ is a basis of $\mathcal{S}_A(\mathbf{\Sigma}_{g}^*)$. By Theorem~\ref{theorem_center}, the center of $\mathcal{S}_A(\mathbf{\Sigma}_{g}^*)$ is the image of the Chebyshev--Frobenius morphism $\operatorname{Ch}_A$ and Theorem~\ref{theorem_chebyshev} implies that $\mathcal{B}^0$ is a basis of~$\mathcal{S}_A(\mathbf{\Sigma}_g^*)$ over its center. Therefore the PI-degree of $\mathcal{S}_A(\mathbf{\Sigma}_g^*)$ is the square root of the cardinal of~$\mathcal{B}^0$. Since $\mathcal{B}^0$ has~$N^{6g}$ elements, this concludes the proof.
\end{proof}

\begin{Remark} In~\cite{LeYu_Survey}, L\^e and Yu announced that they have computed explicitly the center of any stated skein algebra at roots of unity and that they have computed their PI-degree as well, in some still un-prepublished work. For a root of unity of odd order $N$ and a connected marked surface $\mathbf{\Sigma}$ of genus $g$ with $s$ boundary arcs and~$n_{\rm even}$ (resp.~$n_{\rm odd}$) boundary components with an even (resp.~odd) number of boundary arcs, they announced that the PI-degree of $\mathcal{S}_A(\mathbf{\Sigma})$ is $N^{3g-3+n_{\rm even}+\frac{3}{2}(n_{\rm odd}+s)}$. This agrees with our formula in Theorem~\ref{theorem_PI_Deg}.
\end{Remark}

\section{Geometric study}\label{sec_geometric}

\subsection{Poisson bracket arising from deformation quantization}

The algebras $\mathcal{S}_{+1}(\mathbf{\Sigma})$ and $\overline{\mathcal{S}}_{+1}(\mathbf{\Sigma})$ have Poisson brackets defined as follows. Let $\mathcal{S}_{+1}$ be either~$\mathcal{S}_{+1}(\mathbf{\Sigma})$ or $\overline{\mathcal{S}}_{+1}(\mathbf{\Sigma})$ with $A^{1/2}=+1 \in \mathbb{C}$ and denote by $\mathcal{S}_{A_{\hbar}}$ the same algebra taken in the ring~$\mathbb{C}[[\hbar]]$ of formal power series in $\hbar$ with $A_{\hbar}^{1/2}:= \exp(\hbar/2) \in \mathbb{C}[[\hbar]]$.
Consider the basis $B$ of the first item of Theorem~\ref{theorem_basic_ppty} made of stated tangles.
An element $b\in B$ can be seen both as an element of~$\mathcal{S}_{+1}$ or~$\mathcal{S}_{A_{\hbar}}$, and we define a linear isomorphism $\Psi\colon \mathcal{S}_{+1} \otimes_{\mathbb{C}}\mathbb{C}[[\hbar]] \cong \mathcal{S}_{A_{\hbar}}$ by setting $\Psi(b)=b$ for all $b\in B$. Let $\star$ denote the pull-back by $\Psi$ of the product of~$\mathcal{S}_{A_{\hbar}}$.

\begin{Definition}\label{def_Poisson_bracket} The Poisson bracket $\{\cdot, \cdot\}$ on $\mathcal{S}_{+1}$ is defined by
\[ x \star y - y \star x \equiv \hbar \{x,y\} \ \big({\rm mod}~\hbar^2\big) \qquad \text{for all }x,y \in \mathcal{S}_{+1}.\]
\end{Definition}

As a result $\overline{\mathcal{X}}(\mathbf{\Sigma}) \subset \mathcal{X}(\mathbf{\Sigma})$ are affine Poisson varieties. The Poisson bracket $\{\cdot, \cdot\}$ does not depend on the choice of the basis $B$.

 Note that if $f\colon \mathcal{S}_A(\mathbf{\Sigma}_1) \to \mathbf{S}_A(\mathbf{\Sigma}_2)$ is a morphism of algebras for $A^{1/2}$ a generic element, then it induces a morphism of algebras $f_{\hbar}\colon \mathcal{S}_{A_{\hbar}}(\mathbf{\Sigma}_1) \to \mathbf{S}_{A_{\hbar}}(\mathbf{\Sigma}_2)$ so Definition~\ref{def_Poisson_bracket} implies that the induced morphism $f_{+1}\colon \mathcal{S}_{+1}(\mathbf{\Sigma}_1) \to \mathcal{S}_{+1}(\mathbf{\Sigma}_2)$ is Poisson and defines a Poisson map $f^*\colon \mathcal{X}(\mathbf{\Sigma}_2) \to \mathcal{X}(\mathbf{\Sigma}_1)$. We will illustrate this remark on three examples:
 \begin{enumerate}\itemsep=0pt
 \item The splitting morphisms $\theta_{a\#b}\colon \mathcal{S}_A(\mathbf{\Sigma}_{a\#b}) \to \mathcal{S}_A(\mathbf{\Sigma})$ induces a Poisson map $\theta_{a\#b}^*\colon \mathcal{X}(\mathbf{\Sigma}) \to \mathcal{X}(\mathbf{\Sigma}_{a\#b})$. In particular, writing $\SL_2^{\rm D}:= \mathcal{X}(\mathbb{B})$,
 the coproduct $\Delta\colon \mathcal{S}_A(\mathbb{B}) \to \mathcal{S}_A(\mathbb{B})^{\otimes 2}$ induces a Poisson group law on~$\SL_2^{\rm D}$: we say that $\SL_2^{\rm D}$ is a~Poisson--Lie group (see Section~\ref{sec_relcharvar} for details).
 \item The comodule map $\Delta^L\colon \mathcal{S}_A(\mathbf{\Sigma}) \to \mathcal{O}_q[\SL_2]^{\otimes \mathcal{A}} \otimes \mathcal{S}_A(\mathbf{\Sigma})$ induces a Poisson action
\[ \nabla^L\colon \ \big(\SL_2^{\rm D}\big)^{\mathcal{A}} \times \mathcal{X}(\mathbf{\Sigma}) \to \mathcal{X}(\mathbf{\Sigma}).\]
 We say that $\mathcal{X}(\mathbf{\Sigma})$ is a $\big(\SL_2^{\rm D}\big)^{\mathcal{A}}$-Poisson variety.
 \item Let $i\colon \mathbf{m}_1 \hookrightarrow \mathbf{\Sigma}_g^*$ be the embedding of marked surfaces illustrated in Figure~\ref{fig_moment_map}.

\begin{figure}[!h]
\centering \includegraphics[width=7cm]{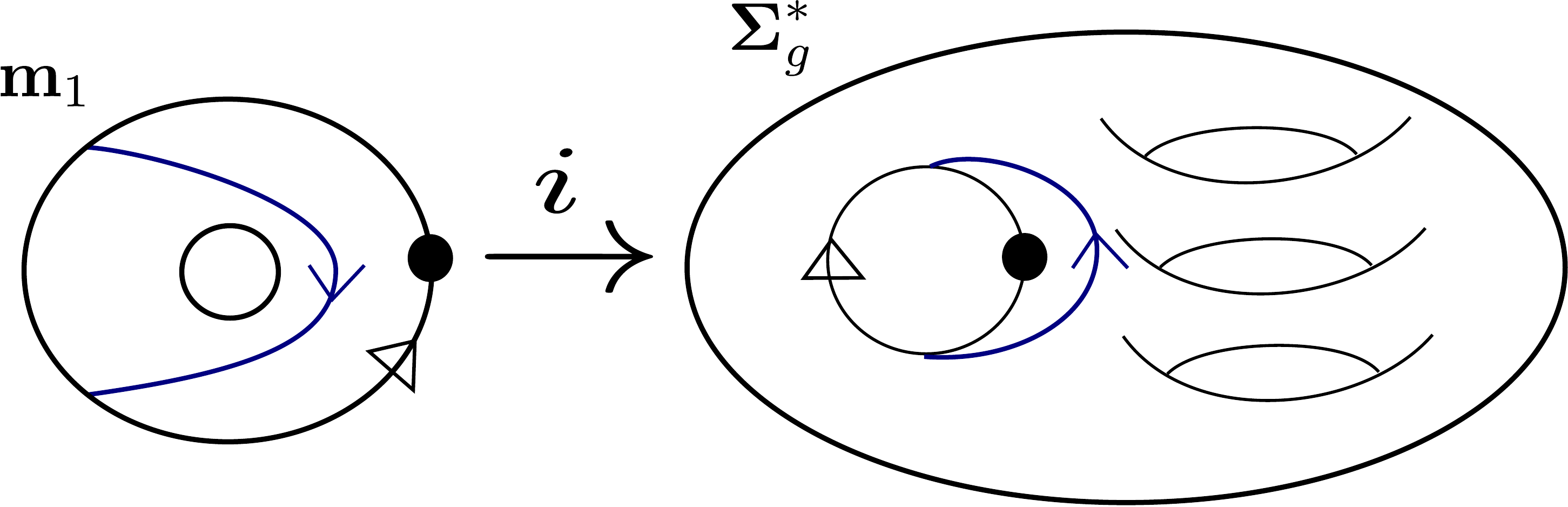}
\caption{The marked surfaces embedding defining the quantum and classical moment maps.}\label{fig_moment_map}
\end{figure}

 The induced algebra morphism
\[ \mu_q := i_*\colon \ \mathcal{S}_A(\mathbf{m}_1) \to \mathcal{S}_A(\mathbf{\Sigma}_g^*) \]
 is called the \textit{quantum moment map}. Writing $\SL_2^{\rm STS}:= \mathcal{X}(\mathbf{m}_1)$, the quantum moment map induces a Poisson map
\[ \mu \colon \ \mathcal{X}(\mathbf{\Sigma}_g^*) \to \SL_2^{\rm STS},\]
 named the \textit{classical moment map}. Note that $\mu$ is a morphism of $\SL_2^{\rm D}$-Poisson varieties.
 \end{enumerate}

For latter use, we now define an important toric action. Let $\iota\colon \mathbb{C}^* \hookrightarrow \SL_2$ be the diagonal embedding $\iota(z) := \left(\begin{smallmatrix} z & 0 \\ 0 & z \end{smallmatrix}\right)$.

\begin{Definition} The toric action of $(\mathbb{C}^*)^{\mathcal{A}}$ on $\mathcal{X}(\mathbf{\Sigma})$ is the action induced by the $(\SL_2)^{\mathcal{A}}$ action through the embedding $\iota^{\mathcal{A}}\colon (\mathbb{C}^*)^{\mathcal{A}} \hookrightarrow (\SL_2)^{\mathcal{A}}$. It reduces to a similar toric action on $\overline{\mathcal{X}}(\mathbf{\Sigma})$.
\end{Definition}

When $\mathbf{\Sigma}$ does not contain any unmarked component, then the Poisson varieties~$\mathcal{X}(\mathbf{\Sigma})$ and $\overline{\mathcal{X}}(\mathbf{\Sigma})$ are smooth so they can be seen as analytic manifolds as well. We can thus partition $\mathbf{X}(\mathbf{\Sigma})$ into its symplectic leaves.

\begin{Definition}\quad\samepage
\begin{enumerate}\itemsep=0pt
\item Let $X$ be a smooth Poisson variety. Consider the equivalence relation $\sim$ on $X$ by
 writing $x\sim y$ if there exists a finite sequence $x=p_0, p_1, \dots, p_k=y$ and functions $h_0, \dots, h_{k-1} \in \mathcal{O}[X]$ such that $p_{i+1}$ is obtained from~$p_i$ by deriving along the Hamiltonian flow of~$h_i$.
 The orbits for this relation are called the \textit{symplectic leaves}: they are the biggest connected smooth symplectic subvarieties of $X$.
\item When $X$ is either $\mathcal{X}(\mathbf{\Sigma})$ or $\overline{\mathcal{X}}(\mathbf{\Sigma})$, consider the equivalence relation $\sim$ where $x\sim y$ if there exists an element $t\in (\mathbb{C}^*)^{\mathcal{A}}$ such that $t\cdot x $ and $y$ belong to the same symplectic leaf. The orbits for this relation are called the \textit{equivariant symplectic leaves}.
\end{enumerate}
\end{Definition}

The classification of representations of (reduced) stated skein algebras at roots of unity is closely related to the computation of the equivariant symplectic leaves as detailed in~\cite{KojuSurvey} and briefly reviewed in Section~\ref{sec_PO}. In particular, Theorem~\ref{theorem1} will follow from the fact that $\overline{\mathcal{X}}(\mathbf{\Sigma}_g^*)$ contains a single equivariant symplectic leaf.

\subsection{Relative representation varieties in particular cases}\label{sec_relcharvar}

 The Poisson variety $\mathcal{X}(\mathbf{\Sigma})$ admits a geometric description named \textit{relative representation variety} first defined by Fock--Rosly in \cite{FockRosly} and studied independently in \cite{AlekseevKosmannMeinrenken, AlekseevMalkin_PoissonCharVar,GHJW_ModSpacesParBd,KojuTriangularCharVar}. The precise relation between relative representation varieties and $\mathcal{X}(\mathbf{\Sigma})$ is proved in \cite[Theorem~1.3]{KojuQuesneyClassicalShadows} (see also \cite[Theorem~4.7]{KojuPresentationSSkein} for an alternative proof). Let us first describe this moduli space in the particular cases of $\mathbb{B}$, $\mathbf{m}_1$, $\mathbb{D}_1^+$. Let us first introduce the classical and quantum $R$-matrices.
 Consider the matrices in $\mathfrak{sl}_2$:
\[ E = \begin{pmatrix} 0 & 0 \\ 1 & 0 \end{pmatrix}, \qquad F = \begin{pmatrix} 0 & 1 \\ 0 & 0 \end{pmatrix},\qquad H= \begin{pmatrix} 1 &0\\0 & -1 \end{pmatrix}.\]
 Consider these matrices as operators acting on a $2$-dimensional vector space $V$ with ordered basis $(v_+, v_-)$ (the standard representation of $\mathfrak{sl}_2$). Define an endomorphism $q^{\frac{H\otimes H}{2}}\in \End(V\otimes V)$ by $q^{\frac{H\otimes H}{2}}\cdot v_{\varepsilon_1} \otimes v_{\varepsilon_2} := A^{\varepsilon_1 \varepsilon_2} v_{\varepsilon_1} \otimes v_{\varepsilon_2}$. Then $\mathscr{R}$ is the matrix in the ordered basis $(v_+\otimes v_+, v_+\otimes v_-, v_-\otimes v_+, v_-\otimes v_-)$ of the operator
\[ \mathscr{R}= \tau \circ q^{\frac{H\otimes H}{2}} \exp_q\big( \big(q-q^{-1}\big)E \otimes F \big) = \tau \circ q^{\frac{H\otimes H}{2}} \circ \big( \mathds{1}_2 + \big(q-q^{-1}\big) E \otimes F \big).\]
Define the classical $r$-matrices:
\[ r^+:= \frac{1}{2} H\otimes H +2E\otimes F, \qquad r^{-} := \frac{1}{2} H\otimes H +2F\otimes E.\]
Then, writing $A^{1/2}=\exp(\hbar/2)$, one has
\begin{align}\label{eq_classicalr}
 & \mathscr{R}\equiv \tau ( \mathds{1}\otimes \mathds{1} + \hbar r^+) \ \big({\rm mod}~\hbar^2\big) \equiv ( \mathds{1}\otimes \mathds{1} + \hbar r^-)\tau \ \big({\rm mod}~\hbar^2\big), \\
 \label{eq_classicalr2}
&\mathscr{R}^{-1}\equiv ( \mathds{1}\otimes \mathds{1} - \hbar r^+)\tau \ \big({\rm mod}~\hbar^2\big) \equiv \tau ( \mathds{1}\otimes \mathds{1} - \hbar r^-) \ \big({\rm mod}~\hbar^2\big).
\end{align}

\textit{The case of the bigon:}
 The quantum group $\mathcal{O}_q[\SL_2]$ is generated by elements $a$, $b$, $c$, $d$ and, writing $N:= \left(\begin{smallmatrix} a & b \\ c & d \end{smallmatrix}\right)$, it is defined by the relations:
\[ N \odot N = \mathscr{R}^{-1} (N\odot N) \mathscr{R}, \qquad \det_q(N):= ad-q^{-1}bc = 1.\]
 Recall that, writing $\alpha_{ij}:= \adjustbox{valign=c}{\includegraphics[width=0.8cm]{Figure-03e}} $, the isomorphism $\mathcal{S}_A(\mathbb{B}) \cong \mathcal{O}_q[\SL_2]$ of Theorem~\ref{theorem_bigon} sends $\alpha_{++}$, $\alpha_{+-}$, $\alpha_{-+}$, $\alpha_{--}$ to $a$, $b$, $c$, $d$ respectively.

 Replacing $A$ by $\exp(\hbar/2)$ and developing using equations \eqref{eq_classicalr}, \eqref{eq_classicalr2}, we find
 \begin{align*}
 & N \odot N \equiv (1- \hbar r^+) \tau (N\odot N) \tau (1+\hbar r^+) \ \big({\rm mod}~\hbar^2\big) \\
 & \quad \Leftrightarrow \quad \tau (N\odot N) \tau -N\odot N \equiv \hbar \big( r^{+} (N\odot N) - (N\odot N) r^+ \big) \ \big({\rm mod}~\hbar^2\big).
 \end{align*}

 Therefore $\mathcal{X}(\mathbb{B})$ can be identified with the variety $\SL_2$ together with the Poisson bracket defined by
 \begin{equation}\label{eq_DrinfeldBracket}
 \{ N \otimes N \}^{\rm D} = r^+ (N \odot N) - (N\odot N) r^+.
 \end{equation}

We will denote by $\SL_2^{\rm D}$ the obtained Poisson--Lie group (the D stands for Drinfel'd who first defined it in~\cite{DrinfeldrMatrix}). We can rewrite equation \eqref{eq_DrinfeldBracket} as
\begin{gather*}
\{ a, b\}^{\rm D} = -ab \{a, c\}^{\rm D} = -ac \{ b, c \}^{\rm D} = 0,\\
\{ d, b \}^{\rm D} = db \{d, c \}^{\rm D} = dc \{ a, d\}^{\rm D}= -2bc.
\end{gather*}

 Consider the double Bruhat cells decomposition
\[\SL_2^{\rm D} = X_{00} \bigsqcup X_{01} \bigsqcup X_{10} \bigsqcup X_{11},\]
 where $\left(\begin{smallmatrix} a & b \\ c & d \end{smallmatrix}\right)$ is in $X_{11}$ if $bd\neq 0$, is in $X_{10}$ if $b=0$, $c\neq 0$, is in $X_{01}$ if $c=0$, $b\neq 0$ and is in $X_{00}$ if it is diagonal. The Weil group of $\SL_2$ is $W= \{w_0, w_1\}$ where $w_0$ is the class of the identity $\dot{w}_0:= \mathds{1}_2$ and $w_1$ is the class of $\dot{w}_1:= \left(\begin{smallmatrix} 0 & 1 \\ -1 & 0 \end{smallmatrix}\right)$. Denote by $B^+$ (resp.~$B^-$) the subgroup of $\SL_2$ of upper (resp. lower) triangular matrices. A simple computation shows that
 \[X_{ij} = B^+ \dot{w}_i B^+ \cap B^- \dot{w}_j B^-.\]

 \begin{Theorem}[{Hodges--Levasseur \cite[Theorem~B.2.1]{HodgesLevasseur_OqG}}] The equivariant symplectic leaves of $\SL_2^{\rm D}$ are the double Bruhat cells~$X_{ij}$.
 \end{Theorem}

\textit{The case of the once-punctured monogon:}
 Let $\alpha \subset \mathbf{m}_1$ be the unique corner arc with endpoint~$v$ and~$w$ such that $v<_{\mathfrak{o}_+} w$ and $\alpha_{ij}$ the stated arc with state $i$ on $v$ and $j$ on $w$ and $h(v)<h(w)$. Graphically, this means that $\alpha_{ij}:= \adjustbox{valign=c}{\includegraphics[width=1cm]{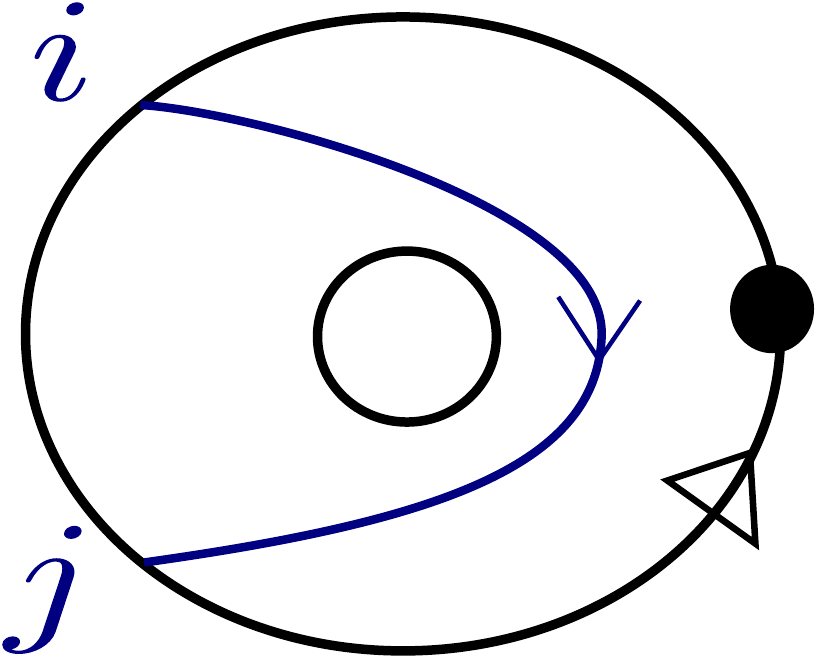}}$. Consider the $2\times 2$ matrix:
\[
N=\begin{pmatrix} a & b \\ c & d \end{pmatrix} := \begin{pmatrix} 0 & -A^{5/2} \\ A^{1/2} & 0 \end{pmatrix} \begin{pmatrix} \alpha_{++} & \alpha_{+-} \\ \alpha_{-+} & \alpha_{--} \end{pmatrix} =\begin{pmatrix} -A^{5/2}\alpha_{-+} & -A^{5/2}\alpha_{--} \\ A^{1/2}\alpha_{++} & A^{1/2}\alpha_{+-} \end{pmatrix}.
\]
 Then $\mathcal{S}_A(\mathbf{m}_1)$ is generated by $a$, $b$, $c$, $d$ with relations
\[ (\mathds{1}_1 \odot N) \mathscr{R}^{-1}(\mathds{1}_1 \odot N) \mathscr{R} = \mathscr{R} (\mathds{1}_1 \odot N) \mathscr{R}^{-1}(\mathds{1}_1 \odot N), \qquad \det_{q^2}(N)= 1.
\]
 The algebra $\mathcal{S}_A(\mathbf{m}_1)$ is called the \textit{braided quantum group} and were first considered by Majid (see~\cite{Majid_QGroups}).
 \par
 Write $\SL_2^{\rm STS}:= \mathcal{X}(\mathbf{m}_1)$: it is the variety $\SL_2$ equipped with the so-called Semenov-Tian-Shansky Poisson bracket obtained by
 replacing $A$ by $\exp(\hbar/2)$ and developing using equations \eqref{eq_classicalr}, \eqref{eq_classicalr2} as before. We find
\begin{gather*}
 \{N\otimes N\}^{\rm STS} = - (\mathds{1}_1 \odot N)r^+ (N\odot \mathds{1}_2) +\tau (N\odot N) \tau r^+ - r^- (N\odot N) \\
 \hphantom{\{N\otimes N\}^{\rm STS} =}{} + (N\odot \mathds{1}_2)r^- (\mathds{1}_2 \odot N).
\end{gather*}
 We can develop these equations to find
 \begin{gather*}
 \begin{split}
& \{c,d\}^{\rm STS} = 2ac \{d,b\}^{\rm STS}= 2ab \{d,a\}^{\rm STS} =0, \\
& \{b,a\}^{\rm STS} = 2ab \{a,c\}^{\rm STS} = 2ac \{c,b\}^{\rm STS}= 2a(a-d).
\end{split}
\end{gather*}
 Note that, unlike $\SL_2^{\rm D}$, $\SL_2^{\rm STS}$ is no longer a~Poisson--Lie group, i.e., the composition law of $\SL_2$ is not Poisson. Consider the (simple) Bruhat cells decomposition:
\[ \SL_2 = \SL_2^0 \bigsqcup \SL_2^1, \]
 where $\SL_2^0=B^-B^+$ is the subset of matrices $\left(\begin{smallmatrix} a & b \\ c & d \end{smallmatrix}\right)\in \SL_2$ with $a\neq 0$ and $\SL_2^1$ is its complementary. Extending the original work of Semenov-Tian-Shansky, Alekseev--Malkin proved that the symplectic leaves of $\SL_2^{\rm STS}$ are the intersections of the conjugacy classes in $\SL_2$ with the so-called dressing orbits of $\SL_2$, which are $(1)$ the big cell $\SL_2^0$ and $(2)$ the subsets
\[ C_b := \left\{ \begin{pmatrix} 0 & b \\ -b^{-1} & d \end{pmatrix}, \, d\in \mathbb{C} \right\}, \qquad b\in \mathbb{C}^*.\]
 We thus obtain

 \begin{Theorem}[{Alekseev--Malkin \cite[Section~4]{AlekseevMalkin_PoissonCharVar}}]\label{theorem_STS}
 The symplectic leaves of $\SL_2^{\rm STS}$ are
 \begin{enumerate}\itemsep=0pt
 \item[$1)$] the intersections $\SL_2^0 \cap C$ where $C$ is a conjugacy class of $\SL_2$,
 \item[$2)$] the singletons $\{g\}$ for $g\in \SL_2^1$.
 \end{enumerate}
 \end{Theorem}

 Note that the toric action of $\mathbb{C}^*$ on $\SL_2^{\rm STS}$ is given by the formula
 \[ z \cdot \begin{pmatrix} a & b \\ c & d \end{pmatrix} = \begin{pmatrix} z^{-1} & 0 \\ 0 & z \end{pmatrix} \begin{pmatrix} a & b \\ c & d \end{pmatrix} \begin{pmatrix} z & 0 \\ 0 & z^{-1} \end{pmatrix} = \begin{pmatrix} a & z^2b \\ z^{-2}c & d \end{pmatrix}.\]
 In particular, $\mathbb{C}^*$ acts transitively on the dressing orbits $C_b$, $b\in \mathbb{C}^*$ which lye in~$\SL_2^1$, so we deduce

 \begin{Corollary}\label{coro_STS}
 The equivariant symplectic leaves of $\SL_2^{\rm STS}$ are the intersections $\SL_2^i \cap C$ where~$C$ is a conjugacy class and $i=0,1$.
 \end{Corollary}

\textit{The case of $\mathbb{D}_1^+$:}
The algebra $\mathbb{D}_q^+(\SL_2):= \mathcal{S}_A(\mathbb{D}_1^+)$ is called the \textit{twisted Heisenberg double} and is presented as follows. Let $\alpha$, $\beta$ be the two arcs in $\mathbb{D}_1^+$ of Figure~\ref{fig_D1+}.

\begin{figure}[!h]\centering\includegraphics[width=2.5cm]{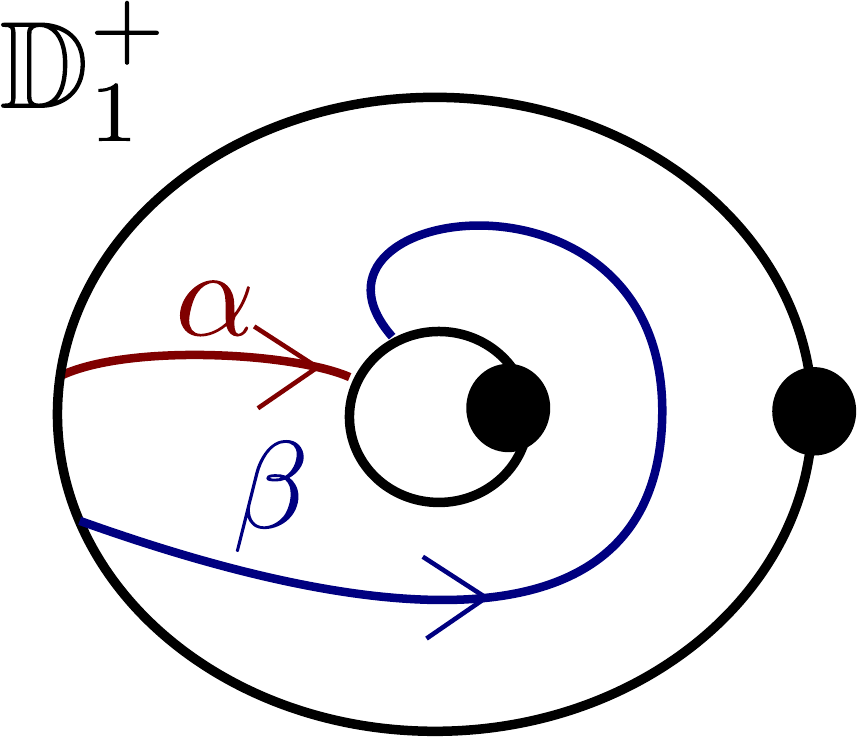}
\caption{Two arcs in $\mathbb{D}_1^+$.}\label{fig_D1+}
\end{figure}

Consider the matrices $N(\alpha)= \left(\begin{smallmatrix} \alpha_{++} & \alpha_{+-} \\ \alpha_{-+} & \alpha_{--} \end{smallmatrix}\right)$, $N(\beta)= \left(\begin{smallmatrix} \beta_{++} & \beta_{+-} \\ \beta_{-+} & \beta_{--} \end{smallmatrix}\right)$.

Then $D_q^+(\SL_2)$ is generated by the $\alpha_{ij}$, $\beta_{ij}$, for $i,j=\pm$ , modulo the relations $\det_q(N(\alpha))=\det_q(N(\beta))=1$ and
\[ N(a)\odot N(a) = \mathscr{R}^{-1} (N(a)\odot N(a))\mathscr{R}, \qquad N(\alpha)\odot N(\beta) = \mathscr{R} (N(\alpha)\odot N(\beta))\mathscr{R}, \]
 for $a=\alpha, \beta$.
 So $D_+(\SL_2):= \mathcal{X}(\mathbb{D}_1^+)$ is the variety $\SL_2\times \SL_2$ with the Poisson bracket described by
\begin{gather*}
 \{N(\alpha) \otimes N(\beta) \}^+ = r_+ (N(\alpha)\odot N(\beta)) + (N(\alpha)\odot N(\beta))r_+, \\
 \{N(a)\otimes N(a)\}^+ = r^+ (N(a) \odot N(a)) - (N(a)\odot N(a))r^+,
\end{gather*}
for $a=\alpha, \beta$.
The Poisson variety $D_+(\SL_2)$ was studied by Alekseev--Malkin in \cite{AlekseevMalkin_PoissonLie} inspired by the work of Semenov-Tian-Shansky. More precisely, they consider the bracket $\{\cdot, \cdot\}^{\rm AM}:= - \{\cdot, \cdot\}^+$ for which, using the notations $r:= r^+$ and $r^*:= -r^-$, one has (compare with \cite[equation~(80)]{AlekseevMalkin_PoissonLie}):
\[ \{N(\alpha)\odot N(\beta) \}^{\rm AM}= - \big( r ( N(\alpha) \odot N(\beta)) + (N(\alpha)\odot N(\beta))r^* \big).\]

Consider the partition
\[ D_+(\SL_2)= D_{00}\bigsqcup D_{01} \bigsqcup D_{10} \bigsqcup D_{11},\]
 where
\[
 D_{ij} = \big\{(g_1, g_2) \,|\, g_2^{-1}g_1 \in \SL_2^i, g_2g_1^{-1} \in \SL_2^j \big\}.
\]

 \begin{Theorem}[{Alekseev--Malkin \cite[Theorem~2]{AlekseevMalkin_PoissonLie}}]\label{theorem_leaves_D1} The symplectic leaves of $D_+(\SL_2)$ are the~$D_{ij}$.
 \end{Theorem}

Note that the symplectic leaves are preserved by the toric $(\mathbb{C}^*)^2$ action.

\subsection{Classical fusion operation}

We now describe the classical equivalent of Theorem~\ref{theorem_fusion}. Let $\mathbf{r}_{\hbar}\colon \mathcal{O}_{\hbar}[\SL_2]^{\otimes 2}\to \mathbb{C}[[\hbar]]$ be the co-R-matrix for parameter $q:=\exp(\hbar)$ and note that
\[ \mathbf{r}_{\hbar} \equiv \epsilon \otimes \epsilon + \hbar r^+ \ \big({\rm mod}~\hbar^2\big).\]
 In this formula, we see $r^+ \in \mathfrak{sl}_2\otimes \mathfrak{sl}_2$ as an element of the Zariski tangent space of $\SL_2\times \SL_2$ at the neutral element $(\mathds{1}_2, \mathds{1}_2)$, i.e., $r^+ \in \operatorname{Der}\big(\mathcal{O}[\SL_2]^{\otimes 2}, \mathbb{C}_{\varepsilon \otimes \varepsilon}\big)$ is a derivation valued in~$\mathbb{C}$ with $\mathcal{O}[\SL_2]^{\otimes 2}$-module structure induced by $\epsilon \otimes \epsilon$.

Let $G$ be an affine algebraic Poisson--Lie group with Poisson structure given by a classical $r$-matrix $r^+ \in \mathfrak{g}\otimes \mathfrak{g}$ (i.e., with Poisson structure given by the cocomutator $\delta\colon \mathfrak{g}\to \mathfrak{g}\otimes \mathfrak{g}$, $\delta(X)=[X, r^+]$). A~$G$-Poisson affine variety is a complex affine variety~$X$ with an algebraic Poisson action $G\times X \to X$.

\begin{Definition} Let $X$ be a $G^2$-Poisson affine variety and denote by $\Delta_{G\times G}\colon \mathcal{O}[X]\to \mathcal{O}[G]\otimes \mathcal{O}[X]$ it comodule map. Wite $\Delta^1:= (\id\otimes \epsilon\times \id)\circ \Delta_{G\times G}$ and $\Delta^2:= (\epsilon \otimes \id \otimes \id)$.
 The \textit{fusion} of~$X$ is the $G$-Poisson affine variety~$X^{\circledast}$ define by
 \begin{enumerate}\itemsep=0pt
 \item As a $\mathbb{C}$-algebra, $\mathcal{O}[X^{\circledast}]=\mathcal{O}[X]$.
 \item For $x\in \mathcal{O}[X]$ and $i=1,2$, write $\Delta^i (x) = \sum x_{(i)}'\otimes x_{(i)}''$. The Poisson bracket is defined by
\[ \{ x, y \}^{\circledast} := \{x,y\} + \sum r^+ (y'_{(2)} \otimes x'_{(1)}) x_{(1)}''y_{(2)}'' - \sum r^+ (x'_{(2)}\otimes y'_{(1)})y''_{(1)}x''_{(2)}.\]
 \item The $G$ action is given by the comodule map $\Delta_G:= (\mu_G \otimes \id) \circ \Delta_{G\times G} $.
 \end{enumerate}
 \end{Definition}

 In the particular case where $X$ is smooth, consider $X$ as a smooth manifold and denote by~$\pi_X $ the Poisson bivector field defining the Poisson structure (i.e., $\{f,g\}(x)=\langle D_xf \otimes D_xg, \pi_{X,x}\rangle $). Let $r^- := \sigma(r^+)$, where $\sigma(x\otimes y)=y\otimes x$. Let $a_{G\times G}\colon \mathfrak{g}\otimes \mathfrak{g} \to \Gamma(X, T_X)$ the infinitesimal action induced by the action of~$G^2$ on~$X$. Then the fusion $X^{\circledast}$ is the manifold $X$ with the Poisson bivector field
 \[ \pi_{X^{\circledast}} = \pi_X + a_{G\times G} (r^- - r^+).\]
 This is using this formula that the concept of fusion was introduced in the work of Alekseev--Malkin \cite{AlekseevMalkin_PoissonCharVar}. As we shall see, for a connected marked surface with non-trivial marking, then~$\mathcal{X}(\mathbf{\Sigma})$ is smooth. Recall the coaction $\Delta^L\colon \mathcal{S}_{+1}(\mathbf{\Sigma}) \to \mathcal{O}[\SL_2]^{\otimes \mathcal{A}} \otimes \mathcal{S}_{+1}(\mathbf{\Sigma})$ induced by gluing some bigons to the boundary arcs. This endow $\mathcal{X}(\mathbf{\Sigma})$ with a structure of $\big(\SL_2^{\rm D}\big)^{\mathcal{A}}$ Poisson variety. In particular, by choosing two boundary arcs~$a$, $b$, we get a structure of $\big(\SL_2^{\rm D}\big)^2$-Poisson variety on~$\mathcal{X}(\mathbf{\Sigma})$.
 As a consequence of Theorem \ref{theorem_fusion}, we get

 \begin{Theorem}\label{coro_fusion}
 If $\mathbf{\Sigma}_{a\circledast b}$ is obtained from $\mathbf{\Sigma}$ by fusioning $a$ and $b$, then $\mathcal{X}(\mathbf{\Sigma}_{a\circledast b})\cong \mathcal{X}(\mathbf{\Sigma})^{a\circledast b}$.
 \end{Theorem}

 \begin{proof}Recall from Theorem \ref{theorem_fusion} the isomorphism of vector spaces $\Psi_{a\circledast b}\colon \mathcal{S}_{\hbar}(\mathbf{\Sigma}) \cong \mathcal{S}_{\hbar}(\mathbf{\Sigma}_{a \circledast b})$ so that the pull-back $\mu_{\circledast}$ of the product in $ \mathcal{S}_{\hbar}(\mathbf{\Sigma}_{a \circledast b})$ is the fusion, in the sense of Definition \ref{def_fusion_quantique}, of the product $\mu$ of $ \mathcal{S}_{\hbar}(\mathbf{\Sigma})$. Denote by $\star_{\circledast}$ and $\star$ the products in $\mathcal{S}_{+1}(\mathbf{\Sigma})\otimes \mathbb{C}[[\hbar]]$ corresponding to $\mu_{\circledast}$ and $\mu$ respectively.
 For $x,y\in \mathcal{S}_{+1}(\mathbf{\Sigma})$, we have
 \[ x \star_{\circledast} y = \sum \mathbf{r}_{\hbar}(y'_{(2)} \otimes x'_{(1)} ) x''_{(1)} \star y''_{(2)}, \qquad \text{and} \qquad y\star_{\circledast} x = \sum \mathbf{r}_{\hbar}(x'_{(2)}\otimes y'_{(1)})y''_{(1)} \star x''_{(2)}.\]
 Using the equalities $ \mathbf{r}_{\hbar}(x\otimes y) \equiv \epsilon(x) \epsilon(y) + \hbar r^+(x\otimes y) \ \big({\rm mod}~\hbar^2\big)$ and $x\star y - y\star x = \hbar \{x,y\} \ \big({\rm mod}~\hbar^2\big)$, we find
\begin{gather*}
x\star_{\circledast} y - y\star_{\circledast} x \equiv \hbar\Big( \{x,y\} + \sum r^+ (y'_{(2)} \otimes x'_{(1)}) x_{(1)}''y_{(2)}'' \\
\hphantom{x\star_{\circledast} y - y\star_{\circledast} x \equiv \hbar\Big(}{}
- \sum r^+ (x'_{(2)}\otimes y'_{(1)})y''_{(1)}x''_{(2)}\Big) \ \big({\rm mod}~\hbar^2\big).\tag*{\qed}
\end{gather*}\renewcommand{\qed}{}
 \end{proof}

The main motivation for the authors of \cite{AlekseevMalkin_PoissonCharVar} to introduce the notion of fusion was to get the following decomposition (which is \cite[Theorem 2]{AlekseevMalkin_PoissonCharVar}):
 \begin{equation}\label{eq_poisson_structure}
 \mathcal{X}(\mathbf{\Sigma}_{g}^*) \cong \mathcal{X}(\mathbf{\Sigma}_{1}^*) ^{\circledast g}, \qquad \mathcal{X}(\mathbf{\Sigma}_1^*) \cong (D_q^+(\SL_2))_{1\circledast 2}.
 \end{equation}
 This reduces the study of $ \mathcal{X}(\mathbf{\Sigma}_{g}^*) $ to the study of $D_+(\SL_2)$ and $\SL_2^{\rm STS}$.

\subsection[The case of Sigma\_g\textasciicircum{}*: the representation variety]{The case of $\boldsymbol{\mathbf{\Sigma}_g^*}$: the representation variety}

 Let $v\in a$ be a based point in the single boundary arc of $\mathbf{\Sigma}_g^*$. As an affine variety, we set
\[\mathcal{X}_{\SL_2}(\mathbf{\Sigma}_g^*) := \Hom(\pi_1(\Sigma_{g,1}, v), \SL_2).\]
 We call it the \textit{representation variety}. Fix a set of generators $\mathbb{G}:=\{ \alpha_1, \beta_1, \dots, \alpha_g, \beta_g\}$ of the free group $\pi_1(\Sigma_{g,1}, v)$ such that the intersection form $\langle \cdot, \cdot \rangle $ satisfies $\langle \alpha_i, \beta_i\rangle = \delta_{i,j}$, $\langle \alpha_i, \alpha_j\rangle =\langle \beta_i, \beta_j \rangle \allowbreak =0$ (i.e., $\alpha_i$, $\beta_j$ are meridian and longitudes as in Figure \ref{fig_arcs_Alekseev}). So we have an isomorphism:
\[ \mathcal{X}_{\SL_2}(\mathbf{\Sigma}_g^*) \cong (\SL_2)^{2g}, \qquad \rho \mapsto (\rho(\alpha_1), \rho(\beta_1), \dots, \rho(\alpha_g), \rho(\beta_g) ).\]
 Let $\gamma_{\partial}:= [\alpha_1, \beta_1] \cdots [\alpha_g, \beta_g]$ be the loop encircling the boundary component. The \textit{moment map} is the map
\[
\mu \colon \ \mathcal{X}_{\SL_2}(\mathbf{\Sigma}_g^*) \to \SL_2^{\rm STS}, \qquad \rho \mapsto \rho(\gamma_{\partial}).
\]
The Poisson structure of $\mathcal{X}_{\SL_2} (\mathbf{\Sigma}_g^*) $ is described explicitly in \cite{AlekseevKosmannMeinrenken,FockRosly, KojuTriangularCharVar} and is characterized by equation~\eqref{eq_poisson_structure} that we take here as a definition. Let us describe the Poisson isomorphism \[
\Psi \colon \ \mathcal{O}[\mathcal{X}_{\SL_2}(\mathbf{\Sigma}_g^*)] \xrightarrow{\cong} \mathcal{S}_{+1}(\mathbf{\Sigma}_g^*)\] explicitly. For each $\gamma \in \mathbb{G}$, denote by $X^{\gamma}_{ij}$ the regular function in $\mathcal{O}[\mathcal{X}_{\SL_2}(\mathbf{\Sigma}_g^*)] $ sending a representation $\rho\colon \pi_1(\Sigma_{g,1})\to \SL_2$ to the $(i,j)$-th matrix coefficient of $\rho(\gamma)$. Therefore $\mathcal{O}[\mathcal{X}_{\SL_2}(\mathbf{\Sigma}_g^*)]$ is the quotient of the polynomial ring $\mathbb{C}\big[X_{ij}^{\gamma}, \gamma \in \mathbb{G}, i,j=\pm\big]$ by the ideal generated by the polynomials $X_{++}^{\gamma}X_{--}^{\gamma} - X_{+-}^{\gamma}X_{-+}^{\gamma} -1$, for $\gamma \in \mathbb{G}$.
\par Represent each $\gamma \in \mathbb{G}$ by an embedded arc in $\Sigma_{g,1}$ with no crossing and such that two such arcs do not intersect. Let~$\gamma_{ij} \in \mathcal{S}_A(\mathbf{\Sigma}_g^*)$ be the class of the arc $\gamma$ with blackboard framing oriented from its endpoint~$v$ to its endpoint~$w$ such that $h(v)<h(w)$ with the state sending~$v$ and~$w$ to~$i$ and~$j$ respectively.

\begin{Theorem}[{\cite[Theorem~1.3]{KojuQuesneyClassicalShadows}, \cite[Theorem~4.7]{KojuPresentationSSkein}}]\label{theorem_skein_charvar} We have an isomorphism of $\SL_2^{\rm D}$-Poisson varieties
\[\Psi \colon \ \mathcal{O}[\mathcal{X}_{\SL_2}(\mathbf{\Sigma}_g^*)] \xrightarrow{\cong} \mathcal{S}_{+1}(\mathbf{\Sigma}_g^*),\]
which intertwines the moment maps and which is characterized by the formula
\[ \Psi \begin{pmatrix} X_{++}^{\gamma} & X_{+-}^{\gamma} \\ X_{-+}^{\gamma} & X_{--}^{\gamma} \end{pmatrix} =
\begin{pmatrix} 0 & -1 \\ 1 & 0 \end{pmatrix}
\begin{pmatrix} \gamma_{++} & \gamma_{+-} \\ \gamma_{-+} & \gamma_{--} \end{pmatrix}
= \begin{pmatrix} - \gamma_{-+} & - \gamma_{--} \\ \gamma_{++} & \gamma_{+-} \end{pmatrix},
\]
for $\gamma \in \mathbb{G}$.
\end{Theorem}

 A key concept introduced by Alekseev--Kosmann-Schwarzbach--Meinrenken in \cite{AlekseevKosmannMeinrenken} is
 \begin{Definition}
 A smooth $G$-Poisson variety $X$ is \textit{non-degenerate} if the map
\[ \big(a_G, \pi^{\#}_X\big) \colon \ \mathfrak{g}\oplus T_X^{*} \to T_X\]
 is surjective, where $a_G \colon \mathfrak{g} \to T_X$ is the infinitesimal~$G$ action and $\pi_X^{\#}$ is the map induced from the Poisson bivector field $\pi_X$.
 \end{Definition}

 \begin{Lemma}[{Ganev--Jordan--Safronov \cite[Proposition~2.13]{GanevJordanSafranov_FrobeniusMorphism} following \cite[Section~10]{AlekseevKosmannMeinrenken}}]\label{lemma_XXX}
 Let~$X$ a~smooth $G^2\times H$-Poisson variety and $X^{\circledast}$ is the $G\times H$-Poisson variety obtained by fusion. If~$X$ is non degenerate then $X^{\circledast}$ is non-degenerate.
 \end{Lemma}

It follows from Theorem~\ref{theorem_leaves_D1} that $\mathbb{D}_+(\SL_2)$ is non-degenerate. So Lemma~\ref{lemma_XXX} and equation~\eqref{eq_poisson_structure} imply that $\mathcal{X}_{\SL_2}(\mathbf{\Sigma}_g^*)$ is non-degenerate as well. This property, together with the explicit description of the symplectic leaves of $\SL_2^{\rm STS}$ in Theorem~\ref{theorem_STS} easily imply

\begin{Theorem}[{Ganev--Jordan--Safronov \cite[Theorem~2.14]{GanevJordanSafranov_FrobeniusMorphism}}]\label{theorem_GJS_geom} The symplectic leaves of \linebreak $\mathcal{X}_{\SL_2}(\mathbf{\Sigma}_g^*)$ are the pull-back by $\mu$ of the dressing orbits, i.e., are
\begin{enumerate}\itemsep=0pt
\item[$1)$] the open dense leaf $\mu^{-1}\big(\SL_2^0\big)$,
\item[$2)$] the leaves $\mu^{-1}(C_b)$, $b\in \mathbb{C}^*$.
\end{enumerate}
\end{Theorem}

Since the moment map is equivariant for the toric $\mathbb{C}^*$ action, the fact that $\mathbb{C}^*$ acts transitively on the $C_b$ implies

\begin{Corollary}\label{coro_leaves}
The equivariant symplectic leaves of $\mathcal{X}(\mathbf{\Sigma}_g^*)$ are the two leaves $\mu^{-1}\big(\SL_2^0\big)$ and $\mu^{-1}\big(\SL_2^1\big)=\overline{\mathcal{X}}(\mathbf{\Sigma}_g^*)$.
\end{Corollary}

So Corollary~\ref{coro_leaves} shows that $\overline{\mathcal{X}}(\mathbf{\Sigma}_g^*)$ has a single equivariant symplectic leaf. This is the key fact that permits to prove Theorem \ref{theorem1} and permits to define mapping class group representations.

\section{Algebraic study}\label{sec_algebraic}

We now review two important concepts, the Azumaya locus and the theory of Poisson orders, which were first introduced by De~Concini--Kac~\cite{DeConciniKacRepQGroups} for the study of the representations of quantum enveloping algebras at roots of unity and further developed by De~Concini--Lybashen\-ko~\cite{DeConciniLyubashenko_OqG} in the study of $\mathcal{O}_qG$ and by various authors including those of~\cite{BrownGordon_ramificationcenters,BrownGoodearl, BrownGordon_PO, FrohmanKaniaLe_UnicityRep}.

\subsection{Azumaya loci}\label{sec_AL}

Let $\mathcal{A}$ be a $\mathbb{C}$-algebra such that:
\begin{enumerate}\itemsep=0pt
\item[(i)] $\mathcal{A}$ is affine (i.e., finitely generated),
\item[(ii)] $\mathcal{A}$ is prime,
\item[(iii)] $\mathcal{A}$ has finite rank over its center $Z$.
\end{enumerate}
By Theorem~\ref{theorem_basic_ppty}, stated skein algebras at roots of unity and their reduced versions satisfy these properties.
Write $S:= Z\setminus \{0\}$. A~theorem of Posner--Formanek \cite[Theorem~I.13.3]{BrownGoodearl} shows that the localization $S^{-1}\mathcal{A}$ is a central simple algebra with center $K=S^{-1}Z$, so is a~matrix algebra is some algebraic closure~$\overline{K}$ of~$K$, i.e., $\mathcal{A}\otimes_Z \overline{K} \cong \Mat_D\big(\overline{K}\big)$. In particular, this implies that the rank $r=D^2$ is a perfect square and justifies the definition of the PI-degree~$D$ of~$\mathcal{A}$.

Write $\mathcal{X}:= \operatorname{Specm}(Z)$ and for $x\in \mathcal{X}$ corresponding to a maximal ideal $\mathfrak{m}_x \subset Z$, consider the finite-dimensional algebra
\[\mathcal{A}_x:= \quotient{\mathcal{A}}{\mathfrak{m}_x \mathcal{A}}.\]

 \begin{Definition} The \textit{Azumaya locus} of $\mathcal{A}$ is the subset
\[ \mathcal{AL}:= \{ x \in \mathcal{X} \,|\, \mathcal{A}_x \cong \Mat_D(\mathbb{C}) \}, \]
 where $\Mat_D(\mathbb{C})$ is the algebra of $D\times D$ matrices. The algebra $\mathcal{A}$ is said \textit{Azumaya} if $\mathcal{AL}= \mathcal{X}$.
 \end{Definition}

In particular, for an Azumaya algebra, the set of isomorphism classes of irreducible representations is in $1$-to-$1$ correspondence with the characters over the center~$Z$.

 \begin{Remark}\label{remark_AzumayaLocus}
 An irreducible representation $\rho\colon \mathcal{A}\to \End(V)$ sends central elements to scalar operators so induces a point $x\in \mathcal{X}$. If $x\in \mathcal{AL}$, then $V$ is $D$-dimensional. By a theorem of Posner, if~$x$ does not belong to the Azumaya locus, then $\mathcal{A}_x$ has PI-degree strictly smaller than $\mathcal{A}$, therefore any irreducible representation $\rho\colon \mathcal{A}\to \End(V)$ inducing $x$ has dimension $\dim(V)<D$. So the Azumaya locus admits the following alternative definition:
\[\mathcal{AL}=\{ x \in \mathcal{X} \,| \,x \text{is induced by an irrep of maximal dimension }D\}.\]
 \end{Remark}

 \begin{Definition}
 Let $\mathcal{A}$ as before.
 \begin{enumerate}\itemsep=0pt
 \item The \textit{reduced trace} is the composition
\[ \operatorname{tr} \colon \ \mathcal{A} \hookrightarrow \mathcal{A}\otimes_Z \overline{K} \cong M_D(\overline{K}) \xrightarrow{\operatorname{tr}_D} \overline{K}, \]
 where $\operatorname{tr}_D ((a_{ij})_{i,j}):= \frac{1}{D}\sum_i a_{ii}$ (so that $\operatorname{tr}(1)=1$).
By \cite[Theorem~10.1]{Reiner03}, $\operatorname{tr}$ takes values in~$Z$.
 \item
 The \textit{discriminant ideal} of $\mathcal{A}$ is the ideal $\mathcal{D} \subset Z$ generated by the elements
\[ \det(\operatorname{tr}(x_ix_j)) \in Z, \qquad (x_1, \dots, x_{D^2}) \in (\mathcal{A})^{D^2}.\]
\end{enumerate}
 \end{Definition}

\begin{Theorem}[{Brown--Milen \cite[Theorem~1.2]{Brown_AL_discriminant}}]\label{theorem_UnicityRep} If $\mathcal{A}$ satisfies $(i)$, $(ii)$ and $(iii)$ then
\[ \mathcal{AL}= \operatorname{Specm}(Z) \setminus V(\mathcal{D}).\]
 In particular, the Azumaya locus is an open dense set.
 \end{Theorem}

 The fact that the Azumaya locus is open dense seems to be well-known to the experts since a long time (see, e.g., \cite{BrownGoodearl, FrohmanKaniaLe_UnicityRep}) though the author was not able to find to whom attribute this folklore result (which is essentially a generalization of De~Concini--Kac pioneered work in \cite{DeConciniKacRepQGroups}). The reduced trace for skein algebra of unmarked surfaces was computed in~\cite{FrohmanKaniaLe_DimSkein} though the discriminant and the associated Azumaya loci are still unknown (in genus~$g\geq 2$).

Recall from Section~\ref{sec_center} the finite branched coverings $\pi\colon \widehat{\mathcal{X}}(\mathbf{\Sigma}) \to \mathcal{X}(\mathbf{\Sigma})$ and $\overline{\pi}\colon \widehat{\overline{\mathcal{X}}}(\mathbf{\Sigma}) \to \overline{\mathcal{X}}(\mathbf{\Sigma})$ induced by the Chebyshev--Frobenius morphisms.

 \begin{Definition} The \textit{fully Azumaya locus} of $\mathcal{S}_A(\mathbf{\Sigma})$ is the subset $\mathcal{FAL}\subset \mathcal{X}(\mathbf{\Sigma})$ of elements $x$ such that every point of the fiber $\pi^{-1}(x)$ belongs to the Azumaya locus of $\mathcal{S}_A(\mathbf{\Sigma})$.
 \par Similarly, the \textit{fully Azumaya locus} of $\overline{\mathcal{S}}_A(\mathbf{\Sigma})$ is the subset $\mathcal{FAL}\subset \overline{\mathcal{X}}(\mathbf{\Sigma})$ of elements $x$ such that every point of the fiber $\overline{\pi}^{-1}(x)$ belongs to the Azumaya locus of $\overline{\mathcal{S}}_A(\mathbf{\Sigma})$.
 \end{Definition}

The fully Azumaya locus was introduced by Brown and Gordon in~\cite{BrownGordon_ramificationcenters} in the case of the bigon. Since the projection map $\pi\colon \widehat{\mathcal{X}}(\mathbf{\Sigma}) \to \mathcal{X}(\mathbf{\Sigma})$ is finite, and since finite morphisms send closed sets to closed sets \cite[Example~2.35(b)]{Hart}, Corollary~\ref{theorem_UnicityRep} implies that the fully Azumaya loci are open dense subsets. The key theorem to work with the fully Azumaya locus is

\begin{Theorem}[{Brown--Gordon \cite[Corollary~2.7]{BrownGordon_ramificationcenters}}]
Let $\mathcal{A}$ be an affine prime $\mathbb{C}$-algebra finitely generated over its center $\mathcal{Z}$ and denote by $D$ its PI-degree. Let $R\subset \mathcal{Z}$ be a subalgebra such that~$\mathcal{Z}$ if finitely generated as a $R$-module. Let $M\in \mathcal{AL}(\mathcal{A})$ and $\mathfrak{m}:=M\cap R$. Then
\[ \quotient{\mathcal{A}}{\mathfrak{m} \mathcal{A}} \cong \Mat_D \left(\quotient{\mathcal{Z}}{\mathfrak{m}\mathcal{Z}}\right).\]
\end{Theorem}

 Recall for $x\in \mathcal{X}(\mathbf{\Sigma})$ the notation $\mathcal{S}_A(\mathbf{\Sigma})_{x}:= \quotient{ \mathcal{S}_A(\mathbf{\Sigma})}{\operatorname{Ch}_A(\mathfrak{m}_x) \mathcal{S}_A(\mathbf{\Sigma})}$. Let $\mathcal{Z}$ be the center of $\mathcal{S}_A(\mathbf{\Sigma})$ and write
\[ Z(x):= \quotient{\mathcal{Z}}{\operatorname{Ch}_A(\mathfrak{m}_x)\mathcal{Z}}.\]
For $x\in \overline{\mathcal{X}}(\mathbf{\Sigma})$, we define $\overline{\mathcal{S}}_A(\mathbf{\Sigma})_{x}$ and $\overline{Z}(x)$ similarly.

\begin{Corollary}\label{theorem_BG_Ramifications}
If $x$ belongs to the fully Azumaya locus of $\mathcal{S}_A(\mathbf{\Sigma})$ and $D$ denotes its PI-degree,
then
\[
\mathcal{S}_A(\mathbf{\Sigma})_{x}\cong \Mat_D(Z(x)).\]

Similarly, if $x$ belongs to the fully Azumaya locus of $\overline{\mathcal{S}}_A(\mathbf{\Sigma})$ and $D'$ denotes its PI-degree, then \[\overline{\mathcal{S}}_A(\mathbf{\Sigma})_{x}\cong \Mat_{D'}\big(\overline{Z}(x)\big).\]
\end{Corollary}

Note that the algebras $\overline{Z}(x)$ are easy to compute explicitly using Theorem~\ref{theorem_center}. For instance, in the case of $\mathbf{\Sigma}_g^*$, the center of $\overline{\mathcal{S}}_A(\mathbf{\Sigma}_g^*)$ is generated by the image of the Chebyshev--Frobenius morphism together
 with the boundary central elements $\alpha_{\partial}^{\pm 1}$ and $\alpha_{\partial}^{N}$ belongs to the image of~$\operatorname{Ch}_A$. Under the isomorphism of Theorem~\ref{theorem_skein_charvar}, $\operatorname{Ch}_A^{-1}\big(\alpha_{\partial}^{N}\big)$ corresponds to the regular function $X^{\gamma_{\partial}}_{-+}$ sending a representation $\rho\colon \pi_1(\Sigma_{g,1})\to \SL_2$ to the lower-left matrix coefficient of $\rho(\gamma_{\partial})=\mu(\rho)$. So
 \begin{gather}\label{eq_CharVar}
 \widehat{\overline{\mathcal{X}}}(\mathbf{\Sigma}_g^*) \cong \left\{ (\rho, z), \, \rho\colon \pi_1(\Sigma_{g,1})\to \SL_2, \, z\in \mathbb{C}^*, \, \text{such that}\, \rho(\gamma_{\partial}) = \left(\begin{smallmatrix} 0 & -z^{-N} \\ z^N & d\end{smallmatrix}\right)\!, \, d\in \mathbb{C} \right\}.\!\!\!
 \end{gather}
 The projection $\overline{\pi}\colon \widehat{\overline{\mathcal{X}}}(\mathbf{\Sigma}_g^*) \to \overline{\mathcal{X}}(\mathbf{\Sigma}_g^*)$ in this case is the regular covering sending $(\rho, z)$ to $\rho$. For $\rho\colon \pi_1(\Sigma_{g,1})\to \SL_2$ with $\rho(\gamma_{\partial})_{++}=0$, writing $\lambda:= \rho(\gamma)_{-+} \in \mathbb{C}^*$, we thus have
\[ \overline{Z}(\rho) \cong \quotient{ \mathbb{C}[X]}{ \big(X^N-\lambda\big)} \cong \quotient{ \mathbb{C}[X]}{\big(X^N-1\big)}.\]

So Corollary~\ref{theorem_BG_Ramifications} together with Remark~\ref{remark_AzumayaLocus} imply

\begin{Corollary}\label{coro_important}
The fully Azumaya locus of $\overline{\mathcal{S}}_A(\mathbf{\Sigma}_g^*)$ is the set of $\rho \in \overline{\mathcal{X}}(\mathbf{\Sigma}_g^*)$ such that
\[ \big(\overline{\mathcal{S}}_A(\mathbf{\Sigma}_g^*) \big)_{\rho} \cong \Mat_D\Big( \quotient{ \mathbb{C}[X]}{\big(X^N-1\big)}\Big),\]
where $D= N^{3g-1}$ is the PI-degree of $\overline{\mathcal{S}}_A(\mathbf{\Sigma}_g^*)$.
\end{Corollary}

The proof of Theorem~\ref{theorem1} will consist in proving that $ \big(\overline{\mathcal{S}}_A(\mathbf{\Sigma}_g^*) \big)_{\rho}$ does not depend, up to isomorphism, on~$\rho$.

\subsection{Poisson orders}\label{sec_PO}

We now prove that if $x$, $y$ belong to the same equivariant symplectic leaf, then $(\mathcal{S}_A(\mathbf{\Sigma}))_x \cong (\mathcal{S}_A(\mathbf{\Sigma}))_y$ using the theory of Poisson orders. The theory began with the work of De~Concini--Kac on $U_q\mathfrak{g}$ in~\cite{DeConciniKacRepQGroups}, the work of De~Concini--Lyubashenko on~$\mathcal{O}_q[G]$ in~\cite{DeConciniLyubashenko_OqG} and was fully developed by Brown--Gordon in~\cite{BrownGordon_PO} that we closely follow.

\begin{Definition}\quad
\begin{itemize}\itemsep=0pt
\item A \textit{Poisson order} is a $4$-tuple $(\mathcal{A}, \mathcal{X}, \phi, D)$ where
\begin{enumerate}\itemsep=0pt
\item[1)] $\mathcal{A}$ is an (associative, unital) affine $\mathbb{C}$-algebra finitely generated over its center $Z$,
\item[2)] $\mathcal{X}$ is a Poisson affine $\mathbb{C}$-variety,
\item[3)] $\phi\colon \mathcal{O}[\mathcal{X}] \hookrightarrow Z$ a finite injective morphism of algebras,
\item[4)] $D\colon \mathcal{O}[\mathcal{X}] \to \operatorname{Der} (\mathcal{A})\colon z\mapsto D_z$ a linear map such that for all $f,g \in \mathcal{O}[\mathcal{X}]$, we have
\[ D_f(\phi(g))= \phi(\{f,g\}).\]
\end{enumerate}
\item Let $G$ be an affine Lie group. A Poisson order $(\mathcal{A}, \mathcal{X}, \phi, D)$ is said $G$-\textit{equivariant} if $G$ acts on $\mathcal{A}$ by automorphism such that its action preserves $\phi(\mathcal{O}[\mathcal{X}])\subset \mathcal{A}$ and such that it is $D$ equivariant in the sense that for every $g\in G$, $z\in \mathcal{O}[\mathcal{X}]$ and $a\in \mathcal{A}$, one has
\[ D_{g\cdot z}(a) = g D_z \big(g^{-1}a\big).\]
\end{itemize}
\end{Definition}

The \textit{equivariant symplectic leaves} are then the $G$-orbits of the symplectic leaves in $\mathcal{X}$. The main result of the theory of Poisson orders is the following

\begin{Theorem}[{Brown--Gordon \cite[Proposition~4.3]{BrownGordon_PO}}]\label{theorem_PO_equivariant} For $(\mathcal{A}, \mathcal{X}, \phi, D)$ a $G$-equivariant Poisson order, if $x,y\in \mathcal{X}$ belong to the same equivariant symplectic leaf then $\mathcal{A}_x\cong \mathcal{A}_y$.
\end{Theorem}

\begin{Corollary} If $\mathcal{X}$ contains an equivariant symplectic leaf which is dense, then it is included into the fully Azumaya locus. In particular, if $\mathcal{X}$ contains a single equivariant symplectic leaf, then $\mathcal{A}$ is Azumaya.
\end{Corollary}

The main source of examples of Poisson orders come from

\begin{Example}\label{example_PO}
Let $\mathcal{A}_q$ a free, affine $\mathbb{C}\big(q^{\pm 1}\big)$-algebra, $N\geq 1$ and, writing $t:= N\big(q^N-1\big)$, the $\quotient{\mathbb{C}\big(q^{\pm 1}\big)}{\big(q^N-1\big)}$ algebra $\mathcal{A}_N:= \quotient{\mathcal{A}}{t}$ and $\pi\colon \mathcal{A}_q \to \mathcal{A}_N$ the quotient map. By fixing a basis $\mathcal{B}$ of $\mathcal{A}_q$ by flatness we can define a linear
 embedding $\hat{\cdot}\colon \mathcal{A}_{N} \to \mathcal{A}_q$ sending a basis element $b\in \mathcal{B}$ seen as element in $\mathcal{A}_{N}$ to the same element $\hat{b}$ seen as an element in~$\mathcal{A}_q$. Note that $\hat{\cdot}$ is a left inverse for~$\pi$. Suppose that the algebra $\mathcal{A}_{+1}=\mathcal{A}_q\otimes_{q=1}\mathbb{C}$ is commutative and suppose there exists a central embedding $\phi\colon \mathcal{A}_{+1} \hookrightarrow \mathcal{A}_{N}$ into the center of $\mathcal{A}_{N}$. Write $\mathcal{X}:=\operatorname{Specm}(\mathcal{A}_{+1})$ and define $D\colon \mathcal{A}_{+1} \to \operatorname{Der}(\mathcal{A}_N)$ by the formula
\[ D_xy:= \pi \left( \frac{ \big[\widehat{\phi(x)}, \hat{y}\big]}{N\big(q^N-1\big)} \right).\]
Clearly $D_x$ is a derivation, is independent on the choice of the basis~$\mathcal{B}$ and preserves $\phi(\mathcal{A}_{+1})$, so it defines a Poisson bracket $\{\cdot, \cdot \}_N$ on $\mathcal{A}_{+1}$ by
\begin{equation}\label{eq_bracket_order}
D_x \phi(y)= \phi (\{x,y\}_N).
\end{equation}
So, writing $\mathcal{X}=\operatorname{Specm}(\mathcal{A}_{+1})$, then $(\mathcal{A}_N, \mathcal{X}, \phi, D)$ is a Poisson order for this bracket. Note that if $\zeta_N$ is an $N$-th root of unity and $\mathcal{A}_{\zeta_N}=\mathcal{A}_q \otimes_{q=\zeta_N} \mathbb{C}$, we get a Poisson order $(\mathcal{A}_{\zeta_N}, \mathcal{X}, \phi, D)$ as well by tensoring by~$\mathbb{C}$.
\end{Example}

In particular, for $A$ a root of unity of odd order then $(\mathcal{S}_A(\mathbf{\Sigma}), \mathcal{X}(\mathbf{\Sigma}), \operatorname{Ch}_A, D)$ and $\big(\overline{\mathcal{S}}_A(\mathbf{\Sigma}), \allowbreak \overline{\mathcal{X}}(\mathbf{\Sigma}), \operatorname{Ch}_A, D\big)$ are Poisson orders. What is non trivial is the fact that the Poisson bracket $\{\cdot, \cdot\}_N$ coming from equation \eqref{eq_bracket_order} coincides with the Poisson bracket coming from quantization deformation in Definition~\ref{def_Poisson_bracket}. This fact is proved in \cite[Lemma~4.6]{KojuRIMS} and essentially follows from the existence of quantum traces.

Let $\varphi\colon \mathcal{O}_q[\SL_2] \to \mathbb{C}\big[X^{\pm 1}\big]$ be the morphism sending the generators $b$ and $c$ to $0$ and sending~$a$ and~$d$ to $X$ and $X^{-1}$. Said differently, $\varphi$ corresponds to the quotient map $\mathcal{O}_q[\SL_2]\cong \mathcal{S}_A(\mathbb{B}) \to \overline{\mathcal{S}}_A(\mathbb{B}) \cong \mathbb{C}\big[X^{\pm 1}\big]$.

Define an algebraic $(\mathbb{C}^*)^{ \mathcal{A}}$-action on $\mathcal{S}_{A}(\mathbf{\Sigma})$ by the coaction:
\[ \Delta^L_{\rm toric} \colon \ \mathcal{S}_{A}(\mathbf{\Sigma}) \xrightarrow{\Delta^L} (\mathcal{O}_q[\SL_2])^{\otimes \mathcal{A}} \otimes \mathcal{S}_{A}(\mathbf{\Sigma}) \xrightarrow{ (\varphi)^{\otimes \mathcal{A}} \otimes \id} \big(\mathbb{C}\big[X^{\pm 1}\big]\big)\otimes \mathcal{S}_{A}(\mathbf{\Sigma}).
\]
Then $\Delta^L_{\rm toric}$ induces similarly a $(\mathbb{C}^*)^{ \mathcal{A}}$-action on $\overline{\mathcal{S}}_{A}(\mathbf{\Sigma})$ by passing to the quotient.
Both actions preserve the image of the Chebyshev--Frobenius morphism and the equivariance of~$D$ for this action is an immediate consequence of the definition of~$D$ so the maps $\Delta^L_{\rm toric}$ endow
 $(\mathcal{S}_A(\mathbf{\Sigma}), \mathcal{X}(\mathbf{\Sigma}), \operatorname{Ch}_A, D)$ and $\big(\overline{\mathcal{S}}_A(\mathbf{\Sigma}), \overline{\mathcal{X}}(\mathbf{\Sigma}), \operatorname{Ch}_A, D\big)$ with a structure of $(\mathbb{C}^*)^{\mathcal{A}}$-equivariant Poisson order.
 Therefore Theorem~\ref{theorem_PO_equivariant} implies

\begin{Corollary}\label{coro_POSkein} If $x,y \in \mathcal{X}(\mathbf{\Sigma})$ belong to the same equivariant symplectic leaf, then $(\mathcal{S}_A(\mathbf{\Sigma}))_x\cong (\mathcal{S}_A(\mathbf{\Sigma}))_y$.
\par Similarly, if $x,y \in \overline{\mathcal{X}}(\mathbf{\Sigma})$ belong to the same equivariant symplectic leaf, then $\big( \overline{\mathcal{S}}_A(\mathbf{\Sigma})\big)_x\cong \big(\overline{\mathcal{S}}_A(\mathbf{\Sigma})\big)_y$.
\end{Corollary}

In the case of $\mathcal{S}_A(\mathbf{\Sigma}_g^*)$, since $\mu^{-1}\big(\SL_2^0\big)$ is an open dense equivariant symplectic leaf by Theorem~\ref{theorem_GJS_geom}, it is included into the Azumaya locus. In the other hand, by Theorem~\ref{theorem_basic_ppty}, the PI-degree of $\overline{\mathcal{S}}_A(\mathbf{\Sigma}_g^*)$ is strictly smaller than the PI-degree of $\mathcal{S}_A(\mathbf{\Sigma}_g^*)$, so Remark~\ref{remark_AzumayaLocus} implies that the leaf $\mu^{-1}\big(\SL_2^1\big)=\overline{\mathcal{X}}(\mathbf{\Sigma}_g^*)$ does not intersect the Azumaya locus of $\mathcal{S}_A(\mathbf{\Sigma}_g^*)$, therefore we have

\begin{Corollary}[{Ganev--Jordan--Safronov \cite[Theorem~1.1]{GanevJordanSafranov_FrobeniusMorphism}}]\label{coro_AL} The Azumaya locus of $\mathcal{S}_A(\mathbf{\Sigma}_g^*)$ is $\mu^{-1}\big(\SL_2^0\big)$. \end{Corollary}

\begin{proof}[Proof of Theorem \ref{theorem1}]
By Corollary~\ref{coro_leaves}, then $\overline{\mathcal{X}}(\mathbf{\Sigma}_g^*)$ has a single equivariant symplectic orbit so Corollary~\ref{coro_POSkein} implies that the isomorphism class of $\big( \overline{\mathcal{S}}_A(\mathbf{\Sigma}_g^*)\big)_{\rho}$ does not depend on $\rho \in \overline{\mathcal{X}}(\mathbf{\Sigma}_g^*)$. Corollary~\ref{coro_important} implies that the fully Azumaya locus is equal to the whole space~$\overline{\mathcal{X}}(\mathbf{\Sigma}_g^*)$. This concludes the proof.
\end{proof}

\section{Mapping class groups representations}\label{sec_representations}

\subsection{Mapping class group action}\label{sec_mcg}

 Let $\Mod(\Sigma_{g,1})$ be the mapping class group of $\Sigma_{g,1}$ and define a right action of $\Mod(\Sigma_{g,1})$ on both $\mathcal{S}_A(\mathbf{\Sigma}_g^*)$ and $\overline{\mathcal{S}}_A(\mathbf{\Sigma}_g^*)$ by the formula
\[ \phi \cdot [D,s] := \big[\phi^{-1}(D), s\circ \phi \big], \qquad \phi \in \Mod(\Sigma_{g,1}).\]
The right action of $\Mod(\Sigma_{g,1})$ preserves the centers of $\mathcal{S}_A(\mathbf{\Sigma}_g^*)$ and $\overline{\mathcal{S}}_A(\mathbf{\Sigma}_g^*)$ as well as the image of the Chebyshev--Frobenius morphism so it induces a left action of $\Mod(\Sigma_{g,1})$ on $\mathcal{X}(\mathbf{\Sigma}_g^*)$, $\widehat{\overline{\mathcal{X}}}(\mathbf{\Sigma}_g^*)$ and ${\overline{\mathcal{X}}}(\mathbf{\Sigma}_g^*)$ in such a way that $\overline{\pi}\colon \widehat{\overline{\mathcal{X}}}(\mathbf{\Sigma}_g^*)\to {\overline{\mathcal{X}}}(\mathbf{\Sigma}_g^*)$ is equivariant. More explicitly, the action of $\Mod(\Sigma_{g,1})$ on ${\mathcal{X}}(\mathbf{\Sigma}_g^*)$ is given by
\[ (\phi \cdot \rho)(\gamma):= \rho( \phi(\gamma)), \qquad \text{for all } \phi \in \Mod(\Sigma_{g,1}), \ \gamma \in \pi_1(\Sigma_{g,1}), \ \rho\colon \ \pi_1(\Sigma_{g,1})\to \SL_2.\]
Since every mapping class in $\Mod(\Sigma_{g,1})$ leaves $\gamma_{\partial}$ invariant, then $\mu( \phi \cdot \rho) = \mu (\rho)$. Recall from equation \eqref{eq_CharVar} that an element of $\widehat{\overline{\mathcal{X}}}(\mathbf{\Sigma}_g^*)$ is a pair $\widehat{\rho}= (\rho, z)$ where $\rho\colon \pi_1(\Sigma_{g,1})\to \SL_2$, $z\in \mathbb{C}^*$ are such that $\rho(\gamma_{\partial})= \left(\begin{smallmatrix} 0 & -z^{-N} \\ z^N & d\end{smallmatrix}\right)$ for some $d\in \mathbb{C}$.

Since the central boundary element $\alpha_{\partial}$ is also preserved by every mapping class in $\Mod(\Sigma_{g,1})$, the action of $\Mod(\Sigma_{g,1})$ on
$\widehat{\overline{\mathcal{X}}}(\mathbf{\Sigma}_g^*)$ is given by the formula
\[ \phi \cdot (\rho, z) = (\phi \cdot \rho, z) \qquad \text{for all } \phi \in \Mod(\Sigma_{g,1}), \ (\rho,z) \in \widehat{\overline{\mathcal{X}}}(\mathbf{\Sigma}_g^*).\]

\subsection{Construction of mapping class groups representations}

Let $\mathcal{X}^0 \subset \mathcal{X}_{\SL_2}(\mathbf{\Sigma}_g^*)$ be the subset of representations $\rho\colon \pi_1(\Sigma_{g,1}) \to \SL_2$ such that $\mu (\rho) \in \SL_2^0$, i.e., $\mathcal{X}^0$ is the Azumaya locus of $\mathcal{S}_A(\mathbf{\Sigma}_g^*)$ by Corollary~\ref{coro_AL}. For each $\rho \in \mathcal{X}^0$, fix an irreducible representation
\[ r_{\rho} \colon \ \mathcal{S}_A(\mathbf{\Sigma}_g^*) \to \End( V(\rho))\]
 with induced central character $r_{\rho}$: so $r_{\rho}$ is unique up to unique isomorphism by Corollary~\ref{coro_AL} and has dimension $N^{3g}$ by Theorem~\ref{theorem_PI_Deg}.
 For $\phi\in \Mod(\Sigma_{g,1})$, consider the representation $\phi \bullet r_{\rho}\colon {\mathcal{S}}_A(\mathbf{\Sigma}_g^*) \to \End( V({\rho}))$ defined by
\[ (\phi \bullet r_{\rho}) (X) := r_{\widehat{\rho}} ( \phi \cdot X), \qquad \text{for all }X\in {\mathcal{S}}_A(\mathbf{\Sigma}_g^*).\]
Since the representation $\phi \bullet r_{\rho}$ has dimension $N^{3g}$, it is irreducible and since its induced central character is $\phi \cdot {\rho}$, there exists an isomorphism $L_{\rho}(\phi)\colon V(\rho) \xrightarrow{\cong} V(\phi \cdot {\rho})$, unique up to multiplication by an invertible scalar, such that:
\begin{equation}\label{eq_Egorov}
\phi \bullet r_{\rho}(X) = L_{\rho}(\phi) ( r_{\rho}(X) ) L_{\rho}(\phi)^{-1}, \qquad \text{for all }X\in {\mathcal{S}}_A(\mathbf{\Sigma}_g^*).
\end{equation}
Note that, for $\phi_1, \phi_2 \in \Mod(\Sigma_{g,1})$ then by unicity, we have
\begin{equation}\label{eq_machintruc} L_{\phi_1 \cdot {\rho}}(\phi_2) \circ L_{\rho}(\phi_1) = c L_{\rho}(\phi_2\circ \phi_1), \qquad \text{for some }c\in \mathbb{C}^*.\end{equation}

Similarly, for each $\widehat{\rho}=(\rho, z) \in \widehat{\overline{\mathcal{X}}}(\mathbf{\Sigma}_g^*)$, we fix one irreducible representation
\[ r_{\widehat{\rho}} \colon \ \overline{\mathcal{S}}_A(\mathbf{\Sigma}_g^*) \to \End( V(\widehat{\rho})),\]
whose induced character on the center is $\widehat{\rho}$. By Theorem~\ref{theorem1}, such a representation is unique up to isomorphism and has dimension $N^{3g-1}$ by Theorem~\ref{theorem_center}. Define the intertwiner $L_{\widehat{\rho}}\colon V(\widehat{\rho}) \to V(\phi \cdot \widehat{\rho})$ in the same manner, i.e., by the formula
\[ r_{\widehat{\rho}}(X^{\phi}) = L_{\widehat{\rho}}(\phi) ( r_{\rho}(X) ) L_{\widehat{\rho}}(\phi)^{-1}, \qquad \text{for all }X\in \overline{{\mathcal{S}}}_A(\mathbf{\Sigma}_g^*).\]

\begin{Definition}Let $G\subset \Mod(\Sigma_{g,1})$ be a subgroup and $\mathcal{O}\subset \Hom(\pi_1(\Sigma_{g,1}), \SL_2)$ a finite $G$-orbit included in some leaf $\mu^{-1}(g)$ for $g=\left(\begin{smallmatrix} a & b \\ c & d\end{smallmatrix}\right)$.
\begin{itemize}\itemsep=0pt
\item If $a\neq 0$, consider the finite-dimensional space
\[ W(\mathcal{O}) := \oplus_{{\rho} \in \mathcal{O}} V({\rho})\]
 and let $\pi\colon G \to \PGL(W(\mathcal{O}))$ be the projective finite-dimensional representation defined by
\[ \pi(g) := \oplus_{\widehat{\rho}\in \mathcal{O}} L_{{\rho}}(g), \qquad \text{for all }g\in G.\]
\item
 If $a=0$, further choose a scalar $z\in \mathbb{C}^*$ such that $z^N=c$ and consider the $G$ orbit $\overline{\mathcal{O}}_z= \{ (\rho, z), \rho \in \mathcal{O}\} \subset \widehat{\overline{\mathcal{X}}}(\mathbf{\Sigma}_g^*)$.
Define the finite-dimensional space
\[ W(\widehat{\mathcal{O}}_z) := \oplus_{\widehat{\rho} \in \widehat{\mathcal{O}}_z} V(\widehat{\rho}).\]
 Let $\pi\colon G \to \PGL\big(W(\widehat{\mathcal{O}}_z)\big)$ be the projective finite-dimensional representation defined by
\[\pi(g) := \oplus_{\widehat{\rho}\in \widehat{\mathcal{O}}_z} L_{\widehat{\rho}}(g), \qquad \text{for all }g\in G.\]
\end{itemize}
\end{Definition}

 \begin{Example}\quad
 \begin{enumerate}\itemsep=0pt
 \item For instance, for $H\subset \SL_2$ a finite subgroup, then
 \[\mathcal{O}_H:= \Hom(\pi_1(\Sigma_{g,1}), H) \subset \Hom(\pi_1(\Sigma_{g,1}), \SL_2)\] is a finite $\Mod(\Sigma_{g,1})$-orbit. In particular, taking $H=\{ \id\}$ for which $\mathcal{O}_H$ is the singleton formed by the trivial representation, then we obtain representations of $\Mod(\Sigma_{g,1})$ of dimension~$N^{3g}$.
 \item Similarly, taking for $G$ the Torelli subgroup and for $\mathcal{O}$ a singleton formed by a diagonal representation, i.e., a representation that factorizes through the diagonal embedding as $\rho\colon \pi_1(\Sigma_{g,1})\to \mathbb{C}^{*} \xrightarrow{\varphi} \SL_2$, we get representations of the Torelli group of dimensions $N^{3g}$.
 \end{enumerate}
 \end{Example}

\begin{Lemma}\label{lemma_equiv} For a finite $G$ orbit $\mathcal{O}$ in some leaf $\mu^{-1}(g)$ with $g\in \SL_2^1$, then the representation $\big(\pi, W\big(\widehat{\mathcal{O}}_z\big)\big)$ does not depend on the choice of $z$ up to isomorphism.
\end{Lemma}

\begin{proof} Recall from Section \ref{sec_PO} the toric action of $\mathbb{C}^*$ on $\overline{\mathcal{S}}_A(\mathbf{\Sigma}_g^*)$ given by the co-action
\[
 \Delta_{\rm toric}^L \colon \ \overline{\mathcal{S}}_A(\mathbf{\Sigma}_g^*) \xrightarrow{\Delta^L_a} \mathcal{O}_q[\SL_2]\otimes \overline{\mathcal{S}}_A(\mathbf{\Sigma}_g^*) \xrightarrow{\varphi \otimes \id} \mathbb{C}[X^{\pm 1}] \otimes \overline{\mathcal{S}}_A(\mathbf{\Sigma}_g^*).
 \]
 Its induced action on the center of $\overline{\mathcal{S}}_A(\mathbf{\Sigma}_g^*)$ is given by
\[
 \lambda \bullet (\rho, z) := \left( \lambda^N \cdot \rho, \lambda z \right), \qquad \text{for all } \lambda \in \mathbb{C}^*, \ (\rho, z) \in \overline{\mathcal{X}}(\mathbf{\Sigma}_g^*),
 \]
 where
\[ \big(\lambda^N \cdot \rho\big) (\gamma) = \begin{pmatrix} \lambda^{N} & 0 \\ 0 & \lambda^{-N} \end{pmatrix} \rho(\gamma) \begin{pmatrix} \lambda^{-N} & 0 \\ 0 & \lambda^{N} \end{pmatrix}.\]
 For $\hat{\rho}=(\rho,z ) \in \overline{\mathcal{X}}(\mathbf{\Sigma}_g^*)$ and $\lambda \in \mathbb{C}^*$, the toric action of $\lambda$ on $\overline{\mathcal{S}}_A(\mathbf{\Sigma}_g^*)$ induces an isomorphism of algebras
\[ F_{\lambda, \hat{\rho}}\colon \ \End(V(\hat{\rho})) \cong \quotient{\overline{\mathcal{S}}_A(\mathbf{\Sigma}_g^*)}{\mathcal{I}_{\hat{\rho}}} \xrightarrow{ \lambda \bullet \cdot} \quotient{\overline{\mathcal{S}}_A(\mathbf{\Sigma}_g^*)}{\mathcal{I}_{\lambda \bullet \hat{\rho}}} \cong \End(V(\lambda \bullet\hat{\rho})).
\]
 Since the group of outer morphisms of a matrix algebra is trivial, the isomorphism $F_{\lambda, \hat{\rho}}$ is inner, so there exists $f_{\lambda, \hat{\rho}} \colon V(\hat{\rho}) \xrightarrow{\cong} V(\lambda \bullet \hat{\rho})$ such that
\[ F_{\lambda, \hat{\rho}} (X) = f_{\lambda, \hat{\rho}} X f_{\lambda, \hat{\rho}}^{-1}, \qquad \text{for all }X\in \End(V(\hat{\rho})).\]
 For $G\subset \Mod(\Sigma_{g,1})$ and $\widehat{\mathcal{O}}\subset \overline{\mathcal{X}}(\mathbf{\Sigma}_g^*)$ a finite $G$-orbit, define an isomorphism
 \[ f_{\lambda, \widehat{\mathcal{O}}}\colon \ W\big(\widehat{\mathcal{O}}\big) \xrightarrow{\cong} W\big(\lambda \bullet \widehat{\mathcal{O}}\big), \qquad f_{\lambda, \widehat{\rho}}:= \oplus_{\widehat{\rho}\in \widehat{\mathcal{O}}} f_{\lambda, \widehat{\mathcal{O}}}.
 \]
 By construction, the isomorphism $ f_{\lambda, \widehat{\mathcal{O}}}$ is $G$-equivariant, so it provides an isomorphism between the representations $\big(\pi, W\big(\widehat{\mathcal{O}}\big)\big)$ and $\big(\pi, W\big(\lambda \bullet \widehat{\mathcal{O}}\big)\big)$. In the particular case where $\lambda$ belongs to the subgroup $\nu(N)\subset \mathbb{C}^*$ of complexes such that $\lambda^N=1$, the action reads $\lambda\bullet (\rho, z)= (\rho, \lambda z)$ and this action acts transitively on the possible~$z$ so the conclusion of the lemma follows.
\end{proof}

 In virtue of Lemma~\ref{lemma_equiv} and by abuse of notations, we will now simply denote by $(\pi, W(\mathcal{O}))$ one of the pairwise isomorphic representations $(\pi, W(\widehat{\mathcal{O}}_z))$ associated to a reduced finite orbit~$\mathcal{O}$.

 \subsection{Kernel of the representations}

 The present paper was originally motivated by the hope to find faithful finite-dimensional representations of the mapping class groups (modulo center), it is thus natural to study the eventual kernel of the representations defined in the previous subsection. We will not achieve this goal but instead provide tools to prove that a given element does not belong to the kernel.

The kernel of the Witten--Reshetikhin--Turaev representations
\[
\rho^{\rm WRT}\colon \ \Mod(\Sigma_g) \to \PGL(V(\Sigma_g))
\] is not known but it is easy to see that it contains the normal subgroup generated by the elements $T_{\gamma}^N - \id$ where $T_{\gamma}$ are the Dehn--Twists associated to simple closed curves $\gamma$. These representations are defined using some (irreducible) representations $r^{\rm WRT}\colon \mathcal{S}_A(\Sigma_g) \to \End(V(\Sigma_g))$ of the usual skein algebras, the representation~$\rho^{\rm WRT}$ can be defined from~$r^{\rm WRT}$ by the same strategy than we used in the previous subsection, that is using an analogue of equation~\eqref{eq_Egorov}. The fact that the kernel of $\rho^{\rm WRT}$ is hard to compute is related to the fact that the kernel of $r^{\rm WRT}$ is unknown (though it is proved in~\cite{BonahonWong4} that is contains the elements $T_N([\gamma])+2\id$).

An advantage of our representations compared to $\rho^{\rm WRT}$ is the fact that the kernel of the representations $r_{\rho}\colon \mathcal{S}_A(\mathbf{\Sigma}_g)\to \End(V(\rho))$ and $r_{\hat{\rho}}\colon \overline{\mathcal{S}}_A(\mathbf{\Sigma}_g)\to \End(V(\widehat{\rho}))$ are well known: they correspond to the ideals $\mathcal{I}_{\rho}$ and $\mathcal{I}_{\widehat{\rho}}$ by definition. We can thus deduce

 \begin{Proposition}\label{prop_kernel}
 Let $G\subset \Mod(\Sigma_{g,1})$ and $\pi\colon G \to \PGL( W(\mathcal{O}))$ a representation associated to a finite orbit $\mathcal{O}$. Let $\phi \in G$ and suppose that there exists a closed curve $\gamma$ and $\rho \in \mathcal{O}$ when the orbit is in the big cell $($resp.~$\rho \in \widehat{\mathcal{O}}_z$ when the orbit is reduced$)$ such that
\[ \phi(\gamma) \not \equiv z \gamma \pmod{\mathcal{I}_{\rho}}, \qquad \text{for all }z\in \mathbb{C}^*.\]
 Then $\pi(\phi)\neq \id$.
 \end{Proposition}

 \begin{proof} This is an immediate consequence of the definitions.
 \end{proof}

\subsection{The use of quantum Teichm\"uller theory in the reduced case}\label{sec_QT_reduced}

The criterion found in Proposition \ref{prop_kernel} is not easy to use in practice. In this subsection, we explain how the use of quantum traces and quantum Teichm\"uller spaces might simplify this criterion for practical computations in the case where the orbit is reduced.
Recall that the triangle~$\mathbb{T}$ is the marked surface made of a disc with three boundary arcs.

 \begin{Definition} A marked surface $\mathbf{\Sigma}$ is \textit{triangulable} if it can be obtained from a finite disjoint union $\mathbf{\Sigma}_{\Delta}= \bigsqcup_i \mathbb{T}_i$ of triangles by gluing some pairs of edges. A \textit{triangulation} $\Delta$ is then the data of the disjoint union $\mathbf{\Sigma}_{\Delta}$ together with the set of glued pair of edges.
 \end{Definition}
 The connected components $\mathbb{T}_i$ of $\mathbf{\Sigma}_{\Delta}$ are called \textit{faces} and their set is denoted $F(\Delta)$. The image in $\Sigma$ of the edges of the faces $\mathbb{T}_i$ are called \textit{edges} of $\Delta$ and their set is denoted $\mathcal{E}(\Delta)$. Note that each boundary arc is an edge in $\mathcal{E}(\Delta)$; the elements of the complementary $\mathring{\mathcal{E}}(\Delta):= \mathcal{E}(\Delta) \setminus \mathcal{A}$ are called \textit{inner edges}. Figure~\ref{fig_triangulation} illustrates a triangulation $\Delta_1$ of $\mathbf{\Sigma}_1^*$. Note that when $(\mathbf{\Sigma}, \Delta)$ is a triangulated marked surface, the fusion $\mathbf{\Sigma}_{1\circledast 2}$ admits a natural triangulation obtained from $\Delta$ by adding one face corresponding to the added triangle. Therefore the triangulation $\Delta_1$ of $\mathbf{\Sigma}_1^*$ induces a triangulation~$\Delta_g$ of~$\mathbf{\Sigma}_g^*$.

 \begin{figure}[!h]\centering
\includegraphics[width=3cm]{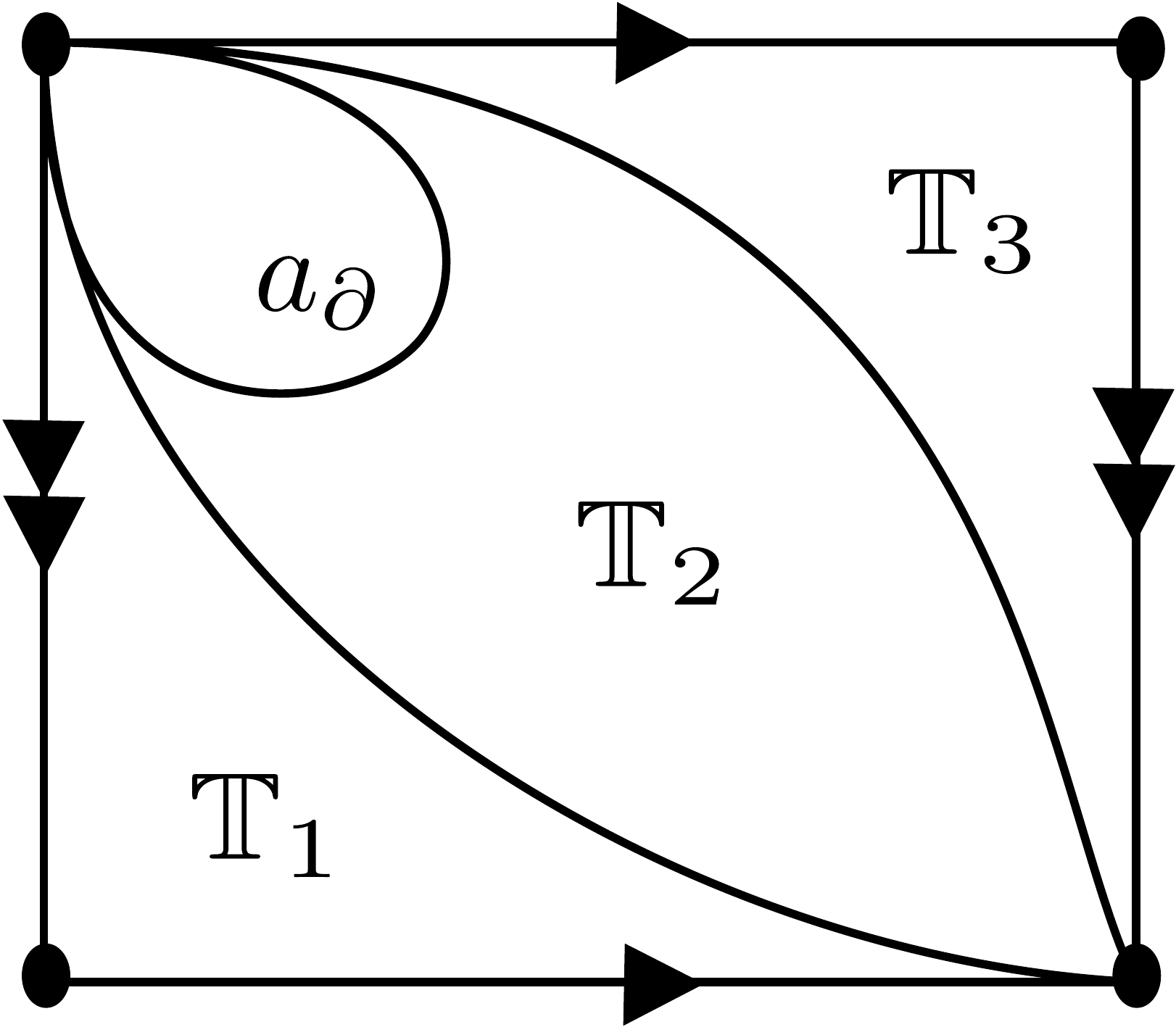}
\caption{A triangulation of $\mathbf{\Sigma}_1^*$.}\label{fig_triangulation}
\end{figure}

\begin{Definition}\quad
\begin{enumerate}\itemsep=0pt
\item Consider a pair $\mathbb{E}=(E,(\cdot, \cdot) )$ where $E$ is a free finite rank $\mathbb{Z}$-module (so $E\cong \mathbb{Z}^n$) and $(\cdot, \cdot)\colon E\times E \to \mathbb{Z}$ is a skew-symmetric pairing. The \textit{quantum torus} $\mathbb{T}_q(\mathbb{E})$ is the algebra generated by generators $Z_e$, $e\in E$ with relations $Z_aZ_b= A^{-(a,b)/4}Z_{a+b}$. Said differently, given $e=(e_1,\dots, e_n)$ a basis of~$E$, the quantum torus $\mathbb{T}_q(\mathbb{E})$ is isomorphic to the complex algebra generated by invertible elements $Z_{e_i}^{\pm 1}$ with relations $Z_{e_i}Z_{e_j}=A^{-(e_i,e_j)/2}Z_{e_j}Z_{e_i}$.
\item Let $(\mathbf{\Sigma}, \Delta)$ be a triangulated marked surface and denote by $\mathcal{E}(\Delta)$ the set of edges of the triangulation. A map $k\colon \mathcal{E}(\Delta)\to \mathbb{Z}$ is \textit{balanced} if for any face $\mathbb{T}$ of the triangulation with edges $a$, $b$, $c$ then $k(a)+k(b)+k(c)$ is even. We denote by $K_{\Delta}$ the $\mathbb{Z}$-module of balanced maps. For $e$ and $e'$ two edges, denote by $a_{e,e'}$ the number of faces $\mathbb{T}\in F(\Delta)$ such that~$e$ and~$e'$ are edges of~$\mathbb{T}$ and such that we pass from $e$ to $e'$ in the counter-clockwise direction in~$\mathbb{T}$. The \textit{Weil--Petersson} form $ (\cdot, \cdot )^{\rm WP}\colon K_{\Delta} \times K_{\Delta} \rightarrow \mathbb{Z}$ is the skew-symmetric form defined by $ ( \mathbf{k}_1, \mathbf{k}_2)^{\rm WP}:= \sum_{e,e'} \mathbf{k}_1(e)\mathbf{k}_2(e')(a_{e,e'}-a_{e',e})$. The \textit{balanced Chekhov--Fock algebra} is the quantum torus $\mathcal{Z}_q(\mathbf{\Sigma}, \Delta):= \mathbb{T}_q\big( K_{\Delta}, (\cdot, \cdot)^{\rm WP}\big)$.
\item If $A$ is root of unity of order $N$, the \textit{Frobenius morphism} is the central embedding
\[ \operatorname{Fr}_N\colon \ \mathbb{T}_{+1}(\mathbb{E}) \hookrightarrow \mathcal{Z} ( \mathbb{T}_q(\mathbb{E}) ), \qquad \operatorname{Fr}_N( Z_a) := (Z_a)^N = Z_{Na}.\]
\end{enumerate}
\end{Definition}

 \begin{Theorem}[{De~Concini--Procesi \cite[Proposition~7.2]{DeConciniProcesiBook}}]\label{theorem_QT} If $A \in \mathbb{C}^*$ is a root of unity, then $\mathbb{T}_q(\mathbb{E})$ is Azumaya.
 \end{Theorem}

The \textit{quantum trace} is an algebra embedding
\[ \Tr^{\Delta}_A\colon \ \overline{\mathcal{S}}_A(\mathbf{\Sigma}) \hookrightarrow \mathcal{Z}_q(\mathbf{\Sigma}, \Delta)\]
defined by Bonahon--Wong for unmarked surfaces in \cite{BonahonWongqTrace} and extended to marked surfaces by L\^e in~\cite{LeStatedSkein}. It is characterized as follows.
First consider the triangle $\mathbb{T}$ with edges $e_1$, $e_2$, $e_3$ and arcs $\alpha_i$ as in Figure~\ref{fig_triangle}.

 \begin{figure}[!h]\centering
 \includegraphics[width=2cm]{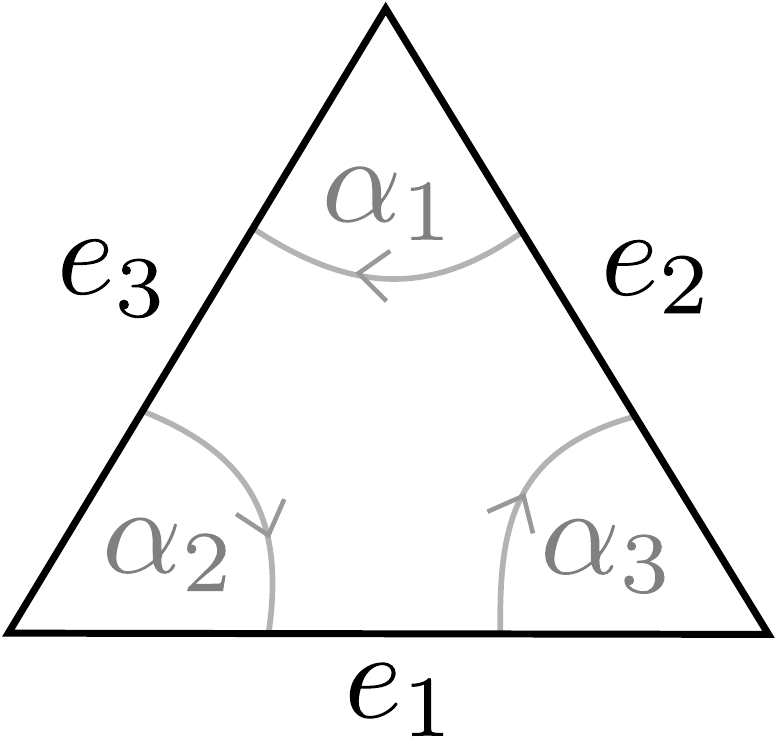}
\caption{The triangle and some arcs.}\label{fig_triangle}
\end{figure}

For $i\in \mathbb{Z}/3\mathbb{Z}$, let $\mathbf{k}_i $ be the balanced map sending $e_i$ to $0$ and $e_{i+1}$, $e_{i+2}$ to $1$.
The quantum trace $\Tr^{\mathbb{T}}_A\colon \overline{\mathcal{S}}_A(\mathbb{T}) \xrightarrow{\cong} \mathcal{Z}_q(\mathbb{T})$ is the isomorphism characterized by the facts that it sends $(\alpha_i)_{-+}$ to $0$, $(\alpha_i)_{++}$ to $Z^{\mathbf{k}_i}$ and $(\alpha_i)_{--}$ to $Z^{-\mathbf{k}_i}$.

Consider a general triangulated marked surface $(\mathbf{\Sigma}, \Delta)$, so $\mathbf{\Sigma}$ is obtained from a disjoint union of triangle $\mathbf{\Sigma}_{\Delta}:= \bigsqcup_{\mathbb{T}\in F(\Delta)} \mathbb{T}$ by gluing some pairs of edges. Let $\widehat{\mathcal{E}}$ be the set of edges (boundary arcs) of the disjoint union $\mathbf{\Sigma}_{\Delta}$ and let $\pi\colon \widehat{\mathcal{E}} \to \mathcal{E}$ be the induced surjective map.

Consider the embedding $\theta^{\Delta}\colon \mathcal{Z}_q(\mathbf{\Sigma}, \Delta) \hookrightarrow \otimes_{\mathbb{T}\in F(\Delta)} \mathcal{Z}_q(\mathbb{T})$ sending $Z^{\mathbf{k}}$ to
 $\otimes_{\mathbb{T}} Z^{\mathbf{k}_{\mathbb{T}}}$ where $\mathbf{k}_{\mathbb{T}} (e)\allowbreak := \mathbf{k}(\pi(e))$.
 \begin{Definition}[{\cite{BonahonWongqTrace, LeStatedSkein}}]
The quantum trace $\Tr^{\Delta}_A$ is the unique algebra morphism making commuting the following diagram
\[
\begin{tikzcd}
\overline{\mathcal{S}}_A(\mathbf{\Sigma})
\ar[r, hook, "\theta^{\Delta}"] \ar[d, dotted, hook, "\Tr^{\Delta}_A"]
&
 \otimes_{\mathbb{T}\in F(\Delta)} \overline{\mathcal{S}}_A(\mathbb{T})
 \ar[d, "\cong"', "\otimes_{\mathbb{T}} \Tr^{\mathbb{T}}_A"]
 \\
 \mathcal{Z}_q(\mathbf{\Sigma}, \Delta)
 \ar[r, hook, "\theta^{\Delta}"]
&
\otimes_{\mathbb{T}\in F(\Delta)} \mathcal{Z}_q(\mathbb{T}).
\end{tikzcd}
\]
\end{Definition}

An easy consequence of the definition is the fact that $\Tr^{\Delta}_{+1} \circ \operatorname{Ch}_A = \operatorname{Fr}_N \circ \Tr^{\Delta}_A$, i.e., that the Chebyshev--Frobenius morphism $\operatorname{Ch}_A$ is the restriction of the Frobenius $\operatorname{Fr}_N$ through the quantum trace (this is how $\operatorname{Ch}_A$ was first defined in~\cite{BonahonWongqTrace}).
The centers and PI-degrees of the balanced Chekhov--Fock algebras were computed in~\cite{BonahonWong2} for unmarked surfaces and in~\cite{KojuQuesneyQNonAb} for marked surfaces. In the particular case of $\mathbf{\Sigma}_g^*$, it is described as follows. Let us fix a triangulation~$\Delta$ of~$\mathbf{\Sigma}_g^*$. Let $\mathbf{k}_{\partial}\colon \mathcal{E}(\Delta) \to \mathbb{Z}$ be the balanced map sending every $e\in \mathcal{E}(\Delta)$ to $\mathbf{k}_{\partial}(e):= 2$. Then $H_{\partial}:= Z^{\mathbf{k}_{\partial}}$ is central in $\mathcal{Z}_q(\mathbf{\Sigma}_g^*, \Delta)$ and we have $\Tr^{\Delta}_A (\alpha_{\partial})= H_{\partial}$.

\begin{Theorem}[{\cite{KojuQuesneyQNonAb}}] \label{theorem_centerCF} Suppose that $A^{1/2}$ is a root of unity of odd order~$N$.
\begin{enumerate}\itemsep=0pt
\item[$1.$] The center of $\mathcal{Z}_q(\mathbf{\Sigma}_g^*, \Delta)$ is generated by the image of the Frobenius morphims $\operatorname{Fr}_N$ together with the elements $H_{\partial}^{\pm 1}$. In particular the quantum trace $\Tr^{\Delta}_A$ embeds the center of $\overline{\mathcal{S}}_A(\mathbf{\Sigma})$ into the center of $\mathcal{Z}_q(\mathbf{\Sigma}_g^*, \Delta)$.
\item[$2.$] The PI-degree of $\mathcal{Z}_q(\mathbf{\Sigma}_g^*, \Delta)$ is equal to $N^{3g-1}$, i.e., to the PI-degree of $\mathcal{S}_A(\mathbf{\Sigma}_g^*)$.
\end{enumerate}
\end{Theorem}

Let $\mathcal{Z}$ be the center of $\mathcal{Z}_q(\mathbf{\Sigma}_g^*, \Delta)$ and
consider the algebraic tori
\[ \mathcal{Y}(\mathbf{\Sigma}_g^*, \Delta):= \operatorname{Specm} ( \mathcal{Z}_{+1}(\mathbf{\Sigma}_g^*, \Delta))\] and $\widehat{\mathcal{Y}}(\mathbf{\Sigma}_g^*, \Delta):= \operatorname{Specm}( \mathcal{Z})$ and the natural regular covering $\pi\colon \widehat{\mathcal{Y}}(\mathbf{\Sigma}_g^*, \Delta) \to \mathcal{Y}(\mathbf{\Sigma}_g^*, \Delta)$. By Theorem~\ref{theorem_centerCF}, $\pi$ is a principal $( \mathbb{Z}/N\mathbb{Z})^{\times}$ bundle; more precisely, $\widehat{\mathcal{Y}}(\mathbf{\Sigma}_g^*, \Delta)$ is identified with the set of pairs $\widehat{\mathbf{z}}=(\mathbf{z}, h_{\partial})$ where $h_{\partial}\in \mathbb{C}^*$ and $\mathbf{z}\in \mathcal{Y}(\mathbf{\Sigma}_g^*, \Delta)$ corresponds to a character sending $H_{\partial}\in \mathcal{Z}_{+1}(\mathbf{\Sigma}, \Delta)$ to $h_{\partial}^N$.

The quantum trace in $A^{1/2}=+1$
\[ \Tr_{+1}^{\Delta}\colon \ \overline{\mathcal{S}}_{+1}(\mathbf{\Sigma}_g^*) \hookrightarrow \mathcal{Z}_{+1}(\mathbf{\Sigma}_g^*, \Delta)\] defines regular maps
$\widehat{\mathcal{NA}}\colon \widehat{\mathcal{Y}}(\mathbf{\Sigma}_g^*, \Delta) \to \widehat{\overline{\mathcal{X}}}(\mathbf{\Sigma}_g^*)$ and
 $\mathcal{NA}\colon \mathcal{Y}(\mathbf{\Sigma}_g^*, \Delta) \to \overline{\mathcal{X}}(\mathbf{\Sigma}_g^*)$ named the \textit{non-abelianization maps}. Since $\Tr_A^{\Delta}(\alpha_{\partial})=H_{\partial}$, $\widehat{\mathcal{NA}}$ is a morphism of principal $( \mathbb{Z}/N\mathbb{Z})^{\times}$ bundles: it sends a pair $(\mathbf{z}, h_{\partial})$ to $(\mathcal{NA}(\mathbf{z}), h_{\partial})$.

Since $\Tr_{+1}^{\Delta}$ is injective, the map $\mathcal{NA}$ is dominant, so by irreducibility of $\overline{\mathcal{X}}(\mathbf{\Sigma}_g^*)$, the image of~$\mathcal{NA}$ is dense. Unfortunately, a precise description of this image is still unknown.

\begin{Definition}A point $\rho \in \overline{\mathcal{X}}(\mathbf{\Sigma}_g^*)$ \textit{admits a }$\Delta$\textit{-lift} if $\rho$ is in the image of $\mathcal{NA}$. An orbit $\mathcal{O}\subset \overline{\mathcal{X}}(\mathbf{\Sigma}_g^*)$ \textit{admits a }$\Delta$\textit{-lift} if each of its points does.
\end{Definition}

For $\widehat{\mathbf{z}} \in \mathcal{Y}(\mathbf{\Sigma}_g^*, \Delta)$ corresponding to a character $\chi_{\widehat{\mathbf{z}}}$ over the center of $\mathcal{Z}_q(\mathbf{\Sigma}_g^*, \Delta)$, let $\mathcal{J}_{\widehat{\mathbf{z}}} \subset \mathcal{Z}_q(\mathbf{\Sigma}_g^*, \Delta)$ be the ideal generated by the kernel of $\chi_{\widehat{\mathbf{z}}}$. Then, writing $\widehat{\rho}:=\mathcal{NA}(\widehat{\mathbf{z}})$ and in virtue of Theorems~\ref{theorem_QT} and~\ref{theorem_centerCF}, we have
\begin{equation}\label{eq_isom}
 \End( V(\widehat{\rho})) \cong \quotient{ \overline{\mathcal{S}}_A(\mathbf{\Sigma})}{\mathcal{I}_{\widehat{\rho}}} \xrightarrow[\cong]{\tr^{\Delta}_A} \quotient{\mathcal{Z}_q(\mathbf{\Sigma}_g^*, \Delta)}{\mathcal{J}_{\widehat{\mathbf{z}}}}.
 \end{equation}
So a point $\rho \in \overline{\mathcal{X}}(\mathbf{\Sigma}_g^*)$ admits a $\Delta$-lift if and only if for one (or equivalently every) of its lifts $\widehat{\rho}\in \widehat{\overline{\mathcal{X}}}(\mathbf{\Sigma}_g^*)$, the simple $ \overline{\mathcal{S}}_A(\mathbf{\Sigma})$ module $V(\widehat{\rho})$ extends (in a non necessarily unique way) to a simple $\mathcal{Z}_q(\mathbf{\Sigma}_g^*, \Delta)$ module through the quantum trace. Equivalently, an orbit $\mathcal{O}$ admits a $\Delta$-lift if and only if the $ \overline{\mathcal{S}}_A(\mathbf{\Sigma})$ module $W(\mathcal{O})$ extends to a~$\mathcal{Z}_q(\mathbf{\Sigma}_g^*, \Delta)$ module. Unfortunately, the author is not aware of any criterion to prove that a given orbit $\mathcal{O}$ admits a $\Delta$-lift, so the usefulness of the methods developed in this subsection is purely conjectural. The criterion of Proposition~\ref{prop_kernel} can be improved as follows.

\begin{Proposition}\label{prop_kernel_QT}
 Let $G\subset \Mod(\Sigma_{g,1})$ and $\pi\colon G \to \PGL\big( W\big(\widehat{\mathcal{O}}\big)\big)$ a representation associated to a finite orbit $\mathcal{O}$ in the reduced cell. Let $\phi \in G$ and consider a simple closed curve $\gamma$ and $\widehat{\rho} \in \widehat{\mathcal{O}}$ which admits a $\Delta$-lift, i.e., such that $\widehat{\rho}=\widehat{\mathcal{NA}}(\widehat{\mathbf{z}})$.
 Suppose that
\[ \Tr_A^{\Delta} ( \phi(\gamma) ) \not \equiv \lambda \Tr_A^{\Delta} ( \gamma ) \pmod{\mathcal{J}_{\widehat{\mathbf{z}}}}, \qquad \text{for all }\lambda \in \mathbb{C}^*.\]
 Then $\pi(\phi)\neq \id$.
 \end{Proposition}

\begin{proof}
The proposition is a direct consequence of Proposition~\ref{prop_kernel} and of equa\-tion~\eqref{eq_isom}.
\end{proof}

In order to be able to use the criterion of Proposition \ref{prop_kernel_QT} efficiently, we need to find a way to prove that, given two distinct simple closed curves $\alpha$ and $\beta$ in $\Sigma_{g,1}$ then
\[ \Tr_A^{\Delta}(\alpha) \not \equiv \lambda \Tr_A^{\Delta}(\beta) \pmod{\mathcal{J}_{\widehat{\mathbf{z}}}}, \qquad \text{for all }\lambda \in \mathbb{C}^*.\]

To achieve this goal we need to find a basis of $\mathcal{Z}_q(\mathbf{\Sigma}_g^*, \Delta)$ as a module over its center. The strategy is quite general: let $\mathbb{T}_q(\mathbb{E})$ be an arbitrary quantum torus with $A^{1/2}$ a root of unity of odd order $N$ and let $E^0 \subset E$ be the submodule defined as the kernel:
\[ E^0 := \ker \big( E\otimes_{\mathbb{Z}} E \xrightarrow{ (\cdot, \cdot) } \mathbb{Z} \to \mathbb{Z}/N\mathbb{Z} \big).\]
Then the center of $\mathbb{T}_q(\mathbb{E})$ is spanned by elements $Z_{e_0}$ for $e_0 \in E^0$ therefore, in order to find a~basis~$B$ of $\mathbb{T}_q(\mathbb{E})$ seen as module over its center, we need to find a representative set $E^1\subset E$ of the coset $\quotient{E}{E^0}$ and we can choose $B:= \Span \big( Z_e, e\in E^1\big)$.
\par It was proved in \cite{KojuQuesneyQNonAb} that the submodule $K_{\Delta}^0 \subset K_{\Delta}$ defined by
\[
 K_{\Delta}^0 :=\ker \big( K_{\Delta} \otimes_{\mathbb{Z}} K_{\Delta} \xrightarrow{ (\cdot, \cdot)^{\rm WP} } \mathbb{Z} \to \mathbb{Z}/N\mathbb{Z} \big)
 \]
is equal to
\begin{equation}\label{eq_K0}
 K_{\Delta}^0 = N K_{\Delta} + \mathbb{Z}\mathbf{k}_{\partial}.
 \end{equation}
Recall that $\mathbf{k}_{\partial}$ is the balanced map sending every edge to $+2$.

Let us now describe the image of simple closed curve by the quantum trace. Let $\gamma\subset \Sigma_{g,1}$ be a simple closed curve isotoped such that it is transversed to $\mathcal{E}(\Delta)$ with minimal number of intersection points.
 A \textit{full state} on $\gamma$ is a map $\hat{s}\colon \mathcal{E}(\Delta)\cap \gamma \rightarrow \{-,+\}$. A pair $(\gamma, \hat{s})$ induces on each face $\mathbb{T}\in F(\Delta)$ a stated diagram in $\mathbb{T}$. A full state $\hat{s}$ is \textit{admissible} if the restriction of $(\gamma,\hat{s})$ on each face does not contain bad arcs. We denote by ${\rm St}^a (\gamma)$ the set of admissible full states. For $\hat{s} \in {\rm St}^a (\gamma)$, we denote by $\mathds{k}(\hat{s})\in K_{\Delta}$ the balanced map defined by
\[ \mathds{k}(\hat{s}) (e) := \sum_{v\in \gamma \cap e} \hat{s}(v), \qquad e\in \mathcal{E}(\Delta).\]

\begin{Lemma}[{\cite[Lemma~2.17]{KojuAzumayaSkein}}]\label{lemma_qtr}
There exist some integers $n(\hat{s})\in \mathbb{Z}$ such that
\begin{equation*}
 \Tr_{A}^{\Delta} ([\gamma]) = \sum_{\hat{s} \in {\rm St}^a(\gamma)} A^{n(\hat{s})/2} Z^{\mathds{k}(\hat{s})}.
 \end{equation*}
\end{Lemma}

We denote by $\underline{{\rm St}}^a(\gamma)\subset K_{\Delta}$ the set of balanced maps of the form $\mathds{k}(\hat{s})$ for $\hat{s}\in {\rm St}^a(\gamma)$.
Let $\mathds{k}_{\gamma} \in K_{\Delta}$ be the balanced map defined by
\[ \mathds{k}_{\gamma}(e) := | \gamma \cap e |, \qquad \text{for all }e\in \mathcal{E}(\Delta).\]
Recall that $\gamma$ is in minimal position with respect to $\mathcal{E}(\Delta) $ so $|\gamma \cap e|$ is the geometric intersection of $\gamma$ and~$e$.
Note also that $\mathds{k}_{\gamma}= \mathds{k}(s^+)\in \underline{{\rm St}}^a(\gamma) $ corresponds to the admissible state $s^+$ sending every point of $\gamma \cap \mathcal{E}(\gamma)$ to~$+$.

\begin{Lemma}\label{lemma_St}
Let $\mathds{k}=\mathds{k}(\hat{s})\in \underline{{\rm St}}^a(\gamma)$.
\begin{enumerate}\itemsep=0pt
\item[$1.$] For all $e\in \mathcal{E}(\Delta)$ then $-|\gamma \cap e | \leq \mathds{k}(e) \leq |\gamma \cap e|$ and $\mathds{k}(e)$ has the same parity than $|\gamma \cap e|$.
\item[$2.$] One has $\mathds{k}(a_{\partial})=0$.
\end{enumerate}
\end{Lemma}

\begin{proof}This is an immediate consequence of the definition of $\underline{{\rm St}}^a(\gamma)$.
\end{proof}

\begin{Notations}
For $\mathbf{k}\in K_{\Delta}$ and $\gamma\subset \Sigma_{g,1}$ a simple closed curve, let
\[ S(\gamma):= \big\{ \mathbf{k}_0 \in K_{\Delta}^0 \text{ such that } \mathbf{k}+\mathbf{k}_0 \in \underline{\St}^a(\gamma) \big\} \subset K_{\Delta}^0.\]
For
 $\hat{\mathbf{z}}$ a character over the center of $\mathcal{Z}_q(\mathbf{\Sigma},\Delta)$, write
\[ x_{\hat{\mathbf{z}}, \mathbf{k}}(\gamma):= \sum_{ \mathbf{k}_0 \in S(\gamma), \, \mathbf{k}+\mathbf{k}_0=\mathds{k}(\hat{s})} \hat{\mathbf{z}}(\mathbf{k}_0) A^{n(\hat{s})/2}.\]
 \end{Notations}

 If $B\subset K_{\Delta}$ is a set of representatives of the coset $\quotient{K_{\Delta}}{K_{\Delta}^0}$ then by definition one has
\[ \Tr_A^{\Delta}([\gamma]) \equiv \sum_{\mathbf{k} \in B} x_{\hat{\mathbf{z}}, \mathbf{k}}(\gamma) Z^{\mathbf{k}} \pmod{\mathcal{J}_{\hat{\mathbf{z}}}},\]
 where the family $\big\{ Z^{\mathbf{k}},\, \mathbf{k}\in B\big\}$ is a basis of the space $\quotient{\mathcal{Z}_q(\mathbf{\Sigma}, \Delta)}{\mathcal{J}_{\hat{\mathbf{z}}}}$. We thus have proved

 \begin{Lemma}\label{lemma_criterion1}
 Let $\alpha, \beta\subset \Sigma_{g,1}$ be two simple closed curves and $\lambda\in \mathbb{C}^*$. Then
\[ \Tr_A^{\Delta}(\alpha) \equiv \lambda \Tr_A^{\Delta}(\beta) \pmod{\mathcal{J}_{\widehat{\mathbf{z}}}}\]
if and only if $x_{\hat{\mathbf{z}}, \mathbf{k}}(\alpha) = \lambda x_{\hat{\mathbf{z}}, \mathbf{k}}(\beta)$ for all $\mathbf{k}\in B$.

In particular, if there exists $\mathbf{k}\in K_{\Delta}$ such that $x_{\hat{\mathbf{z}}, \mathbf{k}}(\alpha)= 0$ and $x_{\hat{\mathbf{z}}, \mathbf{k}}(\beta)\neq 0$, then
\[ \Tr_A^{\Delta}(\alpha) \not \equiv \lambda \Tr_A^{\Delta}(\beta) \pmod{\mathcal{J}_{\widehat{\mathbf{z}}}}, \qquad \text{for all }\lambda \in \mathbb{C}^*.\]
\end{Lemma}

Lemma~\ref{lemma_criterion1} together with Proposition~\ref{prop_kernel_QT} permit to prove that a given mapping class $\phi\in G\subset \Mod(\Sigma_{g,1})$ does not belong to the kernel of a representation $\pi\colon G \to \End( W(\mathcal{O}))$ associated to a reduced orbit, provided that we are able to compute the numbers $x_{\hat{\mathbf{z}}, \mathbf{k}}(\gamma)$.
As an illustration, let us derive from this criterion an easy consequence.

\begin{Lemma}\label{lemma_criterion2} Let $\pi\colon K_{\Delta} \to \quotient{K_{\Delta}}{K_{\Delta}^0}$ be the quotient map.
Suppose that $\gamma \subset \Sigma_{g,1}$ intersects each edge of the triangulation $\Delta$ at most $N-1$ times. Then the restriction $\pi\colon \underline{\St}^a(\gamma) \to \quotient{K_{\Delta}}{K_{\Delta}^0}$ is injective.
\end{Lemma}

\begin{proof}Let $\mathbf{k}_1, \mathbf{k}_2 \in \underline{\St}^a(\gamma)$ be such that $\mathbf{k}_1-\mathbf{k}_2 \in K_{\Delta}^0$. By equation~\eqref{eq_K0}, there exists $\mathbf{k}_0 \in K_{\Delta}$ and $n\in \{0, \dots, N-1\}$ such that
\[ \mathbf{k}_1-\mathbf{k}_2 = n\mathbf{k}_{\partial} +N \mathbf{k}_0.\]
By the second item of Lemma~\ref{lemma_St}, one has $\mathbf{k}_1(a_{\partial})=\mathbf{k}_2(a_{\partial})=0$ so $n=0$. By the first item of Lemma~\ref{lemma_St}, for all $e\in \mathcal{E}(\Delta)$, one has $-(N-1)\leq \mathbf{k}_1(e), \mathbf{k}_2(e) \leq N-1$ and $\mathbf{k}_1(e)\equiv \mathbf{k}_2(e) \pmod{2}$ therefore $\mathbf{k}_0(e)=0$. We thus have proved that $\mathbf{k}_1=\mathbf{k}_2$. This concludes the proof.
\end{proof}

We eventually have the following criterion, which should be compared to Costantino--Martelli criterion in \cite[Lemma~7.1]{CostantinoMartelli} for the Turaev--Viro representations.

\begin{Theorem}\label{theorem_criterion}
Let $\pi\colon G \to \End(W(\mathcal{O}))$ be a representation of $G\subset \Mod(\Sigma_{g,1})$ associated to a~reduced orbit and a root of unity of odd order $N$. Consider a mapping class $\phi\in G$ and a~simple closed curve $\alpha \subset \Sigma_{g,1}$ such that $\alpha$ and $\beta:= \phi(\alpha)$ are not isotopic.
Suppose that there exists a~triangulation $\Delta$ of $\mathbf{\Sigma}_g^*$ such that
\begin{enumerate}\itemsep=0pt
\item[$1)$] there exists $\rho \in \mathcal{O}$ which admits a $\Delta$-lift,
\item[$2)$] both $\alpha$ and $\beta$ intersect each edge of $\Delta$ at most $N-1$ times.
\end{enumerate}
Then $\pi(\phi)\neq \id$.
\end{Theorem}

\begin{proof}Let $\widehat{\mathbf{z}}$ be the character associated to an arbitrary lift of $\rho$.
If $\alpha\neq \beta$ then there exists $e\in \mathcal{E}(\Delta) $ such that $|\alpha \cap e| \neq |\beta \cap e|$ so Lemma~\ref{lemma_St} implies that $\underline{\St}^a(\alpha) \neq \underline{\St}^a(\beta)$. Exchanging~$\alpha$ and~$\beta$ if necessary, we can suppose that there exists $\mathbf{k} \in \underline{\St}^a(\beta)$ such that $\mathbf{k} \notin \underline{\St}^a(\alpha)$. Lemma~\ref{lemma_criterion2} implies that $x_{\hat{\mathbf{z}}, \mathbf{k}}(\alpha)= 0$ and $x_{\hat{\mathbf{z}}, \mathbf{k}}(\beta)\neq 0$ so Lemma~\ref{lemma_criterion1} implies that
\[ \Tr_A^{\Delta}(\alpha) \not \equiv \lambda \Tr_A^{\Delta}(\beta) \pmod{\mathcal{J}_{\widehat{\mathbf{z}}}}, \qquad \text{for all }\lambda \in \mathbb{C}^*.\]
We conclude using Proposition~\ref{prop_kernel_QT}.
\end{proof}

\subsection{The use of quantum Teichm\"uller theory in the non reduced case}\label{sec_QT_nonreduced}

We now extend all the tools developed in the previous subsection to the study of representation $\pi\colon G \to \End(W(\mathcal{O}))$ associated to orbits in the big cell.
Bonahon and Wong's quantum trace permits to embed the reduced stated skein algebras into quantum tori. In order to be able to embed the non-reduced stated skein algebras into quantum tori as well, L\^e and Yu developed in \cite{LeYu_SSkeinQTraces} the following trick. Let $\mathbf{\Sigma}_g^{**}=(\Sigma_{g,1}, \{a_{\partial}', a_{\partial}''\})$ be $\Sigma_{g,1}$ with two boundary arcs. Consider the marked surface embedding $j\colon \mathbf{\Sigma}_g^* \hookrightarrow \mathbf{\Sigma}_g^{**}$ which is the identity outside a neighborhood of $a_{\partial}$ and which sends $a_{\partial}$ to $a_{\partial}'$ as illustrated in Figure~\ref{fig_triangulation_open}. Alternatively, one can think of $\mathbf{\Sigma}_g^{**}$ as been obtained from $\mathbf{\Sigma}_g^*$ by gluing a triangle $\mathbb{T}_0$ with edges $\{a_{\partial}, a_{\partial}', a_{\partial}''\}$ along the unique boundary arc of $\mathbf{\Sigma}_g^*$. In particular, every triangulation $\Delta$ of $\mathbf{\Sigma}_g^*$ induces a triangulation $\Delta^*$ of $\mathbf{\Sigma}_g^{**}$ obtained by adding the face $\mathbb{T}_0$. An illustration in the genus $1$ case is given in Figure~\ref{fig_triangulation_open}.

 \begin{figure}[!h]
\centerline{\includegraphics[width=7cm]{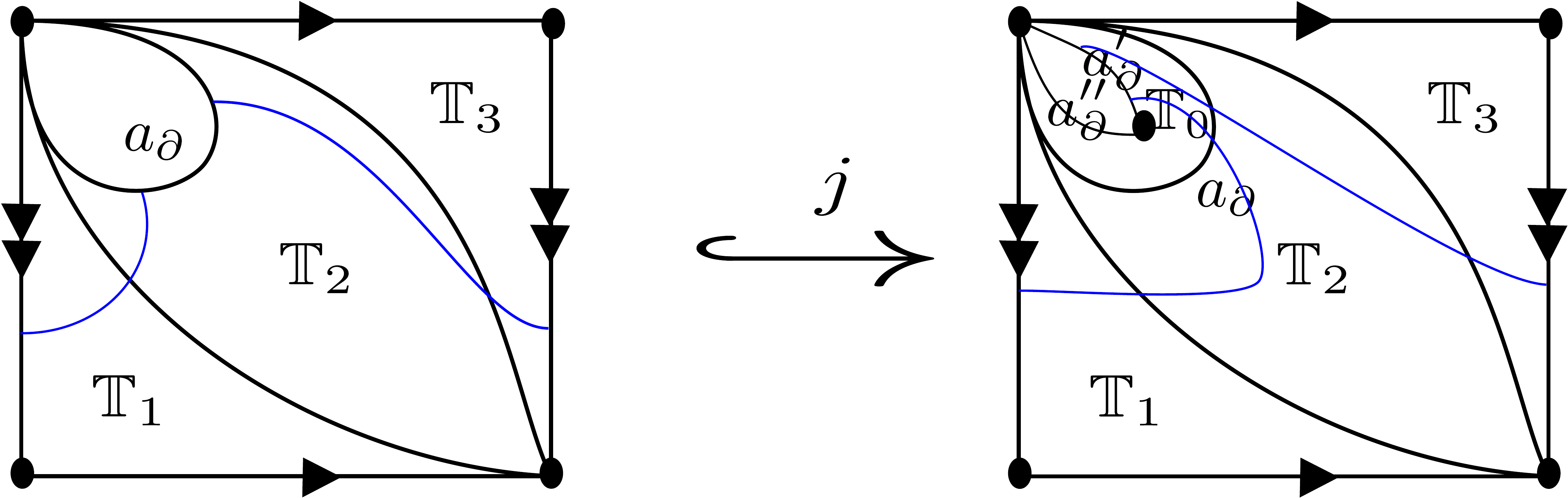} }
\caption{An illustration of the embedding $j$.}\label{fig_triangulation_open}
\end{figure}

L\^e and Yu proved in \cite{LeYu_SSkeinQTraces} that the embedding $j$ induces an embedding
\[ j_* \colon \ \mathcal{S}_A(\mathbf{\Sigma}_g^*) \hookrightarrow \overline{\mathcal{S}}_A(\mathbf{\Sigma}_g^{**})\]
of the non-reduced stated skein algebra of $\mathbf{\Sigma}_g^*$ onto the reduced stated skein algebra of $\mathbf{\Sigma}_g^{**}$. Therefore, by composition, we obtain an embedding
\[ \phi\colon \ \mathcal{S}_A(\mathbf{\Sigma}_g^*) \xrightarrow{j_*} \overline{\mathcal{S}}_A(\mathbf{\Sigma}_g^{**}) \xrightarrow{\Tr_q^{\Delta^*}} \mathcal{Z}_q(\mathbf{\Sigma}_g^{**}, \Delta^*).\]
The image of $\phi$ actually belongs to a smaller quantum torus that we now describe.

 Let $\overline{\mathcal{E}}:= \mathcal{E}(\Delta) \bigsqcup \{\widehat{a_{\partial}}\} $. Let $\overline{K}_{\Delta}$ denote the set of maps $\mathbf{k}\colon \overline{\mathcal{E}}(\Delta) \to \mathbb{Z}$ such that $(1)$ for any face of $\Delta$ with edges $a$, $b$, $c$, then $\mathbf{k}(a)+\mathbf{k}(b)+\mathbf{k}(c)$ is even and $(2)$ $\mathbf{k}(\widehat{a_{\partial}})$ is even. Define an injective linear map $i\colon \overline{K}_{\Delta} \hookrightarrow K_{\Delta^*}$ sending $\mathbf{k}$ to $\mathbf{k}'$ where $\mathbf{k}'(e)=\mathbf{k}(e)$ if $e\in \mathring{\mathcal{E}}(\Delta)$ and $\mathbf{k}'(a_{\partial})=\mathbf{k}(a_{\partial})+\mathbf{k}(\widehat{a_{\partial}})$, $\mathbf{k}'(a'_{\partial})=-\mathbf{k}(a_{\partial})$ and $\mathbf{k}'(a''_{\partial})=0$.
 Let $\langle \cdot, \cdot \rangle\colon \overline{\mathcal{E}} \otimes_{\mathbb{Z}} \overline{\mathcal{E}} \to \mathbb{Z}$ be the skew-symmetric form defined by $\langle x,y\rangle := ( i(x), i(y))^{\rm WP}$.
 Let $\overline{\mathcal{Z}}_{q}(\mathbf{\Sigma}_g^*, \Delta) \subset \mathcal{Z}_{q}(\mathbf{\Sigma}^{**}, \Delta^*)$ be the submodule spanned by elements $Z^{i(\mathbf{k})}$ for $\mathbf{k}\in \overline{K}_{\Delta}$. An easy computation shows that the map $\phi$ takes values in the smaller quantum torus $\overline{\mathcal{Z}}_{q}(\mathbf{\Sigma}, \Delta) \cong \mathbb{T}_q( \overline{K}_{\Delta}, \langle \cdot, \cdot \rangle )$.

 \begin{Definition}[{L\^e--Yu \cite{LeYu_SSkeinQTraces}}] The injective morphism $\phi\colon \mathcal{S}_A(\mathbf{\Sigma}_g^*) \hookrightarrow \overline{\mathcal{Z}}_q(\mathbf{\Sigma}_g^*, \Delta) $ will be referred to as the \textit{refined quantum trace}.
 \end{Definition}

 Let $\overline{K}_{\Delta}^0$ be the kernel of the skew-symmetric form
\[ \overline{K}_{\Delta} \otimes_{\mathbb{Z}} \overline{K}_{\Delta} \xrightarrow{ \langle \cdot, \cdot \rangle } \mathbb{Z} \to \mathbb{Z}/N\mathbb{Z}.\]
 Let $\mathbf{k}_{\widehat{a_{\partial}}} \in \overline{K}_{\Delta}$ be defined by $\mathbf{k}_{\widehat{a_{\partial}}} (\widehat{a_{\partial}})=2$ and $\mathbf{k}_{\widehat{a_{\partial}}} (e)=0$ for all $e\in \mathcal{E}(\Delta)$. So by definition, one has $\overline{K}_{\Delta}=K_{\Delta}\oplus \mathbb{Z} \mathbf{k}_{\widehat{a_{\partial}}} $.

 \begin{Lemma}\label{lemma_refined}
 One has
\[ \overline{K}_{\Delta}^0 = K_{\Delta}^0 \oplus N\mathbb{Z}\mathbf{k}_{\widehat{a_{\partial}}}.\]
 Therefore the PI-degree of $\overline{\mathcal{Z}}_q(\mathbf{\Sigma}_g^*, \Delta) $ is equal to $N^{3g}$, i.e., to the PI-degree of $\mathcal{S}_A(\mathbf{\Sigma}_g^*) $.
 \end{Lemma}

 \begin{proof}
 Let $\overline{\mathbf{k}}_1= \mathbf{k}_1+n_1 \mathbf{k}_{\widehat{a_{\partial}}}$ and $\overline{\mathbf{k}}_2= \mathbf{k}_2+n_2 \mathbf{k}_{\widehat{a_{\partial}}}$ be two maps in $\overline{K}_{\Delta}=K_{\Delta}\oplus \mathbb{Z} \mathbf{k}_{\widehat{a_{\partial}}} $. Then a~straightforward computation shows that
\[ \big\langle \overline{\mathbf{k}}_1, \overline{\mathbf{k}}_2\big\rangle = \big( i\big( \overline{\mathbf{k}}_1\big), i\big( \overline{\mathbf{k}}_2\big)\big)^{\rm WP} = (\mathbf{k}_1, \mathbf{k}_2)^{\rm WP} +n_1\mathbf{k}_1(a_{\partial}) -n_2\mathbf{k}_2(a_{\partial}).\]
 The equality $ \overline{K}_{\Delta}^0 = K_{\Delta}^0 \oplus N\mathbb{Z}\mathbf{k}_{\widehat{a_{\partial}}} $ follows. We deduce the index:
\[ \big|\overline{K}_{\Delta} : \overline{K}_{\Delta}^0\big| = N\times \big|K_{\Delta} : K_{\Delta}^0 \big| = N \times N^{3g-1}= N^{3g}.\]
 Therefore the PI-degree of $\overline{\mathcal{Z}}_q(\mathbf{\Sigma}_g^*, \Delta) $ is equal to $N^{3g}$.
 \end{proof}

Let $\overline{\mathcal{Y}}(\mathbf{\Sigma}_g^*, \Delta):= \operatorname{Specm} \big( \overline{\mathcal{Z}}_{+1}(\mathbf{\Sigma}_g^*, \Delta) \big)$. The refined quantum trace taken at $A^{1/2}=+1$ defines a dominant morphism
\[ \overline{\mathcal{NA}}\colon \ \overline{\mathcal{Y}}(\mathbf{\Sigma}_g^*, \Delta) \to \mathcal{X}(\mathbf{\Sigma}_g^*).\]

A point $\rho \in \mathcal{X}(\mathbf{\Sigma}_g^*)$ is said to \textit{admit a }$\Delta$-lift if it is in the image of~$\overline{\mathcal{NA}}$. The dominance of $\overline{\mathcal{NA}}$ and the irreducibility of $\mathcal{X}(\mathbf{\Sigma}_g^*)$ show that this is a generic condition. A point $\rho$ in the big cell $\mathcal{X}^0=\mu^{-1}\big(\SL_2^0\big)$ admits a $\Delta$-lift if and only if the representation $r_{\rho}\colon \mathcal{S}_A(\mathbf{\Sigma}_g^*) \to \End( V(\rho))$ extends through $\phi$ to a representation of $ \overline{\mathcal{Z}}_q(\mathbf{\Sigma}_g^*, \Delta) $.

For $\gamma \subset \Sigma_{g,1}$ a simple closed curve, since $\gamma$ can be isotoped outside of the triangle $\mathbb{T}_0$, the formula in Lemma \ref{lemma_qtr} still holds, that is one has
\begin{equation*}
 \phi ([\gamma]) = \sum_{\hat{s} \in {\rm St}^a(\gamma)} A^{n(\hat{s})/2} Z^{\mathds{k}(\hat{s})}
 \end{equation*}
 for the same integers $n(\hat{s})$ as in Lemma \ref{lemma_qtr}. Consider $K_{\Delta}$ as the subset of $\overline{K}_{\Delta}$ of maps sending~$\widehat{a}_{\partial}$ to~$0$ and recall the definition of the set $\overline{St}^a(\gamma) \subset K_{\Delta}\subset \overline{K}_{\Delta}$.
 For $\mathbf{z}\in \overline{\mathcal{Y}}(\mathbf{\Sigma}_g^*, \Delta)$ and $\mathbf{k}\in K_{\Delta}$, define the numbers $x_{\mathbf{z}, \mathbf{k}}(\gamma)\in \mathbb{C}$ in the same manner than in the previous subsection, i.e., by the formula
\[x_{{\mathbf{z}}, \mathbf{k}}(\gamma):= \sum_{ \mathbf{k}_0 \in S(\gamma),\, \mathbf{k}+\mathbf{k}_0=\mathds{k}(\hat{s})} {\mathbf{z}}(\mathbf{k}_0) A^{n(\hat{s})/2}.\]
 Let $\mathcal{J}_{\mathbf{z}}\subset \overline{\mathcal{Z}}_q(\mathbf{\Sigma}_g^*, \Delta)$ be the ideal generated by the kernel of $\mathbf{z}$.
 The following analogues of Lemma \ref{lemma_criterion1}, Proposition \ref{prop_kernel_QT} and Theorem \ref{theorem_criterion} hold in the non reduced case.

 \begin{Proposition}\quad
 \begin{enumerate}\itemsep=0pt
 \item[$1.$] For $\mathbf{z} \in \overline{\mathcal{Y}}(\mathbf{\Sigma}_g^*, \Delta)$ and $\lambda \in \mathbb{C}^*$, one has $\phi(\alpha)\equiv \lambda \phi(\beta) \pmod{\mathcal{J}_{\mathbf{z}}}$ if and only if $x_{\mathbf{z}, \mathbf{k}}(\alpha) = \lambda x_{\mathbf{z}, \mathbf{k}}(\beta)$ for all $\mathbf{k}\in K_{\Delta}$.
 \item[$2.$] Let $\pi\colon G\to \End(W(\mathcal{O}))$ a representation associated to an orbit in the big cell, $\phi \in G$ and~$\alpha$ a~simple closed curve such that $\alpha$ and $\beta:=\phi(\alpha)$ are not isotopic. Suppose that there exists $\rho \in \mathcal{O}$ which admits a $\Delta$-lift, i.e., such that $\rho=\overline{\mathcal{NA}}(\mathbf{z})$. Further suppose that
\[ \phi(\alpha) \not \equiv \lambda \phi(\beta) \pmod{\mathcal{J}_{\mathbf{z}}}, \qquad \text{for all }\lambda \in \mathbb{C}^*.\]
 Then $\pi(\phi)\neq \id$.
 \item[$3.$] Let $\pi\colon G\to \End(W(\mathcal{O}))$ a representation associated to an orbit in the big cell, $\phi \in G$ and~$\alpha$ a~simple closed curve such that $\alpha$ and $\beta:=\phi(\alpha)$ are not isotopic. If $(1)$ there exists $\rho \in \mathcal{O}$ which admits a $\Delta$-lift and~$(2)$ both $\alpha$ and $\beta$ intersect the edges of~$\Delta$ at most $N-1$ times, then $\pi(\phi)\neq \id$.
 \end{enumerate}
 \end{Proposition}

 \begin{proof} The proofs are straightforward adaptations of the proofs of Lemma~\ref{lemma_criterion1}, Proposition~\ref{prop_kernel_QT} and Theorem~\ref{theorem_criterion}, using Lemma~\ref{lemma_refined}, left to the reader.
 \end{proof}

 \subsection*{Acknowledgements}
 The author thanks S.~Baseilhac, D.~Callaque, F.~Costantino, A.~Quesney, T.Q.T.~L\^e and P.~Saf\-ronov for valuable conversations. He also thanks the anonymous referees for their very detailed reports which improved the quality and readability of this paper.
He acknowledges support from the Japanese Society for Promotion of Sciences, from the Centre National de la Recherche Scientifique and from the ERC DerSympApp (Grant~768679).

\pdfbookmark[1]{References}{ref}
\LastPageEnding


\begin{thebibliography}{99}
\footnotesize\itemsep=0pt

\bibitem{AlekseevGrosseSchomerus_LatticeCS1}
Alekseev A.Yu., Grosse H., Schomerus V., Combinatorial quantization of the
 {H}amiltonian {C}hern--{S}imons theory.~{I}, \href{https://doi.org/10.1007/BF02099431}{\textit{Comm. Math. Phys.}}
 \textbf{172} (1995), 317--358, \href{https://arxiv.org/abs/hep-th/9403066}{arXiv:hep-th/9403066}.

\bibitem{AlekseevGrosseSchomerus_LatticeCS2}
Alekseev A.Yu., Grosse H., Schomerus V., Combinatorial quantization of the
 {H}amiltonian {C}hern--{S}imons theory.~{II}, \href{https://doi.org/10.1007/BF02101528}{\textit{Comm. Math. Phys.}}
 \textbf{174} (1996), 561--604, \href{https://arxiv.org/abs/hep-th/9408097}{arXiv:hep-th/9408097}.

\bibitem{AlekseevKosmannMeinrenken}
Alekseev A.Yu., Kosmann-Schwarzbach Y., Meinrenken E., Quasi-{P}oisson
 manifolds, \href{https://doi.org/10.4153/CJM-2002-001-5}{\textit{Canad.~J. Math.}} \textbf{54} (2002), 3--29,
 \href{https://arxiv.org/abs/math.DG/0006168}{arXiv:math.DG/0006168}.

\bibitem{AlekseevMalkin_PoissonLie}
Alekseev A.Yu., Malkin A.Z., Symplectic structures associated to
 {L}ie--{P}oisson groups, \href{https://doi.org/10.1007/BF02105190}{\textit{Comm. Math. Phys.}} \textbf{162} (1994),
 147--173, \href{https://arxiv.org/abs/hep-th/9303038}{arXiv:hep-th/9303038}.

\bibitem{AlekseevMalkin_PoissonCharVar}
Alekseev A.Yu., Malkin A.Z., Symplectic structure of the moduli space of flat
 connection on a {R}iemann surface, \href{https://doi.org/10.1007/BF02101598}{\textit{Comm. Math. Phys.}} \textbf{169}
 (1995), 99--119, \href{https://arxiv.org/abs/hep-th/9312004}{arXiv:hep-th/9312004}.

\bibitem{AlekseevSchomerus_RepCS}
Alekseev A.Y., Schomerus V., Representation theory of {C}hern--{S}imons
 observables, \href{https://doi.org/10.1215/S0012-7094-96-08519-1}{\textit{Duke Math.~J.}} \textbf{85} (1996), 447--510,
 \href{https://arxiv.org/abs/q-alg/9503016}{arXiv:q-alg/9503016}.

\bibitem{BenzviBrochierJordan_FactAlg1}
Ben-Zvi D., Brochier A., Jordan D., Integrating quantum groups over surfaces,
 \href{https://doi.org/10.1112/topo.12072}{\textit{J.~Topol.}} \textbf{11} (2018), 874--917, \href{https://arxiv.org/abs/1501.04652}{arXiv:1501.04652}.

\bibitem{BenzviBrochierJordan_FactAlg2}
Ben-Zvi D., Brochier A., Jordan D., Quantum character varieties and braided
 module categories, \href{https://doi.org/10.1007/s00029-018-0426-y}{\textit{Selecta Math. (N.S.)}} \textbf{24} (2018),
 4711--4748, \href{https://arxiv.org/abs/1606.04769}{arXiv:1606.04769}.

\bibitem{BCGPTQFT}
Blanchet C., Costantino F., Geer N., Patureau-Mirand B., Non-semi-simple
 {TQFT}s, {R}eidemeister torsion and {K}ashaev's invariants, \href{https://doi.org/10.1016/j.aim.2016.06.003}{\textit{Adv.
 Math.}} \textbf{301} (2016), 1--78, \href{https://arxiv.org/abs/1404.7289}{arXiv:1404.7289}.

\bibitem{BonahonWongqTrace}
Bonahon F., Wong H., Quantum traces for representations of surface groups in
 {${\rm SL}_2(\mathbb C)$}, \href{https://doi.org/10.2140/gt.2011.15.1569}{\textit{Geom. Topol.}} \textbf{15} (2011),
 1569--1615, \href{https://arxiv.org/abs/1003.5250}{arXiv:1003.5250}.

\bibitem{BonahonWong1}
Bonahon F., Wong H., Representations of the {K}auffman bracket skein
 algebra~{I}: invariants and miraculous cancellations, \href{https://doi.org/10.1007/s00222-015-0611-y}{\textit{Invent. Math.}}
 \textbf{204} (2016), 195--243, \href{https://arxiv.org/abs/1206.1638}{arXiv:1206.1638}.

\bibitem{BonahonWong4}
Bonahon F., Wong H., The {W}itten--{R}eshetikhin--{T}uraev representation of
 the {K}auffman bracket skein algebra, \href{https://doi.org/10.1090/proc/12927}{\textit{Proc. Amer. Math. Soc.}}
 \textbf{144} (2016), 2711--2724, \href{https://arxiv.org/abs/1309.0921}{arXiv:1309.0921}.

\bibitem{BonahonWong2}
Bonahon F., Wong H., Representations of the {K}auffman bracket skein
 algebra~{II}: {P}unctured surfaces, \href{https://doi.org/10.2140/agt.2017.17.3399}{\textit{Algebr. Geom. Topol.}} \textbf{17}
 (2017), 3399--3434, \href{https://arxiv.org/abs/1206.1639}{arXiv:1206.1639}.

\bibitem{BrownGoodearl}
Brown K.A., Goodearl K.R., Lectures on algebraic quantum groups, \textit{Advanced
 Courses in Mathematics. CRM Barcelona}, \href{https://doi.org/10.1007/978-3-0348-8205-7}{Birkh\"auser Verlag}, Basel, 2002.

\bibitem{BrownGordon_ramificationcenters}
Brown K.A., Gordon I., The ramification of centres: {L}ie algebras in positive
 characteristic and quantised enveloping algebras, \href{https://doi.org/10.1007/s002090100274}{\textit{Math.~Z.}}
 \textbf{238} (2001), 733--779, \href{https://arxiv.org/abs/math.RT/9911234}{arXiv:math.RT/9911234}.

\bibitem{BrownGordon_PO}
Brown K.A., Gordon I., Poisson orders, symplectic reflection algebras and
 representation theory, \href{https://doi.org/10.1515/crll.2003.048}{\textit{J.~Reine Angew. Math.}} \textbf{559} (2003),
 193--216, \href{https://arxiv.org/abs/math.RT/0201042}{arXiv:math.RT/0201042}.

\bibitem{Brown_AL_discriminant}
Brown K.A., Yakimov M.T., Azumaya loci and discriminant ideals of {PI}
 algebras, \href{https://doi.org/10.1016/j.aim.2018.10.024}{\textit{Adv. Math.}} \textbf{340} (2018), 1219--1255,
 \href{https://arxiv.org/abs/1702.04305}{arXiv:1702.04305}.

\bibitem{BuffenoirRoche}
Buffenoir E., Roche P., Two-dimensional lattice gauge theory based on a quantum
 group, \href{https://doi.org/10.1007/BF02099153}{\textit{Comm. Math. Phys.}} \textbf{170} (1995), 669--698,
 \href{https://arxiv.org/abs/hep-th/9405126}{arXiv:hep-th/9405126}.

\bibitem{BuffenoirRoche2}
Buffenoir E., Roche P., Link invariants and combinatorial quantization of
 {H}amiltonian {C}hern--{S}imons theory, \href{https://doi.org/10.1007/BF02101008}{\textit{Comm. Math. Phys.}}
 \textbf{181} (1996), 331--365, \href{https://arxiv.org/abs/q-alg/9507001}{arXiv:q-alg/9507001}.

\bibitem{BullockGeneratorsSkein}
Bullock D., A finite set of generators for the {K}auffman bracket skein
 algebra, \href{https://doi.org/10.1007/PL00004727}{\textit{Math.~Z.}} \textbf{231} (1999), 91--101.

\bibitem{BullockFrohmanKania_LGFT}
Bullock D., Frohman C., Kania-Bartoszy\'{n}ska J., Topological interpretations
 of lattice gauge field theory, \href{https://doi.org/10.1007/s002200050471}{\textit{Comm. Math. Phys.}} \textbf{198}
 (1998), 47--81, \href{https://arxiv.org/abs/q-alg/9710003}{arXiv:q-alg/9710003}.

\bibitem{Cooke_FactorisationHomSkein}
Cooke J., Excision of skein categories and factorisation homology,
 \href{https://arxiv.org/abs/1910.02630}{arXiv:1910.02630}.

\bibitem{CostantinoLe19}
Costantino F., L\^e T.T.Q., Stated skein algebras of surfaces, \href{https://doi.org/10.4171/JEMS/1167}{\textit{J.~Eur.
 Math. Soc.}}, {t}o appear, \href{https://arxiv.org/abs/1907.11400}{arXiv:1907.11400}.

\bibitem{CostantinoMartelli}
Costantino F., Martelli B., An analytic family of representations for the
 mapping class group of punctured surfaces, \href{https://doi.org/10.2140/gt.2014.18.1485}{\textit{Geom. Topol.}} \textbf{18}
 (2014), 1485--1538, \href{https://arxiv.org/abs/1210.2666}{arXiv:1210.2666}.

\bibitem{DeConciniKacRepQGroups}
De~Concini C., Kac V.G., Representations of quantum groups at roots of~{$1$},
 in Operator Algebras, Unitary Representations, Enveloping Algebras, and
 Invariant Theory ({P}aris, 1989), \textit{Progr. Math.}, Vol.~92,
 Birkh\"auser Boston, Boston, MA, 1990, 471--506.

\bibitem{DeConciniLyubashenko_OqG}
De~Concini C., Lyubashenko V., Quantum function algebra at roots of~{$1$},
 \href{https://doi.org/10.1006/aima.1994.1071}{\textit{Adv. Math.}} \textbf{108} (1994), 205--262.

\bibitem{DeConciniProcesiBook}
De~Concini C., Procesi C., Quantum groups, in {$D$}-Modules, Representation
 Theory, and Quantum Groups ({V}enice, 1992), \textit{Lecture Notes in Math.},
 Vol.~1565, \href{https://doi.org/10.1007/BFb0073466}{Springer}, Berlin, 1993, 31--140.

\bibitem{DeRenziGeerPatureaRunkel_MCG}
De~Renzi M., Gainutdinov A.M., Geer N., Patureau-Mirand B., Runkel I.,
 \href{https://doi.org/10.1142/S0219199721500917}{\textit{Commun. Contemp. Math.}}, {t}o appear, \href{https://arxiv.org/abs/2010.14852}{arXiv:2010.14852}.

\bibitem{DrinfeldrMatrix}
Drinfel'd V.G., On constant quasiclassical solutions of the Yang--Baxter
 equations, \textit{Soviet Math. Dokl.} \textbf{28} (1983), 667--671.

\bibitem{Faitg_LGFT_SSkein}
Faitg M., Holonomy and (stated) skein algebras in combinatorial quantization,
 \href{https://arxiv.org/abs/2003.08992}{arXiv:2003.08992}.

\bibitem{FockRosly}
Fock V.V., Rosly A.A., Poisson structure on moduli of flat connections on
 {R}iemann surfaces and the {$r$}-matrix, in Moscow {S}eminar in
 {M}athematical {P}hysics, \textit{Amer. Math. Soc. Transl. Ser.~2}, Vol.~191,
 \href{https://doi.org/10.1090/trans2/191/03}{Amer. Math. Soc.}, Providence, RI, 1999, 67--86, \href{https://arxiv.org/abs/math.QA/9802054}{arXiv:math.QA/9802054}.

\bibitem{FrohmanKaniaLe_UnicityRep}
Frohman C., Kania-Bartoszynska J., L\^e T., Unicity for representations of the
 {K}auffman bracket skein algebra, \href{https://doi.org/10.1007/s00222-018-0833-x}{\textit{Invent. Math.}} \textbf{215} (2019),
 609--650, \href{https://arxiv.org/abs/1707.09234}{arXiv:1707.09234}.

\bibitem{FrohmanKaniaLe_DimSkein}
Frohman C., Kania-Bartoszynska J., L\^e T., Dimension and trace of the
 {K}auffman bracket skein algebra, \href{https://doi.org/10.1090/btran/69}{\textit{Trans. Amer. Math. Soc. Ser.~B}}
 \textbf{8} (2021), 510--547, \href{https://arxiv.org/abs/1902.02002}{arXiv:1902.02002}.

\bibitem{GanevJordanSafranov_FrobeniusMorphism}
Ganev I., Jordan D., Safronov P., The quantum {F}robenius for character
 varieties and multiplicative quiver varieties, \href{https://arxiv.org/abs/1901.11450}{arXiv:1901.11450}.

\bibitem{GunninghamJordanSafranov_FinitenessConjecture}
Gunningham S., Jordan D., Safronov P., The finiteness conjecture for skein
 modules, \href{https://arxiv.org/abs/1908.05233}{arXiv:1908.05233}.

\bibitem{GHJW_ModSpacesParBd}
Guruprasad K., Huebschmann J., Jeffrey L., Weinstein A., Group systems,
 groupoids, and moduli spaces of parabolic bundles, \href{https://doi.org/10.1215/S0012-7094-97-08917-1}{\textit{Duke Math.~J.}}
 \textbf{89} (1997), 377--412, \href{https://arxiv.org/abs/dg-ga/9510006}{arXiv:dg-ga/9510006}.

\bibitem{Habiro_QCharVar}
Habiro K., A note on quantum fundamental groups and quantum representation
 varieties for $3$-manifolds, \textit{RIMS K\=oky\=uroku} \textbf{1777}
 (2012), 21--30.

\bibitem{Haioun_Sskein_FactAlg}
Ha\"{\i}oun B., Relating stated skein algebras and internal skein algebras,
 \href{https://doi.org/10.3842/SIGMA.2022.042}{\textit{SIGMA}} \textbf{18} (2022), 042, 39~pages, \href{https://arxiv.org/abs/2104.13848}{arXiv:2104.13848}.

\bibitem{Hart}
Hartshorne R., Algebraic geometry, \textit{Graduate Texts in Mathematics},
 Vol.~52, \href{https://doi.org/10.1007/978-1-4757-3849-0}{Springer-Verlag}, New York~-- Heidelberg, 1977.

\bibitem{Higgins_SSkeinSL3}
Higgins V., Triangular decomposition of ${\rm SL}_3$ skein algebras,
 \href{https://arxiv.org/abs/2008.09419}{arXiv:2008.09419}.

\bibitem{HodgesLevasseur_OqG}
Hodges T.J., Levasseur T., Primitive ideals of {${\bf C}_q[{\rm SL}(3)]$},
 \href{https://doi.org/10.1007/BF02096864}{\textit{Comm. Math. Phys.}} \textbf{156} (1993), 581--605.

\bibitem{Koju2}
Korinman J., Decomposition of some {W}itten--{R}eshetikhin--{T}uraev
 representations into irreducible factors, \href{https://doi.org/10.3842/SIGMA.2019.011}{\textit{SIGMA}} \textbf{15} (2019),
 011, 25~pages, \href{https://arxiv.org/abs/1406.4389}{arXiv:1406.4389}.

\bibitem{KojuTriangularCharVar}
Korinman J., Triangular decomposition of character varieties,
 \href{https://arxiv.org/abs/1904.09022}{arXiv:1904.09022}.


\bibitem{KojuPresentationSSkein}
Korinman J., Finite presentations for stated skein algebras and lattice gauge
 field theory, \textit{Algebr. Geom. Topol.}, {t}o appear, \href{https://arxiv.org/abs/2012.03237}{arXiv:2012.03237}.

\bibitem{KojuSurvey}
Korinman J., Stated skein algebras and their representations,
 \href{https://arxiv.org/abs/2105.09563}{arXiv:2105.09563}.

\bibitem{KojuRIMS}
Korinman J., Stated skein algebras and their representations, \textit{RIMS
 K\=oky\=uroku} \textbf{2191} (2021), 52--71.

\bibitem{KojuAzumayaSkein}
Korinman J., Unicity for representations of reduced stated skein algebras,
 \href{https://doi.org/10.1016/j.topol.2020.107570}{\textit{Topology Appl.}} \textbf{293} (2021), 107570, 28~pages,
 \href{https://arxiv.org/abs/2001.00969}{arXiv:2001.00969}.

\bibitem{KojuMurakami_QCharVar}
Korinman J., Murakami J., Relating quantum character varieties and skein
 modules, {i}n preparation.

\bibitem{KojuQuesneyClassicalShadows}
Korinman J., Quesney A., Classical shadows of stated skein representations at
 odd roots of unity, \href{https://arxiv.org/abs/1905.03441}{arXiv:1905.03441}.

\bibitem{KojuQuesneyQNonAb}
Korinman J., Quesney A., The quantum trace as a quantum non-abelianization map,
 \href{https://doi.org/10.1142/S0218216522500328}{\textit{J.~Knot Theory Ramifications}} \textbf{31} (2022), 2250032, 49~pages,
 \href{https://arxiv.org/abs/1907.01177}{arXiv:1907.01177}.

\bibitem{LeStatedSkein}
L\^e T.T.Q., Triangular decomposition of skein algebras, \href{https://doi.org/10.4171/QT/115}{\textit{Quantum
 Topol.}} \textbf{9} (2018), 591--632, \href{https://arxiv.org/abs/1609.04987}{arXiv:1609.04987}.

\bibitem{LeYu_Survey}
L\^e T.T.Q., Yu T., Stated skein modules of marked 3-manifolds/surfaces, a
 survey, \href{https://doi.org/10.1007/s40306-021-00417-2}{\textit{Acta Math. Vietnam.}} \textbf{46} (2021), 265--287,
 \href{https://arxiv.org/abs/2005.14577}{arXiv:2005.14577}.

\bibitem{LeYu_SSkeinQTraces}
L\^e T.T.Q., Yu T., Quantum traces and embeddings of stated skein algebras into
 quantum tori, \href{https://doi.org/10.1007/s00029-022-00781-3}{\textit{Selecta Math. (N.S.)}} \textbf{28} (2022), 66, 48~pages,
 \href{https://arxiv.org/abs/2012.15272}{arXiv:2012.15272}.

\bibitem{Majid_QGroups}
Majid S., Foundations of quantum group theory, \href{https://doi.org/10.1017/CBO9780511613104}{Cambridge University Press},
 Cambridge, 1995.

\bibitem{PrzytyckiSikora_SkeinDomain}
Przytycki J.H., Sikora A.S., Skein algebras of surfaces, \href{https://doi.org/10.1090/tran/7298}{\textit{Trans. Amer.
 Math. Soc.}} \textbf{371} (2019), 1309--1332, \href{https://arxiv.org/abs/1602.07402}{arXiv:1602.07402}.

\bibitem{Reiner03}
Reiner I., Maximal orders, \textit{London Mathematical Society Monographs New
 Series}, Vol.~28, The Clarendon Press, Oxford University Press, Oxford, 2003.

\bibitem{RT}
Reshetikhin N., Turaev V.G., Invariants of {$3$}-manifolds via link polynomials
 and quantum groups, \href{https://doi.org/10.1007/BF01239527}{\textit{Invent. Math.}} \textbf{103} (1991), 547--597.

\bibitem{Tu}
Turaev V.G., Quantum invariants of knots and 3-manifolds, \textit{De Gruyter
 Studies in Mathematics}, Vol.~18, \href{https://doi.org/10.1515/9783110221848}{Walter de Gruyter \& Co.}, Berlin, 2010.

\end{thebibliography}
\end{document}